\documentclass[a4paper]{amsart}
\usepackage{amsthm}
\usepackage{xy}
\xyoption{all}
\usepackage{float}
\usepackage{xcolor}
\usepackage{graphics}
\usepackage{amssymb}

\usepackage[utf8]{inputenc}
\usepackage[T1]{fontenc}
\usepackage{dsfont}
\usepackage[normalem]{ulem}

\newcommand{\id}{\mathds{1}}

 \newcommand{\DD}{\mathbb{D}}
 \newcommand{\EE}{\mathbb{E}}
 \newcommand{\YY}{\mathbb{Y}}
 \newcommand{\XX}{\mathbb{X}}
 \newcommand{\coh}[1]{\mathrm{coh}(#1)}

 \newcommand{\tauX}{\tau_\XX}
 \newcommand{\vc}{\vec{c}}
 \newcommand{\vect}[1]{\mathrm{vect}(#1)}
 \newcommand{\svect}[1]{\mathrm{\underline{vect}}(#1)}
 \newcommand{\MFgr}[3]{\mathrm{MF}^{#1}(#2,#3)}
 \newcommand{\sMFgr}[3]{\mathrm{\underline{MF}}^{#1}(#2,#3)}
 \newcommand{\vom}{{\vec\omega}}
 \newcommand{\vdom}{\vec{\delta}}
 \newcommand{\vx}{{\vec{x}}}
 \newcommand{\vy}{{\vec{y}}}
 \newcommand{\vu}{{\vec{u}}}
 \newcommand{\vz}{{\vec{z}}}
  \newcommand{\vw}{{\vec{\omega}}}
  \newcommand{\vnull}{\vec{0}}

 \renewcommand{\deg}{\mathrm{deg}\,}
 
 \newcommand{\rk}{\mathrm{rk}}

 \newcommand{\LL}{\mathbb{L}}
 \newcommand{\Ll}{\mathcal{L}}
 \newcommand{\Hom}{\mathrm{Hom}}
 \newcommand{\Ext}{\mathrm{Ext}}
 \newcommand{\HomX}{\mathrm{Hom_{\XX}}}
 \newcommand{\projgr}[2]{\mathrm{proj}^{#1}(#2)}

 \newcommand{\cok}[1]{\mathrm{cok}(#1)}
 \newcommand{\up}[1]{\stackrel{#1}\longrightarrow}
 \renewcommand{\phi}{\varphi}
 \newcommand{\bP}{\bar{P}}
 \newcommand{\bphi}{\bar{\phi}}
 \newcommand{\bpsi}{\bar{\psi}}
 \newcommand{\wtilde}{\widetilde}
 \newcommand{\La}{\Lambda}
 \newcommand{\rla}{\rightleftarrows}
   \newcommand{\mf}[2]{{\underset{#2}{\overset{#1}{\rla}}}}

 \newcommand{\extb}[2]{E_{#1}\langle#2\rangle}
 \newcommand{\extbo}[1]{E\langle#1\rangle}

 \newcommand{\la}{\lambda}
 \newcommand{\ra}{\rightarrow}
 \newcommand{\lra}{\longrightarrow}
 \newcommand{\ZZ}{\mathbb{Z}}
 \newcommand{\QQ}{\mathbb{Q}}

 \newcommand{\End}{\mathrm{End}}
 \newcommand{\Oo}{\mathcal{O}}
 \renewcommand{\mod}[1]{\mathrm{mod}(#1)}
 
 \newcommand{\modgr}[2]{\mathrm{mod}^{#1}(#2)}
 \newcommand{\modgrnull}[2]{\mathrm{mod}_0^{#1}(#2)}
 \newcommand{\CMgr}[2]{{\mathrm{CM}}^{#1}(#2)}
 \newcommand{\sCMgr}[2]{{\mathrm{\underline{CM}}}^{#1}(#2)}
 \newcommand{\Singgr}[2]{\mathrm{Sing}^{#1}(#2)}

 \newcommand{\Knull}{\mathrm{K}_0}

 \newcommand{\Uu}{{\mathcal U}}
 \newcommand{\ExtX}{\mathrm{Ext}^1_{\XX}}

 \newcommand{\ih}[1]{\mathfrak{I}(#1)}
 \newcommand{\pc}[1]{\mathfrak{P}(#1)}

 \def\c{\circ}
\def\b{\bullet}
\newcommand{\bc}[1]{\underset{#1}{\c}}
\newcommand{\tc}[1]{\overset{#1}{\c}}
\newcommand{\tb}[1]{\overset{#1}{\b}}
\newcommand{\bb}[1]{\underset{#1}{\bullet}}
\newcommand{\bUu}{\overline{\Uu}}

\newcommand{\bcv}[1]{\underset{(\vec{#1})}{\c}}
\newcommand{\tcv}[1]{\overset{(\vec{#1})}{\c}}
\definecolor{ashy}{gray}{0.6}

 \theoremstyle{plain}
 \newtheorem{Thm}{Theorem}[section]
 \newtheorem{Prop}[Thm]{Proposition}
 \newtheorem{Lem}[Thm]{Lemma}
 \newtheorem{Cor}[Thm]{Corollary}

 \theoremstyle{definition}
 \newtheorem{Defi}[Thm]{Definition}
 \newtheorem{Ex}[Thm]{Example}
 
 \newtheorem{Rem}[Thm]{Remark}
 \newtheorem{Ob}[Thm]{Observation}

 \numberwithin{equation}{section}

\newenvironment{pf}{\begin{proof}}{\end{proof}}

\advance\textheight by 0.5cm
\advance\textwidth by 1.0cm
\advance\evensidemargin by -0.8cm
\advance\oddsidemargin by -0.8cm

\begin{document}

\title[Domestic triangle singularities]{Matrix factorizations for domestic triangle singularities}

\author[D. E. K\k{e}dzierski, H. Lenzing and H. Meltzer]{Dawid Edmund K\k{e}dzierski, Helmut Lenzing and Hagen Meltzer}

\address{Instytut Matematyki, Uniwersytet Szczeci\'nski, 70451
  Szczecin, Poland }

\email{d.e.kedzierski@wp.pl}

\address{Universit\"{a}t Paderborn; Institut f\"{u}r Mathematik; 33095 Paderborn, Germany}

\email{helmut@math.uni-paderborn.de}

\address{Instytut Matematyki, Uniwersytet Szczeci\'nski, 70451 Szczecin, Poland}

\email{meltzer@wmf.univ.szczecin.pl}



\thanks{}

\maketitle

\begin{abstract}
Working over an algebraically closed field $k$ of any characteristic, we determine the matrix factorizations for the --- suitably graded --- triangle singularities $f=x^a+y^b+z^c$ of domestic type, that is, we assume that  $(a,b,c)$ are integers at least two, satisfying $1/a+1/b+1/c>1$. Using work by Kussin-Lenzing-Meltzer, this is achieved by determining projective covers in the Frobenius category of vector bundles on the weighted projective line of weight type $(a,b,c)$. Equivalently, in a representation-theoretic context, we can work in the mesh category of $\ZZ\tilde\Delta$ over $k$, where $\tilde\Delta$ is the extended Dynkin diagram, corresponding to the Dynkin diagram $\Delta=[a,b,c]$.  Our work is related to, but in methods and results different from, the determination of matrix factorizations for the $\ZZ$-graded simple singularities by Kajiura-Saito-Takahashi. In particular,  we obtain symmetric matrix factorizations whose entries are scalar multiples of monomials, with scalars taken from $\{0,\pm1\}$.
\end{abstract}

\section{Introduction} \label{sect:introduction}
Assuming that $(a,b,c)$ is a triple of integers greater or equal $2$,
we investigate the $\LL$-graded  hypersurface $S=k[x_1,x_2,x_3]/(f)$ determined by the triangle singularity  $f=x_1^a+x_2^b+x_3^c$. Here, $\LL=\LL(a,b,c)$ is the rank one abelian group on generators $\vx_1$, $\vx_2$, $\vx_3$ with relations $a\vx_1=b\vx_2=c\vx_3=:\vc$, and the generator $x_i$ from $S$  is given degree $\vx_i$ ($i=1,2,3$). Note that the polynomial $f$ is homogeneous of degree $\vc$, the \emph{canonical element} of $\LL$.
Let $\XX=\XX(a,b,c)$ be the associated weighted projective line, whose category of coherent sheaves $\coh\XX$ is obtained from $S$ by Serre's construction as the quotient category $\modgr\LL{S}/\modgrnull\LL{S}$, see~\cite[Section 1.8]{Geigle:Lenzing:1987}. Sheafification, given by the natural quotient functor $q:\modgr\LL{S}\to \coh\XX$, then induces an equivalence between the full subcategory $\CMgr\LL{S}$ of $\LL$-graded (maximal) Cohen-Macaulay modules over $S$ and the category $\vect\XX$ of vector bundles on $\XX$~\cite[Theorem 5.1]{Geigle:Lenzing:1987}. Since $S$ is graded Gorenstein, the category $\CMgr\LL{S}$ is a Frobenius category with respect to the exact structure inherited from the abelian category $\modgr\LL{S}$ of finitely generated $\LL$-graded $S$-modules. With respect to this structure, the graded maximal Cohen-Macaulay modules of rank one form the indecomposable projective-injectives of $\CMgr\LL{S}$. The corresponding stable category $\sCMgr\LL{S}$ is triangulated. It is equivalent to the \emph{singularity category} $\Singgr\LL{S}$ introduced and studied by Buchweitz~\cite{Buchweitz:1986} in the ungraded and by Orlov~\cite{Orlov:2009} in the graded case.

Important for the present paper is an alternative description of a singularity category as the \emph{stable category of vector bundles} $\svect\XX$ on the weighted projective line $\XX$, see~\cite{Kussin:Lenzing:Meltzer:2013adv}. To define this category,  we call a sequence $\eta:0\to E'\to E\to E''\to 0$ of vector bundles \emph{distinguished-exact} if $\Hom(L,\eta)$ is exact for each line bundle $L$ on $\XX$. With the exact structure defined by these sequences, the category $\vect\XX$ of vector bundles on $\XX$ is a Frobenius category, equivalent to $\CMgr\LL{S}$, such that the indecomposable projective-injectives are just the line bundles on $\XX$. A fortiori, the stable category $\sCMgr\LL{S}$ of Cohen-Macaulay modules, is equivalent to the factor category $\svect\XX=\vect\XX/[\Ll]$, where $[\Ll]$ is the ideal consisting of all morphisms factoring through a finite direct sum of line bundles.

By results of Buchweitz \cite{Buchweitz:1986} and Orlov \cite{Orlov:2009}, it is known that the singularity category $\mathrm{Sing}^\LL(S)$, in the  $\LL$-graded sense,  and the category of $\LL$-graded  maximal  Cohen-Macaulay modules $\underline{\mathrm{CM}}^\LL(S)$ are equivalent.
Thus the stable category of vector bundles $\svect\XX$ is another incarnation of the singularity category.
In addition, all these categories are triangle equivalent to $\sMFgr\LL{T}{f}$, the \emph{stable category of $\LL$-graded matrix factorizations} of $f$ over the polynomial algebra $T=k[x_1,x_2,x_3]$.

For a base field of characteristic zero, a related category of graded matrix factorizations of a $\ZZ$-graded simple singularity was investigated by H. Kajiura, K. Saito and  A. Takahashi~\cite{KST-1}. While these authors work directly inside the category of matrix factorizations, we work inside the category of vector bundles on the associated weighted projective line, and exploit well-known results on the Auslander-Reiten theory of $\vect\XX$. By contrast, our paper takes as a starting point the study of triangle singularities, and the associated stable category of vector bundles \cite{Kussin:Lenzing:Meltzer:2013adv}. Accordingly, we work over an  algebraically closed field $k$  of arbitrary characteristic. We recall that $\chi_\XX=1-(1/a+1/b+1/c)$ is the \emph{Euler characteristic} of $\XX$ such that domestic type for $\XX$ relates to positive Euler characteristic.

For a weighted projective line $\XX$ of domestic weight type $(a,b,c)$, the main achievement of our paper is two-fold: (A) a complete description of the projective covers (resp.\ the injective hulls) of indecomposable vector bundles, and (B) a complete  description  of all  $\LL$-graded matrix factorizations for singularities $f=x_1^a+x_2^b+x_3^c$ for indecomposable $\LL$-graded (maximal) Cohen-Macaulay modules.

To achieve (A), we start with a fundamental result from \cite{Kussin:Lenzing:Meltzer:2013adv} on the projective covers, and likewise the injective hulls, of indecomposable vector bundles of rank two. For this first step there is no restriction on the Euler characteristic. Then in the second step, assuming domestic type,  we use the knowledge of the Auslander-Reiten quiver for the category $\vect\XX$, and use properly chosen distinguished-exact sequences to ``extend'' the projective covers to indecomposable bundles of higher rank. To achieve (B), we then lift minimal projective resolutions in $\CMgr\LL{S}=\vect\XX$ to matrix factorizations. As a key ingredient of the proof, we use that the indecomposable vector bundles involved are uniquely determined by their projective covers, see Proposition \ref{prop:determined_by_cover}.

We remark that step (A) has a direct interpretation in the representation theory of path algebras of extended Dynkin quivers: Assuming domestic type, it follows from a combination of \cite{Geigle:Lenzing:1987} and \cite{Happel:1988} that the category of indecomposable vector bundles on $\XX$ is equivalent to the mesh category $k(\ZZ\Delta)$ for the extended Dynkin star $\Delta=[a,b,c]$. Our results on projective covers and matrix factorizations thus offer new insight in the nature of the representation theory for path algebras of extended Dynkin type.

\section{Basic concepts} \label{sect:basic.concepts}
We briefly recall the concept of a weighted projective line, where we restrict to the case of \emph{triple weight type}, given by weight triples $(a,b,c)$ of integers greater or equal $2$. For a more general setting and further details we refer to \cite{Geigle:Lenzing:1987}. Throughout this paper, $k$ denotes an algebraically closed field.

Let  $\LL=\LL(a,b,c)$ be the rank one abelian group on
generators $\vx_1$, $\vx_2$, $\vx_3$ with relations
$a\vx_1=b\vx_2=c\vx_3 =: \vc$, where $\vc$ is called the \emph{canonical element}. We note that $\LL$ is naturally isomorphic to the \emph{Picard group} of $\XX$.
The polynomial algebra $T=k[x_1,x_2,x_3]$, $T=T(a,b,c)$ is equipped with an $\LL$-grading by giving $x_i$  degree $\vx_i$ for each $i=1,2,3$.
Further, let $S=S(a,b,c)$ denote the factor algebra $k[x_1,x_2,x_3]/(f)$, where $f=x_1^{a}+x_2^{b}+ x_3^{c}$.
Because $f$ is a homogeneous polynomial, the algebra $S$ is also $\LL$-graded; by $S_\vx$ we denote the finite dimensional vector space of elements of degree $\vx$. The weighted  projective line $\XX=\XX(a,b,c)$ is by definition the $\LL$-graded projective spectrum of the $\LL$-graded algebra $S$. By \cite{Geigle:Lenzing:1987} its category of coherent sheaves  $\coh \XX$ is obtained by Serre's construction as the quotient category of $\modgr\LL{S}$, the category of finitely generated $\LL$-graded $S$-modules, by the Serre subcategory $\modgrnull\LL{S}$ of all finite dimensional (=finite length) modules. By $q:\modgr\LL{S}\to \coh\XX$, $M\mapsto\wtilde{M}$, we note the corresponding quotient functor (sheafification). For the present paper, the following result is of importance. For its proof, we refer to \cite[Theorem 5.1]{Geigle:Lenzing:1987}, and for the last claim to \cite{Kussin:Lenzing:Meltzer:2013adv}. From the first paper we take the following description of $\LL$-graded (maximal) Cohen-Macaulay modules: A finitely generated $\LL$-graded $S$-module $M$ is (maximal) Cohen-Macaulay if and only if it satisfies the condition $\Hom(k(\vx),M)=0=\Ext^1(k(\vx),M)$ for each $\vx\in\LL$. Furtheron, Cohen-Macaulay will always mean \emph{maximal} $\LL$-graded Cohen-Macaulay.

\begin{Prop}
The sheafification functor $q:\modgr\LL{S}\ra \coh\XX$, $M\mapsto\tilde{M}$ induces an equivalence, $q:\CMgr\LL{S}\up{\approx}\vect\XX$ between the category $\CMgr\LL{S}$ of finitely generated $\LL$-graded $S$-modules  and the category $\vect\XX$ of vector bundles on $\XX$. This functor, induces an equivalence between the full subcategories $\projgr\LL{S}$ of finitely generated $\LL$-graded projective modules and the category $\Ll$ of line bundles on $\XX$. Accordingly, sheafification $q$ induces a triangle equivalence between the corresponding stable categories $\sCMgr\LL{S}$ and $\svect\XX$.\hfill\qed
\end{Prop}

We now collect some facts on the category $\coh\XX$. This category is hereditary, that is, all extensions of degree $\geq2$ vanish and it admits Serre duality in the form ${\rm D} \ExtX(F,G)  \simeq  \HomX(G,\tauX  F)$, where
${\rm D}$ denotes the usual duality $\Hom_k(-,k)$ and $\tauX F=F(\vw)$, where $\vom=\vc -\sum_{i=1}^3 \vx_i$ is the \emph{dualizing element} of $\LL$. Consequently, $\coh\XX$ has almost-split sequences, and the Auslander-Reiten translation  $\tauX $ is a self-equivalence of $\coh\XX$ given by the degree shift $X\mapsto X(\vom)$.

The complexity of the classification problem of vector bundles on $\XX$ is largely determined by the \emph{Euler characteristic} of $\XX$, given by the expression
$\chi_\XX=1/a+1/b+1/c-1$. A weighted projective line $\XX$ is said to be of \emph{domestic type}  if  $\chi_\XX>0$.
Consequently, in our setup, $\XX(a,b,c)$ is of domestic type if and only if the weight type is, up to permutation, one of the following $(2,2 ,n)$ ($n\geq2$), $(2,3 ,3)$, $(2,3 ,4)$, or $(2,3 ,5)$.

The concept of matrix factorizations was introduced by D.~Eisenbud \cite{Eisenbud:1980}. For a textbook treatment we refer to \cite{Yoshino:1990}. We recall the definition and some basic facts, adapted to the present $\LL$-graded setting. Let $T=k[x_1,x_2,x_3]$ be the polynomial algebra, viewed as $\LL$-graded algebra, and fix the polynomial $f=x_1^{a}+x_2^{b}+x_3^{c}$. An \emph{$\LL-$graded matrix factorization of $f$} is a pair of homogeneous $T$-linear maps $\phi:P_1\to P_0$ and $\psi:P_0\to P_1(\vc)$ for $\LL$-graded projective $T$-modules $P_0$ and $P_1$, notation
\begin{equation} \label{eqn:matrix.factorization}
P_1\mf{\phi}{\psi}P_0, 
\end{equation}
such that the compositions $\phi\,\psi(-\vc):P_0(-\vc)\to P_0$ and $\psi\,\phi:P_1\to P_1(\vc)$ are both the multiplication maps with $f$. Since $P_0$ and $P_1$ are $\LL$-graded free $T$-modules, we may think of $\phi$ and $\psi$ and $f\,\id$  as matrices whose entries are homogeneous members from $T$, such that the two factorization conditions translate to the matrix equation $\phi\,\psi=f\,\id=\psi\phi$. We note that the  involved degree shifts will mostly be clear from the context. For the matrix description of a matrix factorization \eqref{eqn:matrix.factorization}, we always assume that the decompositions of $P_0$ and $P_0(-\vc)$ (respectively $P_1$ and $P_1(\vc)$) into projectives of rank one are \emph{compatible}, that is, correspond to each other by degree shift $X\mapsto X(-\vc)$. We note further, that when describing a matrix factorization by matrices, we need to keep track of the direct sum decompositions of $P_0$ and $P_1$ into line bundles.

For two matrix factorizations $P_1\mf{\phi}{\psi}P_0$ and $P'_1\mf{\phi'}{\psi'}P'_0$, a pair $(F_1,F_0)$ of (homogeneous) $T$-linear maps is called a morphism of matrix factorizations, if the diagram
$$\xymatrix {P_0(-\vc)\ar[r]^{\psi(-\vc)}\ar[d]^{F_0(-\vc)}&P_1\ar[d]^{F_1}\ar[r]^{\phi}& P_0\ar[d]^{F_0}\\
 P'_0(-\vc)\ar[r]^{\psi'(-\vc)}&P'_1\ar[r]^{\phi'}& P'_0}$$
is commutative. Thinking of $F_1$ and $F_0$ (and also $\phi$ and $\psi$) as matrices whose entries are \emph{homogenous elements from} $T$,  the two commutativity conditions \eqref{eqn:matrix.factorization} translate to matrix equations $F_0\phi=\phi' F_1$ and $F_1\psi=\psi' F_0$. We remark, that a matrix factorization $P_1\mf{\phi}{\psi}P_0$ is indecomposable if and only if its endomorphism ring is local.

For any $\LL$-graded matrix factorization \eqref{eqn:matrix.factorization}, the cokernel $M=\cok{P_1\up{\phi}P_0}$ is annihilated by $f$, hence belongs to $\modgr\LL{S}$. Actually, $M$ belongs to $\CMgr\LL{S}$, and is called the (maximal) graded \emph{Cohen-Macaulay $S$-module determined by $(\phi,\psi)$}, also denoted as $\cok{\phi,\psi}$. Let $\MFgr\LL{T}{f}$ denote the category of all $\LL$-graded matrix factorizations of $f$ over $T$. Let $\Uu$ denotes the full subcategory of trivial matrix factorizations $(1,f)$, then the assignment $(\phi,\psi)\mapsto \cok\phi$ establishes an equivalence between the factor category $\MFgr\LL{T}{f}/[\Uu]$ and the category $\CMgr\LL{S}$ of $\LL$-graded Cohen-Macaulay modules over $S$. By means of the equivalence $q:\CMgr\LL{S}\ra \vect\XX$, we may as well speak of the \emph{vector bundle $E=q(\cok{\phi})$ determined by the matrix factorization $(\phi,\psi)$}.
For any projective $T$-module, the functor $\mathrm{cok}:\MFgr\LL{T}{f}\to \CMgr\LL{S}$ sends $P\mf{1}{f}P$ to zero and $P(-\vc)\mf{f}{1} P$ to the projective $S$-module  $\bar{P}=P/f.P$, and all projective $S$-modules are obtained in this way.

We are now in a position to formulate Eisenbud's matrix factorization theorem~\cite{Eisenbud:1980}, adapted to our $\LL$-graded context. We are thereby following Yoshino's presentation~\cite{Yoshino:1990}.

\begin{Thm}
Let $\Uu$, respectively $\bUu$ be the full subcategory of $\MFgr\LL{T}{f}$ consisting of all $P\mf{1}{f}P$, respectively of all matrix factorizations $P\mf{1}{f}P$ and $P(-\vc)\mf{f}{1}S$.

Then the functor $\mathrm{cok}:\MFgr\LL{T}{f}\to \CMgr\LL{S}$ induces equivalences
$$
\MFgr\LL{T}{f}/[\Uu]\lra \CMgr\LL{S} \quad \textrm{ and }\quad  \MFgr\LL{T}{f}/[\bUu]\lra \sCMgr\LL{S}.
$$
Moreover, the suspension functor of the triangulated category $\sCMgr\LL{S}$ is induced by the (functorial) expression $(P_1\mf{\phi}{\psi}P_0)[1]=P_0\mf{\psi}{\phi}P_1(\vc)$. \hfill\qed
\end{Thm}
We say that a matrix factorization $P_0\mf{\phi}{\psi}P_1$ is \emph{reduced} if $\phi$ and $\psi$ belong to the radical of $\modgr\LL{T}$, that is, if viewed as matrices, $\phi$ and $\psi$ have entries in the graded maximal ideal $(x_1,x_2,x_3)$ of $T$. The cokernel $M$ of a reduced matrix factorization is an $\LL$-graded Cohen-Macaulay module over $S=T/(f)$ without projective summands, moreover --- iterating the formation of matrix factorizations of $f$ over $T$ --- we obtain a sequence
\begin{equation} \label{eqn:sequence:matrix_factorization}
 \cdots \up{\psi} P_1(-\vc)\up{\phi} P_0(-\vc) \up{\psi}P_1\up{\phi}P_0\to M \to 0
\end{equation}
of matrix-factorizations which is 2-periodic up to degree shift with $\vc$.  Reduction modulo $(f)$ then yields the sequence
\begin{equation} \label{eqn:projective_resolution}
 \cdots\up{\bpsi} \bP_1(-\vc)\up{\bphi} \bP_0(-\vc) \up{\bpsi}\bP_1\up{\bphi}\bP_0\to M \to 0,
\end{equation}
which is a minimal $\LL$-graded projective and $2$-periodic resolution of $M$ over $S$. Here, the bar always stands for the reduction modulo $f$.

In order to determine a matrix factorization $P_1\mf{\phi}{\psi}P_0$ for a Cohen-Macaulay module $M$ without projective summands, we will first determine the  minimal projective resolution \eqref{eqn:projective_resolution} of $M$ over $S$, and then lift the $S$-matrix-pair $(\bphi,\bpsi)$ to a matrix pair $(\phi,\psi)$ over $T$, such that \emph{additionally the condition} $\phi\,\psi=f\,\id=\psi\,\phi$ holds.

For weight triples $(2,a,b)$, the suspension functor $[1]$ for $\svect\XX$ is induced by the degree shift $X\mapsto X(\vx_1)$ by $\vx_1$, see \cite[Proposition 6.8]{Kussin:Lenzing:Meltzer:2013adv}. This allows to introduce the concept of a \emph{symmetric matrix factorization} of $f$ for a graded Cohen-Macaulay module $M$ without projective summands (correspondingly for a vector bundle $E$ without line bundle summands). Namely, we call a matrix factorization $P_1\mf\phi\psi P_0$ for $M$ \emph{symmetric} --- recall that then $\phi:P_1\to P_0$ and $\psi:P_0\to P_1(\vc)$ are homogeneous $T$-linear maps --- provided we have $P_0=P_1(\vx_1)$ and further $\psi=\phi(\vx_1)$. Note that this request makes sense since $2\vx_1=\vc$.  In this case, by abuse of notation, we will --- as for ungraded symmetric matrix factorisations --- also write $\phi=\psi$. We will show in Section \ref{sect:matrix.factorizations} that for weight type $(2,a,b)$, each indecomposable vector bundle of rank two is determined by a symmetric matrix factorization of $f$. Moreover, if we deal with domestic type, necessarily given by a weight triple $(2,a,b)$, then each indecomposable vector bundle of rank at least two will admit a symmetric matrix factorization by Theorem \ref{thm:main:domestic}.

Assuming weight triple $(p_1,p_2,p_3)$ we put $\bar{p}=l.c.m.(p_1,p_2,p_3)$.
There is a unique group homomorphism $\delta: \LL \lra \ZZ$ called the \emph{ degree map} which sends $\vx_i$ to $\frac{\bar{p}}{p_i}$.
The kernel of $\delta$ is the torsion group $t\LL$ of $\LL$.
Further, there is a unique group homomorphism $\deg : \Knull (\coh \XX) \lra \ZZ$, called the  \emph{ degree}, such that $\deg([\Oo(\vx)])=\delta(\vx)$ holds
for each $\vx \in \LL$.
For each non-zero $X \in \coh \XX$ at least one of $\rk (X)$ or $\deg(X)$ is
non-zero, yielding a well defined \emph{ slope} $\mu (X)=\frac{\deg(X)}{\rk (X)}$
in the extended rationals $\QQ \cup \{\infty\}$.
The slope of an indecomposable object $X$ is a useful indicator of the position of $X$ in the category $\coh \XX$. In the domestic situation, moreover, each indecomposable vector bundle $X$ is \emph{stable}, that is, satisfies $\mu (X')<\mu(X)$ for each proper subobject $0 \neq X' \subsetneq X$. Still assuming domestic type, stability of a non-zero vector bundle  $X$  implies $\End (X)=k$
and $\ExtX (X,X)=0$, that is the exceptionality of $X$.
For all foregoing facts see \cite{Geigle:Lenzing:1987}.

For certain investigations a refinement of the degree, called \emph{determinant},
is necessary. This is a group homomorphism $\det : \Knull (\coh \XX) \lra \LL$
such that $\det (\Oo(\vx))= \vx$ holds for each $\vx \in \LL$.
In particular, we have $\deg= \delta \circ \det $, see
 \cite[2.7]{LM}.

By means of a line bundle filtration for a vector bundle $E$ one further obtains the formula

\begin{equation}\label{detformula}
\det(E(\vx))=\det(E) +\rk(E) \cdot \vx \quad \text{for all} \quad \vx \in \LL
\end{equation}
We finally recall from \cite{Geigle:Lenzing:1987} that the category $\coh\XX$ has almost-split sequences with the Auslander-Reiten translation given by degree shift $X\mapsto X(\vom)$ with the dualizing element $\vom=\vc-\sum_{i=1}^3\vx_i$.

\subsection*{The category of vector bundles for domestic weight triples}

A weight triple $(a,b,c)$ with entries $\geq2$ has domestic type, if and only if it is one of $(2,2,n)$, $n\geq2$, $(2,3,3)$, $(2,3,4)$ or $(2,3,5)$. The shape of the Auslander-Reiten quiver of $\vect\XX$ then is $\ZZ\Delta$, where $\Delta$ is the extended Dynkin diagram, attached to the Dynkin star $[a,b,c]$. The category of indecomposable vector bundles then is equivalent to the mesh category $k(\ZZ\Delta)$.
In this case, the stable category $\svect\XX$ is equivalent to the bounded derived category $D^b(\mod{\La})$ for the path algebra $\La=kQ$ of Dynkin type $\Delta'$ obtained from $\Delta$ by removing all vertices where the standard additive function on $\Delta$ takes value $1$ \cite[Section 5.1]{Kussin:Lenzing:Meltzer:2013adv}. The table below summarizes the situation.

\small
\begin{center}
\renewcommand{\arraystretch}{1.4}
\begin{tabular}{|c|c|c|c|c|} \hline
weight triple & $(2,2,n)$ & $(2,3,3)$ & $(2,3,4)$& $(2,3,5)$ \\ \hline
$\Delta$  & $\widetilde\DD_{n+2}$ & $\widetilde\EE_6=[3,3,3]$ & $\widetilde\EE_7=[2,4,4]$& $\widetilde\EE_8=[2,3,6]$\\ \hline
$\Delta'$ & $\mathbb{A}_{n-1}=[n-1]$   & $\DD_4=[2,2,2]$ & $\EE_6=[2,3,3]$& $\EE_8=[2,3,5]$\\ \hline
\end{tabular}
\end{center}
\normalsize
For the rest of the paper, it is important to understand how the Picard group $\LL$ acts on the mesh category $k(\ZZ\Delta)$, or on the underlying translation quiver $\ZZ\Delta$, by degree shift. We illustrate this for the weight triple $(2,3,4)$, where a piece of the Auslander-Reiten quiver is depicted below. The considerations are similar for the other domestic weight triples. We first remark that the rank of vector bundles is constant on $\tau$-orbits; the values of the rank are displayed at the right end.

$$
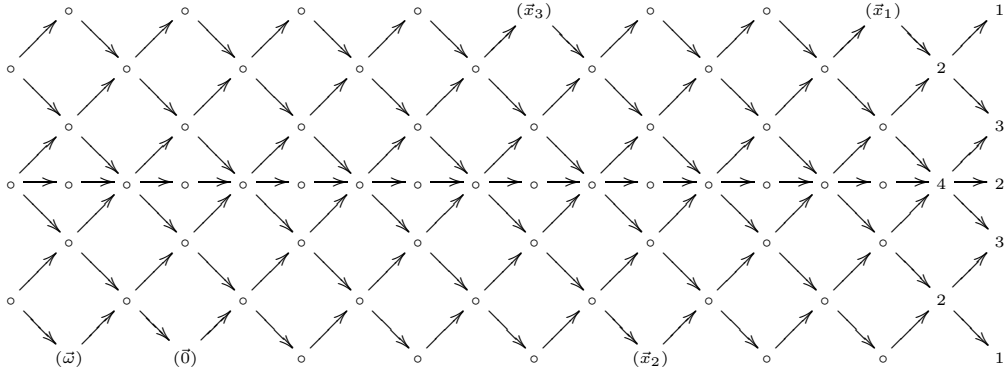
\tiny{
\def\c{\circ}
\def\b{\circ}
\xymatrix @R22pt @C22pt @!0{
& \c\ar[rd]       &                       &\c\ar[rd]       &                       &\c\ar[rd]       &                       &\c\ar[rd]       &               & (\vx_3)\ar[rd]       &                       &\c\ar[rd]       &                       &\c\ar[rd]       &                       &(\vx_1)\ar[rd]       &               &1&\\
\b\ar[ru]\ar[rd]       &                &\b\ar[ru]\ar[rd]       &                &\b\ar[ru]\ar[rd]       &                &\b\ar[ru]\ar[rd]       &                &\b\ar[ru]\ar[rd]&      &\b\ar[ru]\ar[rd]       &                &\b\ar[ru]\ar[rd]       &                &\b\ar[ru]\ar[rd]       &                &2\ar[ru]\ar[rd]&&& &\\
&\b\ar[ru]\ar[rd]&                       &\b\ar[ru]\ar[rd]&                       &\b\ar[ru]\ar[rd]&                       &\b\ar[ru]\ar[rd]&                       &\b\ar[ru]\ar[rd]&                       &\b\ar[ru]\ar[rd]&                       &\b\ar[ru]\ar[rd]&                       &\b\ar[ru]\ar[rd]&                       &3\\
\b\ar[ru]\ar[rd]\ar[r]&\b\ar[r]        &\b\ar[ru]\ar[rd]\ar[r]&\b\ar[r]         &\b\ar[ru]\ar[rd]\ar[r]&\b\ar[r]        &\b\ar[ru]\ar[rd]\ar[r]&\b\ar[r] &\b\ar[ru]\ar[rd]\ar[r]&\b\ar[r]        &\b\ar[ru]\ar[rd]\ar[r]&\b\ar[r]         &\b\ar[ru]\ar[rd]\ar[r]&\b\ar[r]        &\b\ar[ru]\ar[rd]\ar[r]&\b\ar[r] &4\ar[ru]\ar[rd]\ar[r]&2&\\
& \b\ar[ru]\ar[rd]&                       &\b\ar[ru]\ar[rd]&                       &\b\ar[ru]\ar[rd]&                       &\b\ar[ru]\ar[rd]&                       & \b\ar[ru]\ar[rd]&                       &\b\ar[ru]\ar[rd]&                       &\b\ar[ru]\ar[rd]&                       &\b\ar[ru]\ar[rd]&                       &3&\\
\b\ar[ru]\ar[rd]       &                &\b\ar[ru]\ar[rd]       &                &\b\ar[ru]\ar[rd]       &                &\b\ar[ru]\ar[rd]       &                &\b\ar[ru]\ar[rd]&                &\b\ar[ru]\ar[rd]       &                &\b\ar[ru]\ar[rd]       &                &\b\ar[ru]\ar[rd]       &                &2\ar[ru]\ar[rd]       &            &\\
&({\vom})\ar[ru]       &                       &(\vec{0})\ar[ru]       &                       &\c\ar[ru]       &                       &\c\ar[ru]       &                       &\c\ar[ru]       &                       &(\vx_2)\ar[ru]       &                       &\c\ar[ru]       &                       &\c\ar[ru]       &                       &1& &\\}
}
$$
We thus have two $\tau$-orbits of line bundles, the lower and the upper border, three $\tau$-orbits of indecomposable rank-two bundles,  two $\tau$-orbits of indecomposable bundles of rank 3 and a single $\tau$-orbit of rank 4. Since the Picard group acts transitively on the iso-classes of line bundles, we may freely choose the position of the structure sheaf from one of the two line bundle orbits. Once this is done, the position of the other line bundles is fixed, up to a symmetry of $\ZZ \Delta$. To indicate the position of a line bundle $\Oo(\vx)$, we use the bracket notation $(\vx)$ such that the structure sheaf is given by the symbol $(\vec{0})$, and its Auslander-Reiten translate $\tau\Oo$ is given by the symbol $(\vom)$, where $\vom=\vc-(\vx_1+\vx_2+\vx_3)$ and hence $\delta(\vom)=-1$. This now allows to determine easily the values of the degree function for each indecomposable vector bundle.  Since $\Oo(\vx_3)$ has degree $3$, and $\HomX(\Oo,\Oo(\vx_3))=k$, there is only one choice for position $(\vx_3)$, once the position $(\vec{0})$ has been fixed. All further line bundles then are given by one of the symbols $(\vec{0}+n\vom)$, respectively $(\vx_3+n\vom)$, with $n\in\ZZ$.

Corresponding to the positions  $(\vx_1)$, $(\vx_2)$ and $(\vx_3)$ in the mesh-category, the shift actions by $\vx_1$, $\vx_2$ and $\vx_3$ are given as follows: The shift by $\vx_1$ (resp.\ $\vx_3$) is a glide reflection, composed by the reflection with respect to the central horizontal axis with the sixth respectively third power of $\tau^-$. Further, the shift action by $\vx_2$ equals $\tau^{-4}$.

Finally, let us remark that, obviously, the factor category $\vect \XX /[\Ll]$, obtained from $\vect\XX$, for $\XX=\XX(2,3,4)$, by factoring out the two line bundle orbits yields the mesh category $k(\ZZ\EE_6)$, equivalent to $D^b(\mod {kQ})$ for a quiver $Q$ having type $\EE_6$, thus illustrating the facts mentioned at the beginning of this section.

\begin{Rem}
	\label{Remark_shift_by_x_1}
In view of Theorem \ref{Thm:2ab}, it is useful to interpret the degree shift by $\vx_1$ in terms of the Auslander-Reiten quiver of $\vect\XX$ resp.\ $\svect\XX$. For this, we assume domestic type $(2,a,b)$.
\begin{itemize}
\item[$\bullet$] For type $(2,3,5)$ we have $\vx_1=-15\vw$. Thus the degree shift by $\vx_1$ is the translation to the right by $15$ mesh-units.
\item[$\bullet$] For type $(2,3,4)$ we have $\vx_1=-6\vom + (\vx_1-2\vx_3)$. We note that the element $\vx_1-2\vx_3$ has order two. Thus the degree shift by $\vx_1$ is the glide reflection given by composing the reflection on the central axis with the translation to the right by $6$ mesh-units.
\item[$\bullet$] For type $(2,3,3)$, we have $\vx_1=-3\vw$, and the degree shift by $\vx_1$ is the translation to the right by $3$ mesh-units.
\item[$\bullet$] For type $(2,2,n)$, we use that the degree shifts by $\vx_1$, $\vx_2$ and $-\vom$ agree on objects of $\svect\XX$ \cite[page 235]{Kussin:Lenzing:Meltzer:2013adv}, and only deal with the shift action of $\vx_1$ on objects of $\svect\XX$. Further, we need to distinguish whether $n$ is even or odd: For $n=2k$ (resp.\ $n=2k+1$) the degree shifts by $-k\vom$ and $\vx_1+(k\vx_3-\vx_1)$ (resp.\ by $-k\vom$ and $\vx_1+(k\vx_3-\vx_1)$) agree on objects of $\svect\XX$. In the first case the element $k\vx_3-\vx_1$ has order two, while in the second case we obtain $2(k\vx_3-\vx_1)=-\vx_3$.  Hence  the degree shift by $\vx_1$ on the Auslander-Reiten quiver of $\svect\XX$ is the glide reflection given by composing the reflection on the central axis with the translation to the right by $k$ mesh-units (resp.\ by $k+1/2$ mesh-units).
\end{itemize}
\end{Rem}

\section{Projective covers}
	\label{sec_proj_covers}

\emph{When speaking of weight triples, we always assume that the weights are at least two. In the domestic case this just excludes the weight types $(\;)$, $(a)$, and $(a,b)$ where each indecomposable vector bundle is a line bundle, and the matrix factorization problem thus becomes trivial.}

\subsection*{General results}

Assuming an arbitrary weight triple $(p_1,p_2,p_3)$, this section starts by quoting two general results \cite[Theorems 4.2 and 4.6]{Kussin:Lenzing:Meltzer:2013adv} on indecomposable vector bundles of rank two and their projective covers in $\vect\XX$. We recall that the double suspension functor for $\svect\XX$ is induced by degree shift with the canonical element $\vc$. Moreover, for weight triples $(2,a,b)$, the suspension functor itself is induced by the degree shift with $\vx_1$ \cite[Proposition 6.8]{Kussin:Lenzing:Meltzer:2013adv}.  Switching now to weight triples of domestic type, necessarily of type $(2,a,b)$, the aim of this section is to determine the projective cover (likewise the injective hull) for each indecomposable vector bundle of rank $\geq2$.

We assume triple weight type $(p_1,p_2,p_3)$. Let $\delta=\vc+2\vom$ be the \emph{dominant element} of $\LL$. The elements $\vnull\leq\vx\leq\vdom$ then have the form $\vx=\sum_{i=1}^3l_i\vx_i$ with $0\leq l_i\leq p_i-2$.
Following \cite[section~4]{Kussin:Lenzing:Meltzer:2013adv}, a vector bundle $E$ of rank $2$ is called an \emph{extension bundle} if $E$ is the middle term of a non-split exact sequence
\begin{equation} \label{eqn:ext_bundle}
\eta_\vx:\quad0\to L(\vw)\to E \to L(\vx)\to 0,
\end{equation}
where $L$ is a line bundle and $\vnull\leq \vx\leq \vec{\delta}$.
Because $\ExtX(L(\vx),L(\vw) )=k$, the bundle  $E$ is uniquely determined, up to isomorphism; we then denote $E$ by $\extb{L}{\vx}$. For $L=\Oo$ we just write $\extbo{\vx}$. If $\vx=0$, then the sequence $\eta_\vx$ is almost-split, and $E=\extb{L}{0}$ is called an \emph{Auslander bundle}, more precisely the Auslander bundle attached to $L$. Applying degree shift by $\vy$ from $\LL$ to the exact sequence \eqref{eqn:ext_bundle}, we obtain the useful identity
\begin{equation} \label{eqn:ext_bundle:shift}
(\extb{L}{\vx})(\vy)\cong \extb{L(\vy)}{\vx} \quad \text{ for all }\quad 0\leq \vx\leq \delta,\, \vy\in\LL.
\end{equation}

We recall, that an object $E$ in an abelian (resp.\ a triangulated category) is \emph{exceptional} if $\End(E)=k$ and further  $\Ext^d(E,E)=0$ (resp. $\Hom(E,E[d])=0$) for each integer $d\neq0$. For objects of a hereditary category, like $\coh\XX$, the Ext-condition only needs to be checked for $d=1$.

The following three theorems from \cite{Kussin:Lenzing:Meltzer:2013adv} mark the starting point of our investigation.
For the first one we refer to Theorem 4.2 and Corollary 4.11, for the second one to Theorem 4.6 from the quoted paper. We recall that $\vdom=\sum_{i=1}^3 (p_i-2)\vx_i$ denotes the dominant element of $\LL$.

\begin{Thm}[Vector bundles of rank two]
Assume $\XX$ is given by a weight triple $(p_1, p_2, p_3)$. Then the following holds:

\begin{enumerate}
\item[(i)] Each indecomposable vector bundle of rank two is isomorphic to an extension bundle $\extb{L}{\vx}$ for a suitable choice of a line bundle $L$ and an element $\vx$ from $\LL$ satisfying $0\leq \vx \leq \vdom$.
\item[(ii)] Each indecomposable vector bundle of rank two is exceptional in the category $\coh\XX$ of coherent sheaves on $\XX$. It is also exceptional in the stable category of vector bundles $\svect\XX$. \hfill\qed
\end{enumerate}
\end{Thm}

\begin{Thm}[Projective and injective covers]
\label{KLM_proj_cov}
Assume $\XX$ is given by the weight
triple $(p_1, p_2, p_3)$. Let  $\extb{L}{\vx}$, $0 \leq \vx\leq \vdom$, be an extension bundle. Then
its injective hull $\ih{\extb{L}{\vx}}$ and its projective cover $\pc{\extb{L}{\vx}}$ are given by the
following expressions:
\begin{align}
\ih{\extb{L}{\vx}}&= L(\vx)\oplus \bigoplus_{i=1}^3 L((1+l_i)\vx_i+\vw)\\
\pc{\extb{L}{\vx}}&= L(\vw)\oplus \bigoplus_{i=1}^3 L(\vx-(1+l_i)\vx_i),
\end{align}
where $\vx=l_1\vx_1+l_2\vx_2+l_3\vx_3$.

Further, the four line bundle summands $(L_i)_{i=0}^3$
 of $\ih{\extb{L}{\vx}}$ (resp. $\pc{\extb{L}{\vx}}$) are
mutually Hom-orthogonal, that is, they are satisfying
\[\Hom(L_i,L_j)=\left\{\begin{array}{lr}
k& if\ i=j\\
0& if\ i\neq j
\end{array}\right.\]\hfill\qed
\end{Thm}
The next result is a straightforward consequence of \cite[Proposition 6.8]{Kussin:Lenzing:Meltzer:2013adv}.
\begin{Thm}[Weight type $(2,a,b)$] \label{Thm:2ab}
Let $\XX$ be the weighted projective line of type $(2,a,b)$ and $E$ be an indecomposable vector bundle of rank at least $2$. There is a distinguished short exact sequence
$$0\lra E(-\vx_1)\lra \pc{E}\stackrel{\pi_E}{\lra}E\lra 0,$$
where $\pc{E}$ is the projective cover of $E$, and likewise the injective hull $\ih{E(-\vx_1)}$ of $E(-\vx_1)$.
In particular, $\ih{E}=\pc{E}(\vx_1)$ and $\rk \pc{E}=2\rk E$.\hfill\qed
\end{Thm}

The following variant of the `horse-shoe lemma' from homological algebra will be used to determine projective covers for vector bundles of larger rank. A dual result is valid for injective hulls.

\begin{Lem}
	\label{proj_cover_from_sequence}
We assume weight type $(2,a,b)$. Let $X$ and $Y$ be vector bundles with projective covers $\pc{X} \stackrel{\pi_X}{\lra} X$ and $\pc{Y} \stackrel{\pi_Y}{\lra} Y$.
Let
$$(\star) \quad 0\lra X \stackrel{f}{\lra} E \stackrel{g}{\lra} Y\lra 0,$$
be a distinguished exact sequence in $\vect\XX$. (This condition is satisfied if ($\star$) is exact and $\ExtX(\pc{Y},X)$ is zero). Then $\pi_Y$ lifts to a map $\pi^*_Y:\pc{Y}\to E$, yielding a
\def\hh{\left[\begin{array}{cc}
f\circ\pi_X,\alpha
\end{array}\right]}
commutative diagram
$$\xymatrix{
0\ar[r]&\pc{X}\ar[r]\ar[d]_{\pi_X}     &\pc{X}\oplus\pc{Y}\ar[r]\ar[d]_{[f\pi_X,\pi^*_Y]}& \pc{Y}\ar[r]\ar[d]_{\pi_Y}&0\\
0\ar[r]&      X\ar[r]^{f}&E\ar[r]^{g}             & Y\ar[r]& 0
}$$ which establishes $\pc{X}\oplus\pc{Y}$ as a projective cover of $E$. Moreover, the rows of the diagram are distinguished exact and the vertical maps are distinguished epimorphisms.
\end{Lem}

\begin{pf}
We show that the condition $\ExtX(\pc{Y},X)=0$ implies that ($\star$) is distinguished exact. Indeed, applying the functor $\HomX(\pc{Y},-)$ to the exact sequence $(\star)$, we obtain a short exact sequence
$$0\lra \HomX(\pc{Y},X)\stackrel{f\circ-}{\lra} \HomX(\pc{Y},E) \stackrel{g\circ-}{\lra} \HomX(\pc{Y},Y)\lra \ExtX(\pc{Y},X)=0$$ showing that each morphism from $\pc{Y}$ to $X$ lifts to $E$. By definition of the projective cover, further each morphism from a line bundle $L$ to $Y$ lifts to $\pc{Y}$. Putting things together, any morphism from a line bundle $L$ to $Y$ lifts to a map $L\to E$, showing that ($\star$) is distinguished exact.
Any of the two assumptions thus ensures that $\pi_Y$ lifts to a map $\alpha:\pc{Y}\to E$, yielding the above diagram, and its claimed properties, by standard arguments of (relative) homological algebra.

It remains to point out why the minimality condition for the distinguished epimorphism $\pi_E:\pc{X}\oplus\pc{Y}\to E$ holds. Because of weight type $(2,a,b)$, we know from Theorem \ref{Thm:2ab} that $\pc{E}$ and $\pc{X}\oplus \pc{Y}$ have the same rank $2\rk E$, which ensures the claim.
\end{pf}

We keep to assume weight type $(2,a,b)$. To obtain minimal projective resolutions, we have to determine those morphisms that are compositions of a projective cover $\pc{E}\stackrel{\pi_E}{\lra}E$ with the corresponding injective hull $E\stackrel{j_E}{\lra}\ih{E}$. The resulting morphism $u_E:\pc{E}\to \ih{E}$, $u_E=j_E\,\pi_E$, will be called a \emph{cover morphism} for $E$. (Note that such cover morphisms depend on the chosen projective cover and injective hull). Again, we reduce the determination of cover morphisms to the case of smaller rank.

\begin{Lem}
	\label{matrix_frame_from_sequence}
Let $\XX$ be of weight type $(2,a,b)$.
Let $ (\star)\quad 0\lra X \stackrel{f}{\lra} E \stackrel{g}{\lra} Y\lra 0,$ be a distinguished exact sequence in $\vect\XX$. (This condition is satisfied if $(\star)$ is exact and if $\ExtX(\pc{Y},X)=0$). Let $\overline{u}_X$ (respectively $\overline{u}_Y$) be a cover morphisms for $X$ (respectively $Y$).

Then we obtain a cover morphism for $E$ having shape
$$\overline{u}_E=\left[\begin{array}{c|c}
\overline{u}_X&\beta\circ \alpha\\
\hline
0&\overline{u}_Y
\end{array}\right]$$
 where
$g\circ\alpha=\pi_Y$, and $\beta\circ f=j_X$.
\end{Lem}

\begin{pf}
From Lemma \ref{proj_cover_from_sequence}
we obtain the following commutative diagram with exact rows
\newcommand{\pczero}[1]{\mathfrak{P}_0(#1)}
\newcommand{\pcone}[1]{\mathfrak{P}_1(#1)}
\newcommand\mati{\left[\begin{array}{l}1\\0\end{array}\right]}
\newcommand\matpi{\left[\begin{array}{ll}0&1\end{array}\right]}
\newcommand\matproj{\left[\begin{array}{ll}f\circ\pi_{X},&\alpha\end{array}\right]}
\newcommand\matinj{\left[\begin{array}{c}\beta\\ j_{ Y}\circ g\end{array}\right]}
$$\xymatrix{
0\ar[r]&{\pc{X}}\ar[rr]^-{\mati}\ar[dd]_-{\pi_{X}}&&{\pc{X}\oplus \pc{Y}} \ar[rr]^-{\matpi}\ar[dd]^-{\matproj}&&{\pc{Y}}\ar[r]\ar[dd]^-{\pi_{Y}}&0\\ \\
0\ar[r]&X\ar[rr]^-{f}\ar[dd]_-{j_{X}}&&E\ar[dd]^-{\matinj}\ar[rr]^{g}&&Y\ar[r]\ar[dd]^-{j_{ Y}}&0\\ \\
0\ar[r]&{\ih{X}}\ar[rr]^-{\mati}&&{\ih{X}\oplus \ih{Y}}\ar[rr]^-\matpi&&{\ih{Y}}\ar[r]&0.}$$

Therefore $\overline{u}_E=\overline{u}_{E(\vx_1)}=
\left[\begin{array}{c}
\beta\\
j_Y\circ g
\end{array}\right]\circ
\left[\begin{array}{cc}
f\circ \pi_X&\alpha
\end{array}\right]=
\left[\begin{array}{c|c}
\beta\circ f\circ \pi_X&\beta\circ \alpha\\
\hline
j_Y\circ g\circ f\circ \pi_X& j_Y\circ g\circ \alpha
\end{array}\right]=
\left[\begin{array}{c|c}
j_X\circ \pi_X&\beta\circ \alpha\\
\hline
0& j_Y\circ \pi_Y
\end{array}\right]=
\left[\begin{array}{c|c}
\overline{u}_{X(\vx_1)}&\beta\circ \alpha\\
\hline
0&\overline{u}_{Y(\vx_1)}
\end{array}\right]=
\left[\begin{array}{c|c}
\overline{u}_{X}&\beta\circ \alpha\\
\hline
0&\overline{u}_{Y}
\end{array}\right].$
\end{pf}

From now on, we restrict to weighted projective lines $\XX$ of domestic type. For each of the weight types $(2,2,n)$, $(2,3,3)$, $(2,3,4)$ and $(2,3,5)$ we determine in Propositions \ref{proj_cover_2_2_n} to \ref{proj_cover_2_3_5} the projective covers of indecomposable vector bundles, by Theorem \ref{Thm:2ab} thus also their injective hulls. The following propositions list from each Auslander-Reiten orbit of indecomposables of rank $r\geq2$ a particular member, say $E$, and represent its projective cover $\pc{E}=\bigoplus_{i=1}^{2r}\Oo(\vy_i)$ by the sequence $\vy_1,\vy_2,\ldots,\vy_{2r}$ (including multiplicities). For the switch $\ih{E}=\pc{E}(\vx_1)$ from projective covers to injective hulls we further refer to the interpretation of the degree shift by $\vx_1$ given in Remark \ref{Remark_shift_by_x_1}.

\subsection*{Case $(2,2,n)$}

In this case the Auslander-Reiten quiver of $\vect\XX$ has shape $\ZZ\widetilde\DD_{n+2}$. Note that we need to distinguish the cases $n$ even (resp. $n$ odd): For a given integral slope, there are exactly $4$ (resp. $2$) line bundles if  $n$ is even (resp. if $n$ is odd).

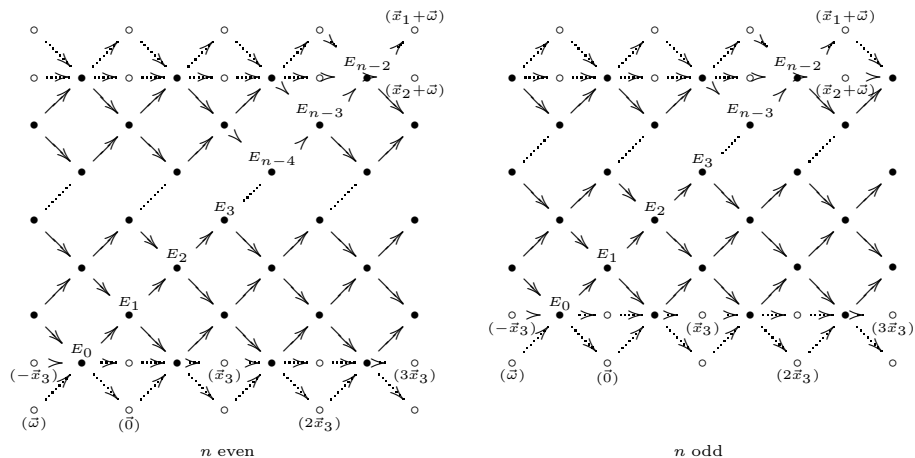
\newcommand{\evencase}{\xymatrix@C18pt@R18pt  @!0{
\c\ar@{..>}[rd]&&\c\ar@{..>}[rd]&&\c\ar@{..>}[rd]&&\c\ar@{..>}[rd]&&\tc{(\vx_1+\vom)}\\
\c\ar@{..>}[r]&\b\ar@{..>}[ru]\ar@{..>}[r]\ar[rd]&\c\ar@{..>}[r]&\b\ar[rd]\ar@{..>}[r]\ar@{..>}[ru]&\c\ar@{..>}[r]&
\b\ar[rd]\ar@{..>}[r]\ar@{..>}[ru]&\c\ar@{..>}[r]&\tb{E_{n-2}}\ar[rd]\ar@{..>}[r]\ar@{..>}[ru]&\bc{(\vx_2+\vom)}\\
\b\ar[rd]\ar[ru]&&\b\ar[rd]\ar[ru]&&\b\ar[rd]\ar[ru]&&\tb{E_{n-3}}\ar[rd]\ar[ru]&&\b\\
&\b\ar[ru]&&\b\ar[ru]&&\tb{E_{n-4}}\ar[ru]&&\b\ar[ru]&\\
\b\ar[rd]\ar@{.}[ru]&&\b\ar[rd]\ar@{.}[ru]&&\tb{E_3}\ar[rd]\ar@{.}[ru]&&\b\ar[rd]\ar@{.}[ru]&&\b\\
&\b\ar[rd]\ar[ru]&&\tb{E_2}\ar[rd]\ar[ru]&&\b\ar[rd]\ar[ru]&&\b\ar[rd]\ar[ru]&\\
\b\ar[rd]\ar[ru]&&\tb{E_1}\ar[rd]\ar[ru]&&\b\ar[rd]\ar[ru]&&\b\ar[rd]\ar[ru]&&\b\\
\bc{(-\vx_3)}\ar@{..>}[r]&\tb{E_0}\ar@{..>}[rd]\ar[ru]\ar@{..>}[r]&\c\ar@{..>}[r]&\b\ar@{..>}[rd]
\ar[ru]\ar@{..>}[r]&\bc{(\vx_3)}\ar@{..>}[r]&\b\ar@{..>}[rd]\ar[ru]\ar@{..>}[r]&
\c\ar@{..>}[r]&\b\ar@{..>}[rd]\ar[ru]\ar@{..>}[r]&\bc{(3\vx_3)}\\
\bc{(\vom)}\ar@{..>}[ru]&&\bc{(\vnull)}\ar@{..>}[ru]&&\c\ar@{..>}[ru]&&\bc{(2\vx_3)}\ar@{..>}[ru]&&\c\\}}

\newcommand{\oddcase}{\xymatrix @C18pt@R18pt @!0{
&\c\ar@{..>}[rd]&&\c\ar@{..>}[rd]&&\c\ar@{..>}[rd]&&\tc{(\vx_1+\vom)}\ar@{..>}[rd]&\\
\b\ar@{..>}[ru]\ar[rd]\ar@{..>}[r]&\c\ar@{..>}[r]&\b\ar@{..>}[ru]
\ar[rd]\ar@{..>}[r]&\c\ar@{..>}[r]&\b\ar@{..>}[ru]\ar[rd]\ar@{..>}[r]&\c\ar@{..>}[r]&
\tb{E_{n-2}}\ar@{..>}[ru]\ar[rd]\ar@{..>}[r] &\bc{(\vx_2+\vom)}\ar@{..>}[r]&\b\\
&\b\ar[ru]&&\b\ar[ru]&&\tb{E_{n-3}}\ar[ru]&&\b\ar[ru]&\\
\b\ar[rd]\ar@{.}[ru]&&\b\ar[rd]\ar@{.}[ru]&&\tb{E_3}\ar[rd]\ar@{.}[ru]&&\b\ar[rd]\ar@{.}[ru]&&\b\\
&\b\ar[rd]\ar[ru]&&\tb{E_2}\ar[rd]\ar[ru]&&\b\ar[rd]\ar[ru]&&\b\ar[rd]\ar[ru]&\\
\b\ar[rd]\ar[ru]&&\tb{E_1}\ar[rd]\ar[ru]&&\b\ar[rd]\ar[ru]&&\b\ar[rd]\ar[ru]&&\b\\
\bc{(-\vx_3)}\ar@{..>}[r]&\tb{E_0}\ar@{..>}[rd]\ar[ru]\ar@{..>}[r]&\c\ar@{..>}[r]&\b\ar@{..>}[rd]
\ar[ru]\ar@{..>}[r]&\bc{(\vx_3)}\ar@{..>}[r]&\b\ar@{..>}[rd]\ar[ru]\ar@{..>}[r]&
\c\ar@{..>}[r]&\b\ar@{..>}[rd]\ar[ru]\ar@{..>}[r]&\bc{(3\vx_3)}\\
\bc{(\vom)}\ar@{..>}[ru]&&\bc{(\vnull)}\ar@{..>}[ru]&&\c\ar@{..>}[ru]&&\bc{(2\vx_3)}\ar@{..>}[ru]&&\c\\}}
\tiny
$$
\begin{array}{cc}
\evencase & \oddcase \\
\textit{n} \textrm{ even} & \textit{n} \textrm{ odd }
\end{array}
$$
\normalsize
Here, the symbol $\bullet$ ( resp. $\circ$) marks the position of a vector bundle of rank $2$ (resp.\  of a line bundle). Specifically, for key values of $\vx$ in $\LL$, the position of the line bundle $\Oo(\vx)$ is indicated by the bracket-symbol $(\vx)$. Moreover, we have marked the position of $n-1$ vector bundles $E_0,\ldots,E_{n-2}$ of rank two, one for each $\tau$-orbit.

\begin{Prop}
	\label{proj_cover_2_2_n}
We assume weight type $(2,2,n)$ and refer to the notations from the above figure. Each indecomposable vector bundle, that is not a line bundle, has rank two. It further lies in the $\tau$-orbit of exactly one of the extension bundles $E_0,\ldots,E_{n-2}$,
where  $E_i$ is determined by the pair $(\Oo, i\vx_3 )$. Moreover, we have
$$\pc{E_i}=\Oo(\vom)\oplus \Oo(i\vx_3-\vx_1)\oplus \Oo(i\vx_3-\vx_2)\oplus \Oo(-\vx_3).
$$
\end{Prop}

\begin{pf}
There are non-split exact sequences $0\lra \Oo(\vw)\lra E_i\lra \Oo(i\vx_3)\lra 0$. Therefore $E_i$ is isomorphic to the extension bundle $\extbo{i\vx_3}$, and the claim concerning projective covers follows from Theorem \ref{KLM_proj_cov}.
\end{pf}

\subsection*{Case $(2,3,3)$}

In this case, we have three $\tau$-orbits of line bundles, and the Auslander-Reiten quiver of $\vect\XX$ has shape $\ZZ\widetilde\EE_6$:
{\footnotesize
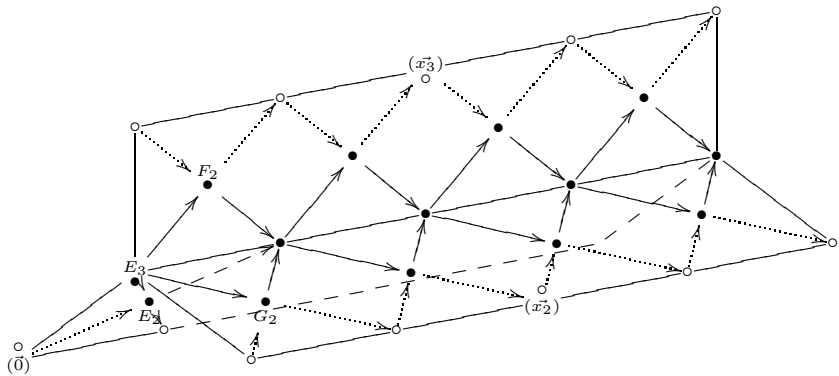
$$\xymatrix @=5.5pt @!0 {&&&&&&&&&&&&&&&&&&&&&&&&&&&&&&&&&&&&&&&&&&&&&&&&&&&&&&&&&&&&\\
&&&&&&&&&&&&&&&&&&&&&&&&&&&&&&&&&&&&&&&&&&&&&&&&&&&&&&&&&&&&\\
&&&&&&&&&&&&&&&&&&&&&&&&&&&&&&&&&&&&&&&&&&&&&&&&*+<1pt>{\c}&&&&&&&&&&&\\
&&&&&&&&&&&&&&&&&&&&&&&&&&&&&&&&&&&&&&&&&&&&&&&&&&&&&&&&&&&&\\
&&&&&&&&&&&&&&&&&&&&&&&&&&&&&&&&&&&&&&*+<1pt>{\c}\ar @{-}@[ashy][-2,10]\ar @{..>}[4,5] &&&&&&&&&&&&&&&&&&&\\
&&&&&&&&&&&&&&&&&&&&&&&&&&&&&&&&&&&&&&&&&&&&&&&&&&&&&&&&&&&&\\
&&&&&&&&&&&&&&&&&&&&&&&&&&&&*+<1pt>{\tcv{x_3}}\ar @{-}@[ashy][-2,10]\ar @{..>}[4,5] &&&&&&&&&&&&&&&&&&&&&&&&&&&&&&&&\\
&&&&&&&&&&&&&&&&&&&&&&&&&&&&&&&&&&&&&&&&&&&&&&&&&&&&&&&&&&&&\\
&&&&&&&&&&&&&&&&&&*+<1pt>{\c}\ar @{-}@[ashy][-2,10]\ar @{..>}[4,5] &&&&&&&&&&&&&&&&&&&&&&&&&\b\ar @{..>}[-6,5]\ar [4,5] &&&&&&&&&&&&&&&&&\\
&&&&&&&&&&&&&&&&&&&&&&&&&&&&&&&&&&&&&&&&&&&&&&&&&&&&&&&&&&\\
&&&&&&&&*+<1pt>{\c}\ar @{-}@[ashy][-2,10]\ar @{..>}[4,5] &&&&&&&&&&&&&&&&&&&&&&&&&\b\ar @{..>}[-6,5]\ar [4,5] &&&&&&&&&&&&&&&&&&&&&&&&&&&\\
&&&&&&&&&&&&&&&&&&&&&&&&&&&&&&&&&&&&&&&&&&&&&&&&&&&&&&&&&&\\
&&&&&&&&&&&&&&&&&&&&&&&\b\ar @{..>}[-6,5]\ar [4,5] &&&&&&&&&&&&&&&&&&&&&&&&&*+<1pt>{\b}\ar @{-}@[ashy][-10,0]\ar @{--}@[ashy][6,-8] &&&&&&&&&&&&\\
&&&&&&&&&&&&&&&&&&&&&&&&&&&&&&&&&&&&&&&&&&&&&&&&&&&&&&&&&&&&\\
&&&&&&&&&&&&&\tb{F_2}\ar @{..>}[-6,5]\ar [4,5] &&&&&&&&&&&&&&&&&&&&&&&&&*+<1pt>{\b}\ar @{-}@[ashy][-2,10]\ar [-6,5]\ar[2,9] &&&&&&&&&&&&&&&&&&&&&&\\
&&&&&&&&&&&&&&&&&&&&&&&&&&&&&&&&&&&&&&&&&&&&&&&&&&&&&&&&&&&&\\
&&&&&&&&&&&&&&&&&&&&&&&&&&&&*+<1pt>{\b}\ar @{-}@[ashy][-2,10]\ar [-6,5]\ar[2,9] &&&&&&&&&&&&&&&&&&&\b\ar @{..>}[2,9]\ar [-4,1]&&&&&&&&&&&&&\\
&&&&&&&&&&&&&&&&&&&&&&&&&&&&&&&&&&&&&&&&&&&&&&&&&&&&&&&&&&&&\\
&&&&&&&&&&&&&&&&&&*+<1pt>{\b}\ar @{-}@[ashy][-2,10]\ar [-6,5]\ar[2,9] &&&&&&&&&&&&&&&&&&&\b\ar @{..>}[2,9]\ar [-4,1]&&&&&&&&&&&&&&&&&&&*+<1pt>{\c}\ar @{-}@[ashy][-6,-8]&&&&\\
&&&&&&&&&&&&&&&&&&&&&&&&&&&&&&&&&&&&&&&&&&&&&&&&&&&&&&&&&&&&\\
&&&&&&&&*+<1pt>{\tb{E_3}}\ar @{-}@[ashy][-10,0]\ar @{-}@[ashy][-2,10]\ar @{-}@[ashy][6,8] \ar [-6,5]\ar[2,9]\ar[2,1] &&&&&&&&&&&&&&&&&&&\b\ar @{..>}[2,9]\ar [-4,1] &&&&&&&&&&&&&&&&&&&*+<1pt>{\c}\ar @{-}@[ashy][-2,10]\ar @{..>}[-4,1]&&&&&&&&&&&&&&\\
&&&&&&&&&&&&&&&&&&&&&&&&&&&&&&&&&&&&&&&&&&&&&&&&&&&&&&&&&&&&\\
&&&&&&&&&{\bb{E_2}}\ar @{..>}[2,1]\ar @{-->}[-4,9] &&&&&&&&\bb{G_2}\ar @{..>}[2,9]\ar [-4,1] &&&&&&&&&&&&&&&&&&&*+<1pt>{\bcv{x_2}}\ar @{-}@[ashy][-2,10]\ar @{..>}[-4,1]&&&&&&&&&&&&&&&&&&&&&&&&\\
&&&&&&&&&&&&&&&&&&&&&&&&&&&&&&&&&&&&&&&&&&&&&&&&&&&&&&&&&&&&\\
&&&&&&&&&&*+<1pt>{\c}\ar @{--}@[ashy][-6,30]&&&&&&&&&&&&&&&&*+<1pt>{\c}\ar @{-}@[ashy][-2,10]\ar @{..>}[-4,1]&&&&&&&&&&&&&&&&&&&&&&&&&&&&&&&&&&\\
&&&&&&&&&&&&&&&&&&&&&&&&&&&&&&&&&&&&&&&&&&&&&&&&&&&&&&&&&&&&\\
*+<1pt>{\bcv{0}}\ar @{..>}[-4,9]\ar @{-}@[ashy][-2,10]\ar @{-}@[ashy][-6,8]  &&&&&&&&&&&&&&&&*+<1pt>{\c}\ar @{..>}[-4,1]\ar @{-}@[ashy][-2,10]&&&&&&&&&&&&&&&&&&&&&&&&&&&&&&&&&&&&&&&&&&&&\\}$$}
For each $\tau$-orbit of line bundles, we select one member, in bracket-symbol notation the line bundles $\Oo$, $\Oo(\vx_2)$ and $\Oo(\vx_3)$. Also, as indicated in the figure, for each of the four remaining $\tau$-orbits, we select one member, resulting in three rank-two bundles $E_2$, $F_2$ and $G_2$ and one bundle $E_3$ of rank three.
\begin{Prop}
	\label{proj_cover_2_3_3}
We assume weight type $(2,3,3)$ and refer to the notations from the above picture. Then each indecomposable bundle of rank at least two, lies in the $\tau$-orbit of exactly one of the vector bundles $E_2$, $F_2$, $G_2$ and $E_3$, where the subindex indicates the rank. The projective covers of these vector bundles are given by the table below.

\begin{tabular}{c|l}
{vector bundle}& {projective cover}\\
\hline
$E_2$& $\vec{0}$, $-\vw-\vx_1$, $-\vw-\vx_2$, $-\vw-\vx_3$\\
$F_2$& $\vx_2-\vx_3$, $-\vx_2$, $\vx_2-\vx_1$, $\vw$\\
$G_2$& $2\vx_2-2\vx_3$, $-\vx_3$, $\vw$, $\vx_3-\vx_1$\\
\hline
$E_3$& $\vw$, $2\vw$, $\vx_3+3\vw$, $\vx_3+4\vw$, $\vx_2+3\vw$, $\vx_2+4\vw$
\end{tabular}
\end{Prop}

\begin{pf}
With the generator $u:=\vx_2-\vx_3$ of the torsion group $t\LL$ of $\LL$, the extension term $H_i$ of the almost-split sequence $0\ra \Oo(i\vu)\ra H_i \ra \Oo(i\vu -\vom)\ra 0$ equals $E_2$, $F_2$ or $G_2$ for $i=0,1$ or $2$, respectively. By means of Theorem \ref{KLM_proj_cov}, the claim on their projective covers follows. It thus remains to determine the projective cover for $E_3$. We note that each distinguished exact sequence with middle term of rank three, necessarily splits. Hence for vector bundles of rank $3$ it is not possible to use the horse-shoe argument from Lemma \ref{proj_cover_from_sequence} in order to reduce the determination of projective covers to smaller rank. We thus need to give a direct argument:

Putting $F:=E_2(\vom)$, we obtain short exact sequences
$$0\to F\stackrel{i_{E_3}}{\lra} E_3 \stackrel{\pi _{E_3}}{\lra} \Oo(-\vw)\to 0,\quad 0\to \Oo(\vw)\stackrel{i_F}{\lra} F \stackrel{\pi _F}{\to} \Oo\to 0.$$
For $i=1,2,3$ there are non-zero maps $x_i^{(-\vw)}:\Oo(-\vw-\vx_i)\to \Oo(-\vw)$, $x_i:\Oo(-\vx_i)\to \Oo$. Because $\ExtX(\Oo(-\vw-\vx_i),\Oo(\vw))=0=\ExtX(\Oo(-\vw-\vx_i),\Oo)$ there are maps $y_i^{(-\vw)}:\Oo(-\vw-\vx_i)\to {E_3}$, such that $\pi_{E_3}\circ y_i^{(-\vw)}=x_i^{(-\vw)}$ for each $i=1,2,3$.
Since further $\ExtX({\Oo(-\vx_i)},{\Oo})=0$, there are maps $y_i:\Oo(-\vx_i) \to F$, such that $\pi_F \circ y_i=x_i$.
We now show that  each map $t:L\to {E_3}$, with $L$ a line bundle, factors through $\pi=(i_{E_3}\circ i_F,(i_{E_3}\circ y_i), (y_i^{(-\vw)}))_{i=1,2,3}$.

We can assume that $\pi_{E_3}\circ t:L\to \Oo(-\vw)$ is not an isomorphism.
Then $\pi_{E_3}\circ t=\sum\limits_{i=1}^3x_i^{(-\vw)}\circ t_i$, where $t_i:L\to \Oo(-\vx_i\-\vw)$. Hence $\pi_{E_3}\circ t=\sum\limits_{i=1}^3x_i^{(-\vw)}\circ t_i= \sum\limits_{i=1}^3 \pi_{E_3}\circ y_i^{(-\vw)}\circ t_i$, so $\pi_{E_3}\left(t- \sum\limits_{i=1}^3 y_i^{(-\vw)}\circ t_i\right)=0$. Therefore there is map $g:L\to F$ such that $i_{E_3}\circ g=t-\sum\limits_{i=1}^3 y_i^{(-\vw)}\circ t_i$.
Again, $\pi_F\circ g$ is not an isomorphism, so $\pi_F\circ g=\sum\limits_{i=1}^3x_i\circ g_i$ for some $g_i:L\to \Oo(-\vx_i)$. Then
$\pi_F(g-\sum\limits_{i=1}^3y_i\circ g_i)=0$. Hence there is a map $h:L\to \Oo(\vw)$, such that $i_F\circ h=g-\sum\limits_{i=1}^3y_i\circ g_i$.
Applying $i_{E_3}$ to this equality we obtain
$$t=i_{E_3}\circ i_F\circ h+\sum\limits_{i=1}^3 \left[y_i^{(-\vw)}\circ t_i+(i_{E_3}\circ y_i)\circ g_i\right].$$
Thus $t$ factors through $\pi=(i_{E_3}\circ i_F,(i_{E_3}\circ y_i), (y_i^{(-\vw)}))_{i=1,2,3}$.

Moreover, since $\HomX(\Oo(-\vx_1), \Oo(-\vw-\vx_2))$ is non-zero and the spaces $\HomX(\Oo(-\vw-\vx_2),{E_3})$ and  $\HomX(\Oo(-\vx_1),{E_3})$ are one-dimensional, we can factor $y_1:\Oo(-\vx_1)\to {E_3}$ through $y_2^{(-\vw)}:\Oo(-\vw-\vx_2)\to {E_3}$. It is easy to see that  the line bundles $\Oo(\vx_3+\vw)$, $\Oo(\vx_2+\vw)$, $\Oo(\vx_1+2\vw)$, $\Oo(\vx_3)$, $\Oo(\vx_2)$, $\Oo(\vx_1+\vw)$ are Hom-orthogonal, which implies minimality of $\pc{{E_3}}$. Alternatively, minimality can be deduced from Theorem \ref{Thm:2ab}.
\end{pf}

\subsection*{Case $(2,3,4)$}
In this case the Auslander-Reiten quiver of $\vect{\XX}$ has shape $\ZZ\widetilde\EE_7$. It contains two $\tau$-orbits of line bundles, three $\tau$-orbits of rank-two bundles, two $\tau$-orbits of rank-three bundles and one $\tau$-orbit of rank-four bundles. We chose, in addition to the structure sheaf, one member of each $\tau$-orbit, as indicated in the figure, and also mark the position of $\tau G_2$.
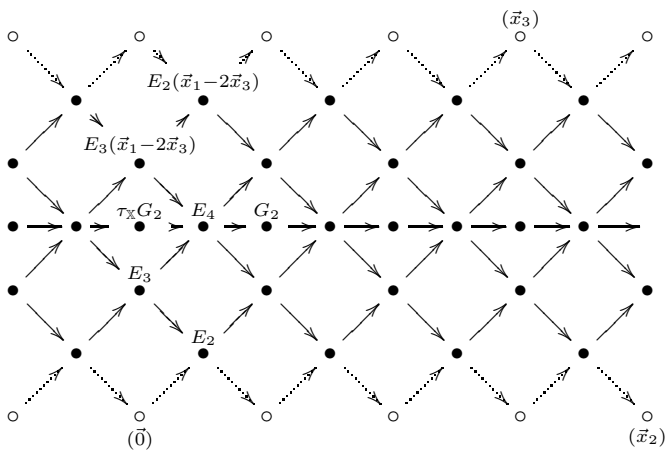
{
$$
\xymatrix @!0{
\c\ar@{..>}[rd]&&\c\ar@{..>}[rd]&&\c\ar@{..>}[rd]&&\c\ar@{..>}[rd]&&\tc{(\vx_3)}\ar@{..>}[rd]&&\c\\
&\b\ar@{..>}[ru]\ar[rd]&&\tb{E_2(\vx_1-2\vx_3)}\ar@{..>}[ru]\ar[rd]&&\b\ar@{..>}[ru]\ar[rd]&&\b\ar@{..>}[ru]\ar[rd]&&\b\ar@{..>}[ru]\ar[rd]&\\
\b\ar[ru]\ar[rd]&&\tb{E_3(\vx_1-2\vx_3)}\ar[ru]\ar[rd]&&\b\ar[ru]\ar[rd]&&\b\ar[ru]\ar[rd]&&\b\ar[ru]\ar[rd]&&\b\\
\b\ar[r] &\b\ar[r]\ar[ru]\ar[rd]& \tb{\tauX G_2}\ar[r]& \ar[r]\ar[ru] \ar[rd]\tb{E_4}&\tb{ G_2}\ar[r]& \b\ar[r]\ar[ru]\ar[rd] &\b\ar[r]& \b\ar[r]\ar[ru]\ar[rd] &\b\ar[r]&\b\ar[r]\ar[ru]\ar[rd]&\\
\b\ar[ru]\ar[rd]&&\tb{E_3}\ar[ru]\ar[rd]&&\b\ar[ru]\ar[rd]&&\b\ar[ru]\ar[rd]&&\b\ar[ru]\ar[rd]&&\b\\
&\b\ar[ru]\ar@{..>}[rd]&&\tb{E_2}\ar[ru]\ar@{..>}[rd]&&\b\ar[ru]\ar@{..>}[rd]&&\b\ar[ru]\ar@{..>}[rd]&&\b\ar[ru]\ar@{..>}[rd]&\\
\c\ar@{..>}[ru]&&\bc{(\vec{0})}\ar@{..>}[ru]&&\c\ar@{..>}[ru]&&\c\ar@{..>}[ru]&&\c\ar@{..>}[ru]&&\bc{(\vx_2)}
}.$$}
\begin{Prop}
	\label{proj_cover_2_3_4}
We assume weight type $(2,3,4)$ and refer to the notations from the picture above. Then each indecomposable vector bundle of rank at least two lies in the $\tau$-orbit of exactly one of the vector bundles
$E_i$, $i=2,3,4$, $F_j=E_j(\vx_1-2\vx_3)$ for $j=2,3$, and $\tauX G_2$, where the subindex indicates the rank. The
projective covers of these vector bundles are given by the table below.\medskip

\begin{tabular}{l|l}
{vector bundle} & {projective cover}\\
\hline
$E_2$ & $\vec{0}$, $ -\vw-\vx_1$, $ -\vw-\vx_2$, $-\vw-\vx_3$\\
$F_2$ & $\vw$, $ \vx_2-\vx_1$, $ -\vx_2$, $ \vx_2-\vx_3$\\
$\tauX G_2$ & $\vx_2-2\vx_3$, $ -\vx_2$, $ \vx_1-2\vx_2$, $ \vw-\vx_3$\\
\hline
$E_3$ & $\vw$, $ \vx_3-\vx_1$, $ -\vx_2$, $\vx_2-\vx_1$, $-\vx_3$, $\vx_2-\vx_3$\\
$F_3$ & $2\vx_3-\vx_2$, $ -\vx_3$, $ \vw-\vx_3$, $ \vx_2-2\vx_3$, $ \vx_1-3\vx_3$, $\vx_1-2\vx_2  $\\
\hline
$E_4$ & $\vx_2-2\vx_3$, $ -\vx_2$, $ \vx_2-\vx_1$, $ \vw-\vx_3$, $\vw, \vx_3-\vx_1$, $ \vx_3-\vx_2$, $ -\vx_3$\\
\end{tabular}
\end{Prop}

\begin{pf}
Since $F_i=E_i(\vx_1-2\vx_3)$ for $i=2,3$ and $\pc{E(\vx)}=\pc{E}(\vx)$, it suffices to determine the projective covers of $E_2,\ldots,E_4$ and of $\tauX G_2$. For the rank two bundles $E_2$ and $\tauX G_2$ this is an application of Theorem \ref{KLM_proj_cov}. Concerning $E_3$, the proof is similar to the proof of Proposition \ref{proj_cover_2_3_3}. It thus remains to determine the projective cover of $E_4$ by using Lemma \ref{proj_cover_from_sequence}. For this we consider the almost split sequence $0\lra \tauX G_2 \lra E_4 \lra G_2\lra 0$.
If $L$ is a line bundle summand of $\pc{G_2}$, then, by stability, $\mu L<\mu G_2$, and $\ExtX(L,\tauX G_2)=D\HomX(G_2,L)=0$. Similarly, if $L'$ is a line bundle summand of $\pc{\tauX G_2}$, then $\ExtX(G_2,L')=0$. Therefore, the above sequence satisfies the  assumptions of Lemma \ref{proj_cover_from_sequence}, such that $\pc{E_4}=\pc{\tauX G_2}\oplus \pc{G_2}$, as claimed.
\end{pf}

\subsection*{Case $(2,3,5)$}
In this case the Auslander-Reiten quiver of $\vect\XX$  has the form $\ZZ\wtilde{\EE}_8$. We have just one $\tau$-orbit of rank $r$ for $r=1,5,6$, and two $\tau$-orbits of rank $r$ for each $r=2,3,4$.
$$
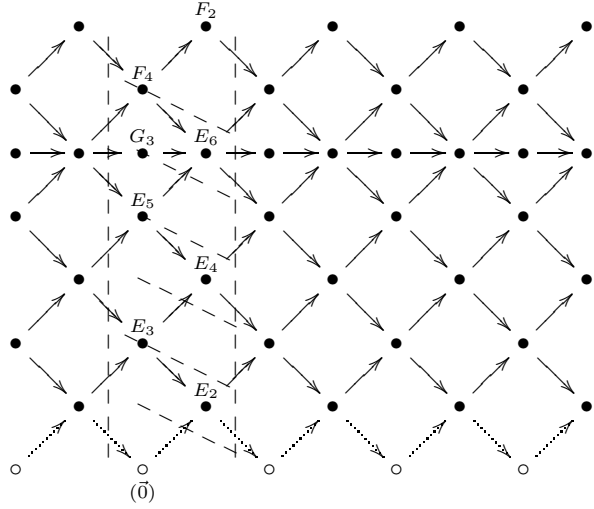
\xymatrix @!0{
&\b\ar[rd]&\ar@<-3ex>@{--}@[red][ddddddd]\ar@<-5ex>@{--}@[red][drr]&\tb{F_2}\ar[rd]&\ar@<-3ex>@{--}@[red][ddddddd]&\b\ar[rd]&&\b\ar[rd]&&\b\\
\b\ar[ru]\ar[rd]&&\ar@<-5ex>@{--}@[red][drr]\tb{F_4}\ar[ru]\ar[rd]&&\b\ar[ru]\ar[rd]&&\b\ar[ru]\ar[rd]&&\b\ar[ru]\ar[rd]&\\
\b\ar[r]&\b\ar[r]\ar[ru]\ar[rd]&\ar@<-5ex>@{--}@[red][drr]\tb{G_3}\ar[r]&\tb{E_6}\ar[r]\ar[ru]\ar[rd]&\b\ar[r]&\b\ar[r]\ar[ru]\ar[rd]&\b\ar[r]&\ar[r]\b\ar[ru]\ar[rd]&\b\ar[r]&\b\\
\b\ar[ru]\ar[rd]&&\ar@<-5ex>@{--}@[red][drr]\tb{E_5}\ar[ru]\ar[rd]&&\b\ar[ru]\ar[rd]&&\b\ar[ru]\ar[rd]&&\b\ar[ru]\ar[rd]&\\
&\b\ar[ru]\ar[rd]&\ar@<-5ex>@{--}@[red][drr]&\tb{E_4}\ar[ru]\ar[rd]&&\b\ar[ru]\ar[rd]&&\b\ar[ru]\ar[rd]&&\b\\
\b\ar[ru]\ar[rd]&&\tb{E_3}\ar@<-5ex>@{--}@[red][drr]\ar[ru]\ar[rd]&&\b\ar[ru]\ar[rd]&&\b\ar[ru]\ar[rd]&&\b\ar[ru]\ar[rd]&\\
&\b\ar@{..>}[rd]\ar[ru]&&\tb{E_2}\ar@{..>}[rd]\ar[ru]&&\b\ar@{..>}[rd]\ar[ru]&&\b\ar@{..>}[rd]\ar[ru]&&\b\\
\c\ar@{..>}[ru]&&\bc{(\vec{0})}\ar@{..>}[ru]&&\c\ar@{..>}[ru]&&\c\ar@{..>}[ru]&&\c\ar@{..>}[ru]&\\
}$$
The marked region is a fundamental domain with respect to the Auslander-Reiten translation.

\begin{Prop}
	\label{proj_cover_2_3_5}
We assume weight type $(2,3,5)$ and use the notations from the picture above.
Then each indecomposable bundle of rank $\geq 2$ lies in the Auslander-Reiten orbit of exactly one of the vector bundles $E_i$, $F_j$ and $G_l$,  having rank $i$ for $i=2,3,4,5,6$ (resp.\ for $j=2,4$, resp.\ for $l=3$). Moreover, the projective covers are given by the table below:

\begin{tabular}{c|l}
{vector bundle}& {projective cover}\\
\hline
$E_2$ & $\vec{0}$, $\vx_3-2\vx_2$, $\vx_3-\vx_1$, $\vx_2-\vx_1$\\
$F_2$ & $\vx_1-3\vx_3$, $\vw-\vx_2$, $-\vx_3$, $\vx_2-3\vx_3 $\\
\hline
$E_3$ & $\vw$, $\vx_2-\vx_1$, $-\vx_3$, $\vx_3-\vx_1$, $-\vx_2$, $\vx_3-2\vx_2 $\\
$G_3$ & $\vx_1-3\vx_3$, $\vx_2-\vx_1$, $\vw-\vx_3,\vx_2-3\vx_3,-\vx_2,-2\vx_3 $\\
\hline
$E_4$ & $\vx_2-\vx_1$, $\vw-2\vx_3$, $\vx_2-3\vx_3$, $-\vx_2$, $\vw$, $\vx_3-\vx_1$, $\vx_3-\vx_2$, $ -\vx_3$\\
$F_4$ & $\vx_3-\vx_2$, $-2\vx_3$, $-\vx_3$, $\vx_3-\vx_1$, $\vx_1-3\vx_3$, $\vw-\vx_2$, $-\vx_3$, $ \vx_2-3\vx_3$\\
\hline
$E_5$ & $\vx_3-\vx_2$, $-\vx_3$, $\vx_2-3\vx_3$, $\vx_3-\vx_1$, $\vw-\vx_2$, $\vw-2\vx_3$, $\vx_2-2\vx_3$,\\
 & $ -\vx_2$, $\vx_2-\vx_1$, $\vw-\vx_3$\\
\hline
$E_6$ & $\vx_1-3\vx_3$, $\vx_2-\vx_1$, $\vw-\vx_3$, $\vx_2-3\vx_3$, $-\vx_2$, $-2\vx_3$, $\vx_2-2\vx_3$,\\
 &$ \vx_3-\vx_2$, $-\vx_3$, $\vw-\vx_3$, $\vx_3-\vx_1$, $\vw-\vx_2$
\end{tabular}
\end{Prop}

\begin{pf}
In the case of the vector bundles of rank $2$ and $3$ the proof is similar as in Proposition \ref{proj_cover_2_3_3}. For the remaining cases we will use Lemma \ref{proj_cover_from_sequence}.

In the case $E_4$ we consider the exact sequence
$0\lra \tauX^2 F_2\lra E_4\lra \tauX^{-2}F_2\lra 0$. It is easy to see that this sequence satisfies the  assumptions of Lemma \ref{proj_cover_from_sequence}, ie. $\ExtX(\pc{\tauX^{-2}F_2}, \tauX^2 F_2)=D\HomX(\tauX F_2, \pc{F_2})=0$
and $\ExtX(\tauX^{-2}F_2,\ih{\tauX^2 F_2})=$\newline $=D\HomX(\tauX\ih{F_2},F_2)=0$. Hence
$\pc{E_4}=\pc{\tauX^2 F_2}\oplus \pc{\tauX^{-2}F_2}$, so the results follow from Theorem \ref{KLM_proj_cov}, applied to the extension bundle $\tauX^2F_2$ (resp. $\tauX^{-2}F_2$) which are determined by the pairs $(\Oo(\vx_3-\vx_2), \vx_3)$ (resp. $(\Oo,\vx_3)$).

For the vector bundles $F_4$, $E_5$ and $E_6$, we use the exact sequences $0\lra \tauX F_2\lra F_4\lra F_2\lra 0$, $0\lra\tauX G_3\lra E_5\lra \tauX^-F_2\lra 0$ and $0\lra G_3\lra E_6\lra \tauX^-G_3\lra 0$, respectively. It is straightforward to check that these satisfy the conditions of Lemma \ref{proj_cover_from_sequence}. The claim follows.
\end{pf}

\begin{Rem}
Later, when calculating  matrix factorizations for the vector bundle $E_6$, we will use two different exact sequences representing $E_6$, namely $0\lra G_3\lra E_6\lra \tauX^-G_3\lra 0$ and $0\lra \tauX F_4\lra E_6\lra \tauX^-F_2\lra 0$. While both sequences yield the same projective cover, we will obtain different shapes for the corresponding matrix factorizations, because the two procedures yield matrix factorizations with a different number of zero entries.
\end{Rem}

We conclude this section with the observation that indecomposable vector bundles are uniquely determined by their projective covers, provided $\XX$ is domestic. This result turns out to be central for determining a matrix factorization for indecomposable vector bundles.

\begin{Prop} \label{prop:determined_by_cover}
We assume a domestic weight triple $(2,a,b)$. Let $E$ and $F$ be two indecomposable vector bundles. Then $E$ and $F$ are isomorphic if and only if their projective covers $\pc{E}$ and $\pc{F}$ are isomorphic.
\end{Prop}
\begin{pf}
We may assume that $\pc{E}=\pc{F}$ where $E$ and $F$ have rank at least two.

The first part of the proof  holds for arbitrary weight triples $(2,a,b)$: From the (distinguished) exact sequence $0\to E(-\vx_1)\to \pc{E} \to E \to 0$
using \ref{detformula} we obtain  that $\det(\pc{E})=2\det(E)+\vc$ and, moreover, $\rk(\pc{E})=2\rk(E)$. Therefore $\pc{E}$ determines determinant, degree, rank and slope of $E$. In particular, $\pc{E}=\pc{F}$ implies that $E$ and $F$ have the same rank and the same slope.

Next, we establish the claim separately for the domestic weight triples $(2,2,n)$, $(2,3,3)$, $(2,3,4)$ and $(2,3,5)$.

\emph{Case $(2,2,n)$:} We refer to the notations of Proposition \ref{proj_cover_2_2_n}. By an appropriate $\tau$-shift, we may assume that the bundles $E_0$, $E_2$, $E_4$, $\ldots$ (resp. the bundles $E_1$, $E_3$, $E_5$, $\ldots$) have the same slope. In each of the two families one checks that they have distinct systems of line bundle summands in the projective cover having maximal slope.

\emph{Case $(2,3,3)$:} With the notations of Proposition \ref{proj_cover_2_3_3}, the rank-two bundles $E_2$, $F_2$ and $G_2$ have the same slope, but different line bundle summands of their projective cover with maximal slope, namely $\Oo$, $\Oo(\vx_3-\vx_2)$ and $\Oo(2(\vx_3-\vx_2))$, respectively. Since there is a unique $\tau$-orbit of indecomposable rank three bundles, each of these bundles is determined by its slope.

\emph{Case $(2,3,4)$:} We refer to the notations from Proposition \ref{proj_cover_2_3_4}.  The rank-two bundles $E_2$ and $E_2(\vx_3-2\vx_1)$ have the same half-integral slope and belong to different $\tau$-orbits. They have different line bundle summands of maximal slope in their respective projective covers, namely $\Oo$ and $\Oo(\vx_3-2\vx_1)$. The members from the third $\tau$-orbit of rank-two bundles, in particular $T$, are distinguished from members from the other two $\tau$-orbits by their slope which is integral. In a similar way, the bundles $E_3$ and $E_3(\vx_3-2\vx_1)$ have the same slope, they represent the $\tau$-orbits of rank-three bundles, and have different line bundle summands of maximal slope in their projective covers, namely $\Oo(\vom)$ and $\Oo(\vom +\vx_3-2\vx_1)$. Finally, their is just one $\tau$-orbit of indecomposable rank-four bundles.

\emph{Case $(2,3,5)$:} We refer to the notations from Proposition \ref{proj_cover_2_3_5}. Here, the claim reduces to show that $E_2$ , $F_2$ ($E_3$ , $F_3$ and $E_4$, $F_4$) can be distinguished in terms of their projective covers. To distinguish $E_2$ and $F_2$ we observe that $\Oo$ and $\Oo$, respectively $\Oo(3\vom)$ are line bundle summands of their projective covers. Concerning $E_3$ and $F_3$ the line bundle summands $\Oo$ and $\Oo(4\vom)$ have maximal slopes in their respective projective covers. Finally, the Auslander-Reiten orbits of $E_4$ and $F_4$ are distinguished by their integral (resp.\ half-integral) slopes.
\end{pf}

Assuming an arbitrary weight triple, a corresponding result holds true for indecomposable bundles of rank two. For the proof we refer to \cite{Lenzing:Ruan:2013}.

\begin{Prop} \label{prop:determined_by_cover:abc}
We assume that $\XX$ has triple weight type $(a,b,c)$. Then each indecomposable vector bundle $E$ of rank two is uniquely determined by its projective cover $\pc{E}$.\hfill\qed
\end{Prop}

\begin{Rem}
Assuming a weighted projective line $\XX=\XX(a,b,c)$ of Euler characteristic $\chi_\XX\leq0$, that is, assuming $\XX$ of tubular or wild type, it is no longer true that each indecomposable vector bundle $E$ is uniquely determined by its projective cover $\pc{E}$. Let's assume that the base field $k$ is uncountable and, for simplicity, further that $\XX$ has weight type $(2,a,b)$. By perpendicular calculus, see \cite{Geigle:Lenzing:1991}, there exists a weighted projective line $\YY$ of tubular type and a full embedding $\coh\YY\hookrightarrow \coh\XX$ that preserves the rank. From the tubular families in $\coh\YY$ we then deduce the existence of a one-parameter family $(E_\alpha)$ of indecomposable vector bundles over $\XX$, all having the same positive rank $r$. This, in turn, implies that each projective hull $\pc{E_\alpha}$ has fixed rank $2r$. Since the grading group  $\LL$ is countable, this leaves only countably many possibilities for the isomorphism classes of $\pc{E_\alpha}$, forcing many non-isomorphic $E_\alpha$'s to have the same projective cover.

This leads to a modified question, where the authors do not know the answer. \emph{Assume that $E$ and $F$ are exceptional vector bundles with isomorphic projective covers $\pc{E}$ and $\pc{F}$. Does this imply that $E$ and $F$ are isomorphic?} In support for a positive answer, we mention that, for $\XX$ domestic,  each indecomposable vector bundle is exceptional. Also, because of triple weight type, all extension bundles are exceptional. Moreover since, by a result of Hübner \cite{Huebner:1996}, see \cite{Meltzer:2004} for a proof, exceptional vector bundles are determined by their classes in the Grothendieck group $\Knull\XX$, there exist only countably many isoclasses of exceptional vector bundles, thus preventing the contradiction of the above argument.
\end{Rem}

\section{Factorization frame attached to a vector bundle}

We assume a domestic weight triple $(2,a,b)=(p_1,p_2,p_3)$. Then a matrix factorization for an indecomposable vector $E$ bundle of rank $r$ can be obtained from its minimal projective resolution by first determining its \emph{factorization frame} consisting of a pair of $2r\times 2r$-matrices, see Definition \ref{def_matrix_frame}, and then adjusting the entries of the factorization frame by suitable scalars.

We now establish the key result of this section. For this, it is convenient to identify the Frobenius categories $\CMgr\LL{S}$ and $\vect\XX$ by means of sheafification $\CMgr\LL{S}\stackrel{\tilde{}}{\lra}\vect\XX$, $M\mapsto \tilde{M}$. In the same context, by a matrix factorization $(u,v)$, attached to a vector bundle $E$ without line bundle summands, we mean a matrix factorization of $f=x_1^{p_1}+x_2^{p_2}+x_3^{p_3}$ over $T=k[x_1,x_2,x_3]$ attached to the Cohen-Macaulay module $M$ corresponding to $E$. We are going to compare the rows of the commutative diagram

\begin{equation} \label{eqn:matrixfact:projresI}
\xymatrix{
\cdots\ar[r]& P_0(-\vc)\ar[r]^{v}\ar[d]_\nu & P_1\ar[d]_\nu \ar[r]^{u} & P_0\ar[d]_\nu \ar[r]^\pi& M\ar@{=}[d] \ar[r]& 0 \\
\cdots\ar[r]& \bar{P}_0(-\vc)\ar[r]^{\bar{v}} & \bar{P}_1 \ar[r]^{\bar{u}} & \bar{P}_0 \ar[r]^{\bar{\pi}}& M \ar[r]& 0 \\}
\end{equation}
where the upper row is a $T$-matrix factorization of $f$ for $M$, $\pi$ is a $T$-projective cover of $M$, and where the lower row is a minimal $S$-projective resolution of $M$. The vertical maps, and the bar notation stand for the reduction modulo $(f)$.

In the above setting, assume we fix decompositions of $P_0$ and $P_1$ into indecomposable $T$-projectives and consider the decompositions for the $T$-projectives $P_0(-n\vc)$, induced by degree shift, and the corresponding decompositions of the $S$-projectives $\bar{P}_0(-n\vc)$ and $\bar{P}_1(-n\vc)$, induced by reduction modulo $(f)$. We then say that we have chosen \emph{compatible decompositions} for \eqref{eqn:matrixfact:projresI}. To achieve such compatible decompositions, we may alternatively start with decompositions of $\bar{P}_0$ and $\bar{P_1}$, and then extend them to the remaining members of \eqref{eqn:matrixfact:projresI} by degree shift and by taking $T$-projective covers.

We say that an element $x=x_1^{l_1}x_2^{l_2}x_3^{l_3}$, viewed as a member of $T$ or $S$, is a \emph{monomial with small exponents} if $0\leq l_i\leq p_i-1$ holds for $i=1,2,3$ and, moreover, $\sum_{i=1}^3 l_i>0$. In particular, $x$ belongs to the graded maximal ideal of $T$ (respectively $S$).

\begin{Thm} \label{thm:hom_dimensionsI}
We assume the above setting \eqref{eqn:matrixfact:projresI} for a weighted projective line $\XX$ of domestic type, where $M$ is an indecomposable $\LL$-graded Cohen-Macaulay module of rank at least two, attached to the (indecomposable) vector bundle  $E$ on $\XX$.

Let $\vy_0, \vy_1$ be members of $\LL$ such that $T(\vy_i)$ is a $T$-direct summand of $P_i$ ($i=0,1$), and, accordingly, $S(\vy_i)$ is an $S$-direct summand of $\bar{P}_i$. Then reduction modulo $(f)$ induces an isomorphism
\begin{equation} \label{eqn:iso:homspacesI}
T_{\vy_0-\vy_1}=\Hom_T(T(\vy_1),T(\vy_0))\stackrel{\cong}\lra \Hom_S(S(\vy_1),S(\vy_0))=S_{\vy_0-\vy_1}.
\end{equation}
Moreover, the Hom-spaces from \ref{eqn:iso:homspacesI} are either both zero, and then $\vy_1\not\leq \vy_0$ or else $0 < \delta(\vy_0-\vy_1)< \delta(\vc)$, and then $T_{\vy_0-\vy_1}=kx$ (and also $S_{\vy_0-\vy_1}=kx$) for a monomial $x$ with small exponents.
\end{Thm}

\begin{pf}
Since $f$ belongs to the graded maximal ideal of $S$, a finitely generated graded $T$-module $M$ is zero if and only if $M/fM$ is zero. In particular, $\Hom_T(T(\vy_1),T(\vy_0))=0$ if and only if $\Hom_S(S(\vy_1),S(\vy_0))=0$.

Next, we switch to the context of vector bundles, and use the existence of (distinguished) exact sequences
$0\lra E(-\vx_1)\lra \bar{P}_0\lra E \lra 0$ and $ 0\lra E(-\vc)\lra \bar{P}_1\lra E(-\vx_1)\lra 0$, where the $\bar{P_0}$ and $\bar{P_1}$ are projective in the Frobenius category $\vect\XX$. Assume that $\HomX(\Oo(\vy_1),\Oo(\vy_0))\neq 0$. We are going to show that $\HomX(\Oo(\vy_1),\Oo(\vy_2))=kx$, where $x$ is a small monomial: If $E$ has slope $\mu E=q$, then $\mu E(-\vx_1)=q-\delta(\vx_1)$ and $\mu E(-\vc)=q-\delta(\vc)$. Because $\XX$ is domestic, all indecomposable vector bundles are stable by \cite[Proposition 5.5]{Geigle:Lenzing:1987}. Since $\bar{P}_1$ is the injective hull of $E(-\vc)$, we have $\HomX(E(-\vc),\Oo(\vy_1))\neq0$, hence $q-\delta(\vc)<\delta(\vy_1)$. Since $\bar{P}_1$ is the projective cover of $E(-\vx_1)$, we have $\HomX(\Oo(\vy_1),E(-\vx_1))\neq0$, hence $\delta(\vy_1) < q-\delta(\vx_1)$. Similarly, using that $\bar{P}_0$ is the injective hull of $E(-\vx_1)$ and also the projective cover of $E$, we obtain the inequalities $q-\delta(\vx_1)<\delta(\vy_0) < q$. Putting things together, we finally get
$$q-\delta(\vc)<\delta (\vy_1)<q-\delta(\vx_1)< \delta(\vy_0) <q,$$
and, in particular, $\delta(\vy_0-\vy_1)<\delta(\vc)$. Since $\HomX(\Oo(\vy_1),\Oo(\vy_0))=S_{\vy_1-\vy_0}$ we get $0<\delta(\vy_1-\vy_0)<\delta(\vc)$. We put $\vu=\vy_0-\vy_1$. Since $S_{\vy_0-\vy_1}\neq 0$ by assumption, then the inequalities $0<\vec{u}$ and $0<\delta(\vec{u})<\delta(\vc)$ follow.
Note that $\vu=0$ is not possible, because $\delta(\vu)>0$.
 Writing $\vy_0-\vy_1$ in normal form $\vy_0-\vy_1=l_1\vx_1+l_2\vx_2+l_3\vx_3+l\vc$ with $0\leq l_i \leq p_i-1$, we obtain $l\geq 0$ from $\vy_0-\vy_1>0$. Assuming that $l\geq 1$, then yields $\delta(\vy_0-\vy_1)\geq \delta(\vc)$, which is impossible. Thus $l=0$, and $S_{\vy_0-\vy_1}=k\vx_1^{l_1}\vx_2^{l_2}\vx_3^{l_3}$, establishing the last assertion. From this it finally follows that the map from \eqref{eqn:iso:homspaces} is an isomorphism.
\end{pf}

For the Corollary below, we keep the notations and assumptions of Theorem \ref{thm:hom_dimensionsI}.
\begin{Cor} \label{cor:same_matrices}
We assume compatible decompositions for the members of Theorem \eqref{thm:hom_dimensionsI}. Then the $T$-matrix factorization $(u,v)$ of $f$, associated to $E$, and the $S$-minimal projective resolution $(\bar{u},\bar{v})$ are represented by the `same' matrix pair $(U,V)$, whose entries are scalar multiples of monomials with small exponents, interpreted as elements of $T$ (respectively of $S$).\hfill\qed
\end{Cor}
The scalar factors from Corollary \ref{cor:same_matrices}, and hence the matrices $(U,V)$, are usually difficult to determine. We thus introduce an intermediate concept, called a \emph{factorization frame} for $E$. By Theorem \ref{thm:hom_dimensionsI} or Corollary \ref{cor:same_matrices}, factorization frames always exist for indecomposable bundles, provided we deal with domestic weight triples. But, factorization frames may also exist in other situations. In particular, extension bundles admit factorization frames, without any restriction on the weight triple $(p_1,p_2,p_3)$.

Assuming an arbitrary weight triple $(a,b,c)$, Theorem \ref{KLM_proj_cov} provides us with an explicit projective cover for extension bundles, and thus for indecomposable vector bundles of rank two. Hence a result very close in content to Theorem~\ref{thm:hom_dimensionsI} can be shown for indecomposable Cohen-Macaulay modules of rank two for then arbitrary weight triples $(a,b,c)$, by just following the lines of the proof for Theorem~\ref{thm:hom_dimensionsI}:

\begin{Thm} \label{thm:hom_dimensions}
We assume the above setting \eqref{eqn:matrixfact:projresI} for a weighted projective line $\XX$ of type $(a,b,c)$, where $M$ is an indecomposable $\LL$-graded Cohen-Macaulay module of rank two.

Let $\vy_0, \vy_1$ be members of $\LL$ such that $T(\vy_i)$ is a $T$-direct summand of $P_i$ ($i=0,1$), and, accordingly, $S(\vy_i)$ is an $S$-direct summand of $\bar{P}_i$. Then reduction modulo $(f)$ induces an isomorphism
\begin{equation} \label{eqn:iso:homspaces}
T_{\vy_0-\vy_1}=\Hom_T(T(\vy_1),T(\vy_0))\stackrel{\cong}\lra \Hom_S(S(\vy_1),S(\vy_0))=S_{\vy_0-\vy_1}.
\end{equation}
Moreover, the Hom-spaces from \ref{eqn:iso:homspaces} are either both zero, in which case $\vy_1\not\leq \vy_0$ or else $0 < \delta(\vy_0-\vy_1)< \delta(\vc)$, in which case $T_{\vy_0-\vy_1}=kx$ (and also $S_{\vy_0-\vy_1}=kx$) for a monomial $x$ with small exponents. \hfill\qed
\end{Thm}

\begin{Defi}
	\label{def_matrix_frame}
We assume a domestic weight triple. \emph{A factorization frame} $(U,V)$ for an indecomposable vector bundle $E$ of rank $r\geq 2$, is a pair of $2r\times 2r$-matrices over $T$, obtained by Theorem \ref{thm:hom_dimensions} from line bundle decompositions
\begin{equation} \label{eqn:proj_resolution}
\bar{P}_0=\bigoplus\limits_{j=1}^{2r}\Oo(\vec{z}_j),\quad \bar{P}_1=\bigoplus\limits_{i=1}^{2r}\Oo(\vec{y}_i), \quad \bar{P}_0(-\vc)=\bigoplus\limits_{j=1}^{2r}\Oo(\vec{z}_j-\vc)
\end{equation}
of the projectives from a minimal projective resolution of $E$, as follows: The $(i,j)-$ entry of $U$ is defined as follows:

\begin{itemize}
\item[(a)] If $\HomX(\Oo(\vy_i), \Oo(\vec{z}_j))=0$, then the entry is zero.
 \item[(b)] Otherwise, $\HomX(\Oo(\vy_i), \Oo(\vec{z}_j))=S_{\vec{z}_j-\vy_i}=kx_1^{l_1}x_2^{l_2}x_3^{l_3}$ with $0\leq l_i\leq (p_i-1)$. Then the $(i,j)$-entry is given by the monomial  $x_1^{l_1}x_2^{l_2}x_3^{l_3}$.
\end{itemize}
The matrix $V$ is defined in a similar fashion from the decompositions of $P_0(-\vc)$ and $P_1$.
\end{Defi}
By Definition \ref{def_matrix_frame}, a factorization frame for $E$ depends on the line bundle decompositions for the terms of a minimal projective resolution of $E$. However, assuming domestic type, for $\rk E\leq 5$, the factorization frame attached to $E$ is unique by Proposition \ref{prop:multiplicity:line_bundle_summands}.
We note also that usually, a factorization frame $(U,V)$ for $E$ will not satisfy the matrix factorization property $UV=f\,\id=VU$. However, by Corollary \ref{cor:same_matrices}, each factorization frame for $E$ can be specialized to a matrix factorization for $E$:
\begin{Lem} \label{lem:specialization}
We assume domestic type and assume $E$ to be indecomposable of rank at least two. Then
each factorization frame $(U,V)$ for $E$ can be specialized to a $T$-matrix factorization for $f$, representing $E$, by modifying the entries of the factorization frame by (possibly zero) scalars in such a way that the resulting matrices $u$, $v$ satisfy $uv=f\,\id=vu$. Conversely, each matrix factorization $(u,v)$ for $f$, representing $E$, arises this way.\hfill\qed
\end{Lem}

Let $E$ be an indecomposable vector bundle of rank $\geq2$ for domestic weight type. Our next result implies that, with the single exception of the members $E$ from the single $\tau$-orbit of indecomposable rank-six bundles for weight type $(2,3,5)$, the decomposition of the projective cover $\pc{E}$ into line bundles is multiplicity-free.

\begin{Prop} \label{prop:multiplicity:line_bundle_summands}
Let $\XX$ be of domestic type and let $E$ be an indecomposable vector bundle of rank $r\geq 2$, with the projective cover $\pc{E}=\bigoplus\limits_{i=1}^{2r}L_i$.
If $ \rk E\leq 5$, then $L_1,\dots, L_{2r}$ are pairwise non-isomorphic line bundles. If $\rk E=6$, and then $\XX$ of weight type $(2,3,5)$, there are exactly $11$ non-isomorphic line bundle among $L_1,\dots, L_{12}$.
\end{Prop}
\begin{pf}
Case-by-case analysis, based on Propositions \ref{proj_cover_2_2_n} to \ref{proj_cover_2_3_5}.
\end{pf}

Of course, not every matrix factorization $(u,v)$, obtained from a factorization frame $(U,V)$, attached to an indecomposable vector bundle $E$ by specialization, will satisfy $\cok{u,v}=E$, see Example \ref{expl:decompose}, where the matrices $u$ and $v$ contain too many zero entries.
We thus investigate situations when the modified entries of a factorization frame are not allowed to be zero.
For this the following general  result from  \cite[Prop. 3.8]{Kussin:Lenzing:Meltzer:2013adv} will be useful. Here, we say that a map $h:E\to L'$ is a \emph{component map} of the injective hull $j_E:E\to \ih E$, if $L'$ is a line bundle and $E=L'\oplus E'$ such that $h$ is the restriction of $j_E$ to the summand $L'$. Component maps of projective covers are defined in a similar way.

\begin{Prop}
\label{fact}
Let $\XX$ be a weighted projective line of triple weight type.
Let $E$ be a vector bundle and $L$, $L'$ be line bundles. Then the  following properties
hold.

\begin{itemize} \label{prop:component.maps}
\item[(i)]
 If $h:E\to L'$ is a component map of the injective hull $j_E:E\to \ih E$, then $h$ is an epimorphism in $\coh\XX$.

\item[(ii)] If $l : L\to E$ is a component map of the projective cover $\pi_E : \pc E\to E$,
 then $l$ is a monomorphism in $\coh\XX$ and moreover the cokernel of $l$, formed in $\coh\XX$, is a vector bundle.
\end{itemize}
\end{Prop}

\begin{pf}
Property (i) immediately follows from the explicit description of the injective hull by Theorem \ref{KLM_proj_cov}, while the proof of property (ii) uses the case-by-case description of projective covers in Section \ref{sec_proj_covers}.
\end{pf}

As an immediate consequence we get.

\begin{Cor} \label{fact2}
Let $\XX$ be a weighted projective line of triple weight type.
For an exceptional vector bundle $E$ we have

\begin{itemize}
\item[(i)] $\ExtX(E, \ih E)=0$,
\item[(ii)] $\ExtX(\pc E,  E)=0$,
\item[(iii)] $\ExtX(\pc E, \ih E)=0$.
\end{itemize}
\end{Cor}

\begin{pf} Concerning (i) let $L'$ be a line bundle summand of $\ih E$ and $v: E \ra L'$ be a corresponding component map of the injective hull $j_E: E\to \ih{E}$. By Proposition \ref{prop:component.maps} then $v:E\to L'$ is an epimorphism. Thus, by heredity of $\coh\XX$, the condition $\ExtX(E,E)=0$ implies that $\ExtX(E,L')=0$. This happens for each line bundle summand $L'$ of $\ih{E}$, therefore  $\ExtX(E, \ih E)=0$. The proof of (ii) is dual. Concerning property (iii), we argue as before: Since $v$ is an epimorphism,  property (ii) implies $\ExtX(\pc{E},L')=0$ for each line bundle summand $L'$ of $\ih{E}$, and property (iii) follows.
\end{pf}

\begin{Lem}
 \label{lem_composition_is_non_zero}
Let $\XX$ be a weighted projective line of triple weight type and let $E$ be an exceptional vector bundle.  Let $\pi_E : \pc E\to E$ (respectively\ $j_E:E\to \ih{E}$) be the projective cover (respectively the injective hull) of $E$. Let $L$ (respectively\ $L'$) be a line bundle summand of the projective cover $\pc{E}$ (respectively the injective hull $\ih{E}$), and let $u: L\to E$ (respectively\ $v:E\to L'$) be corresponding component maps of $\pi_E : \pc E\to E$ (respectively $j_E:E\to \ih{E}$).

We assume that $\HomX(L,E)=k$ and $\HomX(E,L')=k$. Then the composition $v\,u$ is the zero map if and only if $ \HomX(L,L')$ equals zero.
\end{Lem}

\begin{pf}
Assume, for contradiction, that  $\HomX(L,L')\neq 0$ but $v\,u =0$. By Proposition~\ref{prop:component.maps} the map $v:E\to L'$ is an epimorphism. By heredity of $\coh\XX$, the condition $\ExtX(E,E) =0$ then implies that $\ExtX(E,L') =0$.

Next, we apply  the functor $\HomX(-,L')$ to the non-split exact sequence $ (\star) \quad 0\to L \stackrel{u}{\lra} E\stackrel{p}{\lra}F\to 0$, yielding exactness of the sequence $\HomX(E,L')\stackrel{-\circ u}{\to}\HomX(L,L') \to \ExtX(F,L') \to \ExtX(E,L')=0$. Because $\HomX(E,L')=k$ and  $v \circ u =0$, the map $-\circ u:\HomX(E,L')\to\HomX(L,L')$ is zero. Therefore $\ExtX(F,L')\cong\HomX(L,L')$ is non-zero by assumption. On the other hand, applying the functor $\HomX(-,E)$ to $(\star )$, we get exactness of $ 0 \to \HomX(F,E) \to \HomX(E,E) \to \HomX(L,E) \to \ExtX(F,E) \to \ExtX(E,E)=0$. Since, by assumption, $\HomX(E,E)=k=\HomX(L,E)$ we obtain that $\dim\HomX(F,E)=\dim\ExtX(F,E)$, and we are going to show that both terms vanish.

Assuming $\HomX(F,E)\neq0$, we compose the epimorphism $p$ with a non-zero map from $F$ to $E$ and thus obtain an endomorphism of $E$, that is neither zero nor an isomorphism, in obvious contradiction to $\Hom(E,E)=k$. Hence $\HomX(F,E)=0$ and then $\ExtX(F,E)=0$.  Because $v:E\to L'$ is an epimorphism, the condition $\ExtX(F,E)=0$ finally implies that $\ExtX(F,L')=0$, contrary to what was established before.
\end{pf}

The following consequence yields certain limitations for specializing factorization frames to matrix factorizations. We adhere to the notations of Definition \ref{def_matrix_frame} and further refer to Lemma \ref{lem:specialization}.

\begin{Cor}
	\label{cor_matrix_frame_dim_1}
Assume a matrix factorization $(u,v)$ of $f$ over $T$ that is attached to $E$, is obtained from a factorization frame $(U,V)$ for $E$ by specialization. Assume, in particular, that $u_{ij}=\la_{ij}U_{ij}$. If $\Hom(\Oo(\vy_i),E)=k=\Hom(E,\Oo(\vz_j))$ then the scalar $\la_{ij}$ must be non-zero. A similar result holds for the specialization of $V$ to $v$.
\end{Cor}

In the general situation the dimension of the homomorphism space between $E$ and a line bundle summand of $\pc{E}$ or $\ih{E}$ can by greater that one. We have more precise information for indecomposable vector bundles of rank $2$ and $3$:

\begin{Lem}
	\label{observation_dim_1} The following assertions hold.
\begin{itemize}
\item[(a)] Assuming $\XX$ of arbitrary weight triple, let $L$, (respectively $L'$) be a direct summand of the projective cover (respectively the injective hull) of an extension bundle $E$. Then $\HomX(L,E)=k$, (respectively $\HomX(E,L')=k$).
\item[(b)] Assuming $\XX$ of domestic weight type, let $L$, (respectively $L'$) be a direct summand of the projective cover (respectively injective hull) of an indecomposable rank $3$ bundle $E$. Then $\HomX(L,E)=k$, (respectively $\HomX(E,L')=k$).
\end{itemize}
\end{Lem}

\begin{pf}
We prove statement $(a)$ the proof of statement $(b)$ is similar. Let $E$ be an extension bundle on $\XX$, thus $E$ is the middle term of an exact sequence $0\lra \bar{L}(\vw)\lra E\lra \bar{L}(\vx)\lra 0$, for some line bundle $\bar{L}$ and an element $\vx=l_1\vx_1+l_2\vx_2+l_3\vx_3$ such that $0\leq l_i\leq p_i-2$ for $i=1,2,3$. It follows from Theorem \ref{KLM_proj_cov} that each direct summand $L$ of $\pc{E}$ has one of the following shapes $\bar{L}(\vw)$ or $\bar{L}(\vx-(1+l_i\vx_i))$ for $i=1,2,3$. Therefore $\ExtX(L',\bar{L}(\vw))=\HomX(\bar{L},L')=0$. Since the direct summands of the projective cover of $E$ are $\HomX$-orthogonal, we conclude that  $\dim\HomX(L',E)$ is  equal to $\dim\HomX(L',\bar{L}(\vw))=1$ for $L=\bar{L}(\vw)$ or $\dim\HomX(L',\bar{L}(\vx))=1$ for $L\neq\bar{L}(\vw)$.
\end{pf}

As the following example shows, achieving the condition $uv=f\,\id=vu$ by specialization of a factorization frame for $E$, is not sufficient for obtaining a matrix factorization, representing $E$.

\begin{Ex} \label{expl:decompose}
Let $\XX$ be a weighted projective line of type $(2,3,4)$. We consider the almost split sequence $0\lra\tau G_2\lra E_4\lra G_2\lra 0$ from section \ref{sec_proj_covers}, Proposition \ref{proj_cover_2_3_4}. This sequence satisfies the assumptions of Lemma \ref{matrix_frame_from_sequence}. Hence the factorization frame for $E_4$ has the shape
$${U}_{E_4}=\left[\begin{array}{c|c}
\overline{u}_{\tau G_2}& \overline{b}\\
0& \overline{u}_{G_2}
\end{array}\right],
\quad
{V}_{E_4}=\left[\begin{array}{c|c}
\overline{v}_{\tau G_2}& \overline{b}\\
0& \overline{v}_{G_2}
\end{array}\right].$$
If we chose scalars, such that $b=0$, and ${u}_{\tau G_2}$, ${v}_{\tau G_2}$
(resp. $u_{G_2}$, $v_{G_2}$) are matrix factorizations for $\tau G_2$ (respectively $G_2$), then we obtain a matrix factorizations for $\tau G_2\oplus G_2$, not for $E_4$.
\end{Ex}

Checking the indecomposability of a matrix factorization, obtained by specializing matrix frames for bundles of rank two and three, the following observation will be helpful.

\begin{Ob}
Let $\XX$ has a domestic type and $E$ be an indecomposable vector bundle of rank $2$ or $3$. Then the line bundle summands of $\pc E$ are pairwise Hom-orthogonal.
\end{Ob}

\begin{pf}
For rank two this is a general fact, see Theorem \ref{KLM_proj_cov}. For rank three, this follows by inspection of the projective covers for the domestic weight triples.
\end{pf}

\section{Matrix factorizations} \label{sect:matrix.factorizations}
\emph{During the whole section, we will freely switch from the notation $(x_1,x_2,x_3)$ to $(x,y,z)$, whenever this is preferable for typographical reasons.}

This section presents explicit matrix factorizations for the following cases. Keeping the assumption to deal with a weight triple $(a,b,c)$ of integers greater or equal $2$, we first present a general result on the matrix factorizations of the $\LL$-graded triangle singularity $f=x^a+y^b+z^c$ for indecomposable bundles of rank two.
Next, we restrict to weight triple $(2,a,b)$, where we obtain \emph{symmetric} matrix factorizations for indecomposable vector bundles of rank two.

For the second part of the section we restrict to weight triples of domestic case, that is, to the weight triples $(2,2,n)$, $(2,3,3)$, $(2,3,4)$, and $(2,3,5)$. Here, we determine \emph{symmetric} matrix factorizations for each indecomposable vector bundle (equivalently, each $\LL$-graded indecomposable Cohen-Macaulay module) of rank $\geq 2$, where for rank $2$ we use the general result for extension bundles. We emphasize that our methods work over any characteristic, yielding matrices whose entries are scalar multiples of monomials with small exponents, where the scalars are taken from $\{0,\pm 1\}$.
Still restricting to domestic weight types, we note that the concepts of simple singularities and triangle singularities, as studied in \cite{KST-1}, \cite{Lenzing:Pena:2011}, respectively \cite{Kussin:Lenzing:Meltzer:2013adv}, agree exactly for weight type $(2,3,5)$. Due to the different approaches, the resulting matrix factorizations for $f=x^2+y^3+z^5$ are different.

\subsection*{General results} \label{ssect:extension.bundles}

We recall that the group $\LL$ acts on $\coh\XX$, and related mathematical objects like $T=k[x_1,x_2,x_3]$, $S=T/(f)$, $\CMgr\LL{S}$ and $\vect\XX$, by degree shift. The next observation largely simplifies the determination of explicit matrix factorizations.

\begin{Lem}\label{lem:degree.shift} We assume that $\XX$ has triple weight type.
Let $E$ be a vector bundle, admitting a matrix factorization $(u,v)$. Then for each $\vx$ in $\LL$, also $E(\vx)$ admits the matrix factorization $(u,v)$. In particular, all members of a $\tau$-orbit in $\vect\XX$ admit matrix factorizations by the same matrices.
\end{Lem}
\begin{proof}
If $M$ is the $\LL$-graded Cohen-Macaulay $S$-module corresponding to $E$, and $P_0\up{u}P_1\up{v}P_0\to M \to 0$ is the start of a $2$-periodic minimal projective $A$-resolution for $M$, then application of the degree shift with $\vx$ yields the start $P_0(\vx)\up{u}P_1(\vx)\up{v}P_0(\vx)\to M(\vx) \to 0$ of a $2$-periodic minimal projective $A$-resolution for $M(\vx)$, just keeping the matrices $u$ and $v$.
\end{proof}

Assume that $\XX$ is a weighted projective line of triple weight type $(a,b,c)$, not necessarily domestic.
We recall from \cite[Theorem~4.2]{Kussin:Lenzing:Meltzer:2013adv} that each indecomposable vector bundle $E$ of rank two is an extension bundle, that is, it is the middle term $\extb{L}{\vx}$ of `the' non-split exact sequence $0\to L(\vom)\to E\to L(\vx)\to 0$ for some line bundle $L$ and some $\vnull\leq\vx\leq\vdom$, where $\vdom=\vc+2\vw$.  By Lemma~\ref{lem:degree.shift} we obtain matrix factorizations $(u_\vx,v_\vx)$ for $\extb{L}{\vx}$ where the matrices $u_\vx$ and $v_\vx$ are independent of $L$. To see this, one uses formula \eqref{eqn:ext_bundle:shift}. In the following, we are going to construct matrix factorizations for many vector bundles $E$ (or Cohen-Macaulay modules $M$). In order to describe such a matrix factorization uniquely, we list the projective cover $\pc{E}$ of $E$ (or $M$) together with the matrix pair $(u,v)$ of the factorization. Representing $\pc{E}$ as a direct sum of line bundles $\Oo(\vy_j)$, $j=1,\ldots,s$, then the triple notation $(u,v,\pc{E})$ or, equivalently, $(u,v,(\vy_j))$ determines $E$ (or $M$) uniquely, up to isomorphism.

\begin{Prop} \label{prop:mf:abc:rk2}
Let $\XX$ be a weighted projective line of triple weight type $(a,b,c)$. Then the extension bundle  $E=\extb{L}{\vx}$, where $\vx=\sum_{i=1}^3 l_i\vx_i$ and $\vnull\leq \vx\leq \vdom$, admits the matrix factorization $(u_\vx,v_\vx, (\vw, \vx-(1+l_i\vx_i))_{i=1,2,3} )$,
where
$$u_{\vx}=\left[\begin{array}{cccc}
0& x^{(1+l_1)}& y^{(1+l_2)}& z^{(1+l_3)}\\
x^{(1+l_1)}&0&z^{c-(1+l_3)}&-y^{b-(1+l_2)}\\
y^{(1+l_2)}&-z^{c-(1+l_3)}&0&x^{a-(1+l_1)}\\
z^{(1+l_3)}&y^{b-(1+l_2)} &-x^{a-(1+l_1)}&0
\end{array}\right] $$
$$v_{\vx}=\left[\begin{array}{cccc}
0& x^{a-(1+l_1)}& y^{b-(1+l_2)}& z^{c-(1+l_3)}\\
x^{a-(1+l_1)}&0&-z^{(1+l_3)}&y^{(1+l_2)}\\
y^{b-(1+l_2)}&z^{(1+l_3)}&0&-x^{(1+l_1)}\\
z^{c-(1+l_3)}&-y^{(1+l_2)} &x^{(1+l_1)}&0
\end{array}\right]. $$
\end{Prop}

\begin{pf}
From Theorem \ref{KLM_proj_cov}, the projective cover of the extension bundle $\extb{L}{\vx}$ is given by
$\pc{\extb{L}{\vx}}= L(\vw)\oplus \bigoplus_{i=1}^3 L(\vx-(1+l_i)\vx_i).$

The corresponding factorization frame for $\extb{L}{\vx}$ has the shape
$$U_{\vx}=\left[\begin{array}{cccc}
0& x^{(1+l_1)}& y^{(1+l_2)}& z^{(1+l_3)}\\
x^{(1+l_1)}&0&z^{c-(1+l_3)}&y^{b-(1+l_2)}\\
y^{(1+l_2)}&z^{c-(1+l_3)}&0&x^{a-(1+l_1)}\\
z^{(1+l_3)}&y^{b-(1+l_2)} &x^{a-(1+l_1)}&0
\end{array}\right] $$
$$V_{\vx}=\left[\begin{array}{cccc}
0& x^{a-(1+l_1)}& y^{b-(1+l_2)}& z^{c-(1+l_3)}\\
x^{a-(1+l_1)}&0&z^{(1+l_3)}&y^{(1+l_2)}\\
y^{b-(1+l_2)}&z^{(1+l_3)}&0&x^{(1+l_1)}\\
z^{c-(1+l_3)}&y^{(1+l_2)} &x^{(1+l_1)}&0
\end{array}\right].$$
From  Corollary \ref{cor_matrix_frame_dim_1} and Observation  \ref{observation_dim_1} we need to chose non-zero scalars, such that $u_{\vx}v_{\vx}=f\,\id=v_{\vx}u_{\vx}$ and $(u_{\vx},v_{\vx})$ is indecomposable.
By Theorem \ref{KLM_proj_cov} the line bundle summands of the projective cover of $\extb{L}{\vx}$  are Hom-orthogonal, it is then easy to check that the above choice of scalars yields an indecomposable matrix factorization. That we get, indeed, a matrix factorization attached to $E$, then follows from Proposition \ref{prop:determined_by_cover}.
\end{pf}

We next assume weight type $(2,a,b)$ and show that each extension bundle $E=\extb{L}{\vx}$ admits a  \emph{symmetric} matrix factorization.  We recall from Theorem \ref{Thm:2ab}, that the suspension functor $[1]$ for $\svect{\XX}$ is induced by the degree shift $E\mapsto E(\vx_1)$. The minimal projective resolution of $E$ has the form
\begin{equation} \label{eqn:resolution:symmetric}
\pc{E}(-\vc)\stackrel{\bar{v}=\bar{u}(-\vx_1)}{\lra} \pc{E}(-\vx_1)\stackrel{\bar{u}}{\lra}\pc{E}\lra E\lra 0.
\end{equation}
Interpreting the above sequence on the level of $S$-modules, this shows already that the symmetry condition $\bar{v}=\bar{u}(-\vx_1)$ is satisfied over $S=T/(f)$. It remains to lift the maps $\bar{u}$ and $\bar{v}$ to $T$-linear maps $u$ and $v$ such that $v=u(-\vx_1)$ holds and, moreover, the matrix factorization identity $u(-\vx_1)u=f\,\id$ holds.

\begin{Prop} \label{prop:mf:2ab:rk2}
	\label{Thm:sym_mat_fact_for_extension_bundle}
Let $\XX$ be a weighted projective line of triple weight type $(2,a,b)$. Then each extension bundle $E=\extb{L}{\vx},$ where $\vx=l_2\vx_2+l_3\vx_3$, and $0\leq l_2\leq a-2$, $0\leq l_3\leq b-2$, admits a symmetric matrix factorization
$(u_{\vx},v_{\vx}, (\vw, \vx-\vx_1,l_3\vx_3-\vx_2, l_2\vx_2-\vx_3))$, of $f=x^2+y^a+z^b$, where

$$u_\vx=v_\vx=\left[\begin{array}{cccc}
x&0&-z^{b-l_3-1}&y^{a-l_2-1}\\
0& x& y^{l_2+1} & z^{l_3+1}\\
-z^{l_3+1}& y^{a-l_2-1}& -x&0\\
y^{l_2+1}& z^{b-l_3-1}& 0&-x
\end{array} \right].$$ \hfill\qed
\end{Prop}

\begin{pf}
We fix a decomposition of $\pc{E}$ into line bundles, that is transferred to the other terms by degree shifts with $-\vx_1$ and $-\vc$. We then obtain a matrix factorization $(u_\vx,v_\vx)$, such that $u_\vx$ and $v_\vx=u_\vx(-\vx_1)$ are represented by the same matrix. Defining the above specialization $u_\vx$ of the factorization frame for $E$, we satisfy the matrix factorization condition $u_\vx^2 =f\,\id$ over $T$. Using similar arguments as before, it is again easy to check that this matrix factorization has a trivial endomorphism algebra. Hence it is indecomposable, and then by Proposition \ref{prop:determined_by_cover:abc} it represents $E$.
\end{pf}

\subsection*{The domestic case}

The main result of this section, and actually the main result of this paper, concerns the domestic case, necessarily of type $(2,a,b)$, where, for each indecomposable vector bundles of rank at least two, we determine explicitly a \emph{symmetric} matrix factorization.

\begin{Thm} \label{thm:main:domestic}
We assume a triangle singularity $f=x^2+y^a+z^b$ of domestic type. For each indecomposable bundle $E$ of rank at least two, we obtain a symmetric matrix factorization $u^2=f\,\id$ of $f$ representing $E$. The matrix entries of $u$ are scalar multiples of monomials in $x,y,z$ with small exponents. Moreover,  the scalars may be taken from $\{0,\pm1\}$.
\end{Thm}
\begin{pf}
Interpreting the minimal projective resolution \eqref{eqn:resolution:symmetric} as a sequence of $\LL$-graded Cohen-Macaulay modules over $S$, we lift it to a matrix factorization of $f$ over $T$, thus obtaining a commutative diagram as follows.
\begin{equation} \label{eqn:matrixfact:projres}
\xymatrix{
\cdots\ar[r]& P_0(-\vc)\ar[r]^{v}\ar[d]_\nu & P_0(-\vx_1)\ar[d]_\nu \ar[r]^{u} & P_0\ar[d]_\nu \ar[r]^\pi& M\ar@{=}[d] \ar[r]& 0 \\
\cdots\ar[r]& \bar{P}_0(-\vc)\ar[r]^{\bar{v}=\bar{u}(-\vx_1)} & \bar{P}_0(-\vx_1) \ar[r]^{\bar{u}} & \bar{P}_0 \ar[r]^{\bar{\pi}}& M \ar[r]& 0 \\}
\end{equation}
First we are going to show that the matrix factorization $(u,v)$ is symmetric, that is that $v=u(-\vx_1)$ holds. Now, reduction modulo $(f)$ sends $v$ and $u(-\vx_1)$ to the same map $\bar{v}$. Since $\XX$ has domestic type, and invoking compatible decompositions of the $P_i$ and $\bar{P}_i$ into indecomposable projectives, Theorem \ref{thm:hom_dimensions} shows that $v=u(-\vx_1)$. Moreover, by the same theorem, the matrix entries (with respect to the chosen decomposition) are scalar multiples of monomials with small exponents.

That the scalars, actually, can be chosen among $0$ and $\pm1$ follows through a case-by-case analysis from the following Propositions.
\end{pf}

For the rest of this section, we derive explicit matrix factorizations for all indecomposable vector bundles of rank at least two. We recall, see Lemma \ref{lem:degree.shift}, that a single matrix factorization is sufficient to represent all members of a fixed Auslander-Reiten orbit.

\subsubsection*{The triangle singularity $x^2+y^2+z^n$ $(n\geq 2)$.} Here, the projective covers are given by Proposition \ref{proj_cover_2_2_n}. Since for type $(2,2,n)$ all indecomposable vector bundles, that are not line bundles, have rank two, the corresponding matrix factorizations are a special case of Proposition \ref{Thm:sym_mat_fact_for_extension_bundle}.
\begin{Prop} \label{prop:mf:22n}
Let $E_i$ denote the extension bundle $\extb{L}{i\vx_3}$ for $i=0,1,\ldots, n-2$. Then $E_i$ yields a symmetric matrix factorization $(u_{E_i}, v_{E_i}, \pc{E_i})$ of $x^2+y^2+z^n$  as follows:
$$u_{E_i}= v_{E_i}=
\left[\begin{array}{cccc}
x& 0 & -z^{n-i-1} & y\\
0 & x& y & z^{i+1} \\
-z^{i+1} & y & -x & 0\\
y & z^{n-i-1} & 0 & -x
\end{array}\right],$$
for $i=0,1,\ldots, n-2$.\hfill\qed
\end{Prop}

\subsubsection*{The triangle singularity $x^2+y^3+z^3$} We will use the projective covers described in Section \ref{sec_proj_covers}, Proposition \ref{proj_cover_2_3_3} and use the notations from the Auslander-Reiten quiver of type $(2,3,3)$ depicted there.

\begin{Prop} \label{prop:mf:233}
For the singularity $x^2+y^3+z^3$ we obtain the following symmetric matrix factorizations
$\left(u_2,v_2, \pc{E_2}\right)$,
$\left(u_2,v_2,\pc{F_2}\right)$,
$\left(u_2,v_2,\pc{G_2}\right)$,
$\left(u_3,v_3,\pc{E_3}\right)$,
where
$u_2=v_2=\left[\begin{array}{cccc}
x&0&-z^{2}&y^2\\
0& x& y & z\\
-z& y^2& -x&0\\
y& z^2& 0&-x
\end{array} \right],$
and
$u_3=v_3= \left[\begin{array}{cccccc}
x & yz & 0 & y^{2} & 0 & -z^{2} \\
0 & -x & z & 0 & -y & 0 \\
0 & z^{2} & x & yz & 0 & y^{2} \\
y & 0 & 0 & -x & z & 0 \\
0 & -y^{2} & 0 & z^{2} & x & yz \\
-z & 0 & y & 0 & 0 & -x%
\end{array} \right].$
\end{Prop}

\begin{pf}
The vector bundles $E_2$, $F_2$, $G_2$ are extension bundles determined by the pairs $(\Oo(-\vw), \vec{0})$, $(\Oo(\vx_2+\vw), \vec{0})$, $(\Oo(\vx_3+\vw), \vec{0})$, respectively. Therefore the claim for those bundles results from Proposition \ref{Thm:sym_mat_fact_for_extension_bundle}. Concerning the vector bundle $E_3$, it is first checked that $u_3v_3=f\,\id=v_3u_3$. Since, moreover, the direct summands of $\pc{E_3}$ are mutually Hom-orthogonal, it is easy to check that the matrix factorization $(u_3,v_3)$ is indecomposable. By Proposition \ref{prop:determined_by_cover} it then represents the vector bundle in question.
\end{pf}

\subsubsection*{The triangle singularity $x^2+y^3+z^4$} \label{ssect:2,3,4}

For the weight type $(2,3,4)$, each indecomposable vector bundle is of rank $1$, $2$, $3$ or $4$. As in Proposition \ref{Thm:sym_mat_fact_for_extension_bundle}, and using the notations introduced there, we only need to determine matrix factorizations for the vector bundles $E_2,E_3,E_4$ and $F_2$, since by symmetry $E_i$ and $E_i(\vx_1-2\vx_3)$ will yield the same matrix pair. Moreover, matrix factorizations for the rank-two bundles $E_2$, and $G_2$ are already given by Proposition \ref{Thm:sym_mat_fact_for_extension_bundle}. we thus obtain:

\begin{Prop} \label{prop:mf:234}
For the singularity $x^2+y^3+z^4$ we obtain the following symmetric matrix factorizations
$(u_{E_i}, v_{E_i}, \pc{E_i})$ for $i=2,3,4$, $(u_{E_i}, v_{E_i}, \pc{E_i(\vx_1-2\vx_3)})$ for $i=2,3$, $(u_{G_2}, v_{G_2}, \pc{G_2})$,
where
$$u_{E_2}=v_{E_2}=\left[\begin{array}{cccc}
x&0&-z^{3}&y^2\\
0& x& y & z\\
-z& y^2& -x&0\\
y& z^3& 0&-x
\end{array}\right]$$

$$u_{E_3}=v_{E_3}=\left[\begin{array}{cccccc}
x & 0 & z^{3} & 0 & -y^{2} & -yz^{2} \\
0 & x & yz & 0 & z^{2} & -y^{2} \\
z & 0 & -x & y & 0 & 0 \\
0 & 0 & y^{2} & x & yz & z^{3} \\
-y & z^{2} & 0 & 0 & -x & 0 \\
0 & -y & 0 & z & 0 & -x%
\end{array}\right]$$
$$u_{E_4}=v_{E_4}=\left[\begin{array}{cccc|cccc}
x & 0 & -z^{2} & y^{2} & 0 & -yz & 0 & 0 \\
0 & x & y & z^{2} & z & 0 & 0 & 0 \\
-z^{2} & y^{2} & -x & 0 & 0 & 0 & 0 & -yz \\
y & z^{2} & 0 & -x & 0 & 0 & z & 0 \\
\hline
0 & 0 & 0 & 0 & -x & 0 & -z^{2} & y^{2} \\
0 & 0 & 0 & 0 & 0 & -x & y & z^{2} \\
0 & 0 & 0 & 0 & -z^{2} & y^{2} & x & 0 \\
0 & 0 & 0 & 0 & y & z^{2} & 0 & x%
\end{array}\right]$$
$$u_{G_2}=v_{G_2}=\left[\begin{array}{cccc}
x&0&-z&y^2\\
0& x& y & z^2\\
-z^2& y^2& -x&0\\
y& z^2& 0&-x
\end{array}\right],$$
\end{Prop}

\begin{pf}
It easy to check that the above matrices satisfy the matrix equations $u^2=f\,\id$. By Proposition \ref{prop:determined_by_cover}, it remains to prove the indecomposability of those matrix factorizations by showing that their endomorphism rings are trivial. For the indecomposable vector bundles of rank two and three, the indecomposability easily follows, since the indecomposable direct summands of their projective covers are Hom-orthogonal. It remains to check indecomposability for $(u_{E_4}, v_{E_4})$, where
$$u_{E_4}=v_{E_4}=\left[\begin{array}{c|c}
u_{\tauX T}&B\\
\hline
0&u_{T}
\end{array}\right]$$
by using the explicit form of the projective cover of $E_4$ by means of Proposition~\ref{proj_cover_2_3_4}. The same argument yields the shape of
an endomorphism  $(K,H)$ for $(u_{E_4}, v_{E_4})$ as follows:
$$H=\left[\begin{array}{c|c}
H_1&0\\
\hline
H_3&H_4
\end{array}\right]=\left[\begin{array}{cccc|cccc}
f_{1} & 0 & 0 & 0 & 0 & 0 & 0 & 0 \\
0 & f_{2} & 0 & 0 & 0 & 0 & 0 & 0 \\
0 & 0 & f_{3} & 0 & 0 & 0 & 0 & 0 \\
0 & 0 & 0 & f_{4} & 0 & 0 & 0 & 0 \\
\hline
0 & 0 & 0 & f_{5,4}z & f_{5} & 0 & 0 & 0 \\
0 & 0 & 0 & 0 & 0 & f_{6} & 0 & 0 \\
0 & f_{7,2}z & 0 & 0 & 0 & 0 & f_{7} & 0 \\
0 & 0 & 0 & 0 & 0 & 0 & 0 & f_{8}%
\end{array}\right],$$
$$K=\left[\begin{array}{c|c}
K_1&0\\
\hline
K_3&K_4
\end{array}\right]=\left[ \begin{array}{cccc|cccc}
g_{1} & 0 & 0 & 0 & 0 & 0 & 0 & 0 \\
0 & g_{2} & 0 & 0 & 0 & 0 & 0 & 0 \\
0 & 0 & g_{3} & 0 & 0 & 0 & 0 & 0 \\
0 & 0 & 0 & g_{4} & 0 & 0 & 0 & 0 \\
\hline
0 & 0 & 0 & g_{5,4}z & g_{5} & 0 & 0 & 0 \\
0 & 0 & 0 & 0 & 0 & g_{6} & 0 & 0 \\
0 & g_{7,2}z & 0 & 0 & 0 & 0 & g_{7} & 0 \\
0 & 0 & 0 & 0 & 0 & 0 & 0 & g_{8}%
\end{array}\right].$$
It follows that $Hu_{E_4} =u_{E_4}K$. Now, in the block matrix form
$$0=\left[\begin{array}{c|c}
H_1&0\\
\hline
H_3&H_4
\end{array}\right]
\left[\begin{array}{c|c}
u_{\tauX T}&B\\
\hline
0&u_{T}
\end{array}\right]-
\left[\begin{array}{c|c}
u_{\tauX T}&B\\
\hline
0&u_{T}
\end{array}\right]
\left[\begin{array}{c|c}
K_1&0\\
\hline
K_3&K_4
\end{array}\right]$$
$$=
\left[\begin{array}{c|c}
H_1u_{\tauX T}-u_{\tauX T}K_1&H_1B-BK_4\\
\hline
H_3u_{\tauX T}-u_TK_3&H_3B+H_4u_T-u_TK_4
\end{array}\right],
$$
and we obtain that $(K_1,H_1)$ is an endomorphism for $(u_{\tauX T}, v_{\tauX T})$. Therefore $H_1=K_1=\lambda \id_4$. Moreover
$$0=H_3u_{\tauX T}-u_TK_3=\left[\begin{array}{cccc}
yzf_{5,4} & z^{3}f_{5,4}+z^{3}g_{7,2} & 0 &
-xzf_{5,4}+xzg_{5,4} \\
0 & -yzg_{7,2} & 0 & 0 \\
0 & xzf_{7,2}-xzg_{7,2} & yzf_{7,2} &
z^{3}f_{7,2}+z^{3}g_{5,4} \\
0 & 0 & 0 & -yzg_{5,4}%
\end{array}\right].$$
Thus $K_3=H_3=0$, hence $(K_4,H_4)$ is an endomorphism for $(u_T,v_T)$, and from the indecomposability of $(u_T,v_T)$, we get $H_4=K_4=\mu \id_4.$ The equation $H_1B-BK_4$ implies that $\lambda=\mu$. Therefore  $(u_{E_4},v_{E_4})$ is indecomposable, and the claim follows from Proposition \ref{prop:determined_by_cover}.
\end{pf}

\subsubsection*{The triangle singularity $x^2+y^3+z^5$} \label{ssect:2,2,5}

In this case, each indecomposable vector bundle on $\XX$ is of rank $m$, $1 \le m \le 6$, and there is a single $\tauX$-orbit of vector bundles of rank $5$, respectively of rank $6$.

For a representative system of indecomposable vector bundles $E_2,E_3,\ldots,E_6$, $F_2,F_4$ and $G_3$ of rank at least two, we use the choices and notations of Section \ref{sec_proj_covers}.

\begin{Prop} \label{prop:mf:235}
For the singularity $f=x^2+y^3+z^5$ we obtain the following $8$ matrix factorizations
$(u_{E_2},v_{E_2},\pc{E_2}),\dots, (u_{E_6},v_{E_6},\pc{E_6})$, and $(u_{F_2},v_{F_2},\pc{F_2})$, $(u_{F_4},v_{F_4},\pc{F_4})$, and $(u_{G_3},v_{G_3},\pc{G_3})$ where
$$u_{E_2}=v_{E_2}=\left[\begin{array}{cccc}
x&0&-z^4&y^2\\
0& x& y & z\\
-z& y^2& -x&0\\
y& z^4& 0&-x
\end{array}\right]$$
$$u_{E_3}=v_{E_3}=\left[\begin{array}{cccccc}
x & 0 & -y^{2} & 0 & z^{4} & -yz^{3} \\
0 & x & yz & 0 & y^{2} & z^{4} \\
-y & 0 & -x & z^{3} & 0 & 0 \\
0 & 0 & z^{2} & x & yz & -y^{2} \\
z & y & 0 & 0 & -x & 0 \\
0 & z & 0 & -y & 0 & -x%
\end{array}\right]$$
$$u_{E_4}=v_{E_4}=\left[\begin{array}{cccc|cccc}
x & 0 & z^{3} & y^{2} & 0 & 0 & 0 & yz \\
0 & x & -y & z^{2} & 0 & 0 & z & 0 \\
z^{2} & -y^{2} & -x & 0 & 0 & yz & 0 & 0 \\
y & z^{3} & 0 & -x & z & 0 & 0 & 0 \\
\hline
0 & 0 & 0 & 0 & x & 0 & -z^{3} & -y^{2} \\
0 & 0 & 0 & 0 & 0 & x & y & -z^{2} \\
0 & 0 & 0 & 0 & -z^{2} & y^{2} & -x & 0 \\
0 & 0 & 0 & 0 & -y & -z^{3} & 0 & -x%
\end{array} \right]$$
$$u_{E_5}=v_{E_5}=\left[\begin{array}{cccccc|cccc}
x & 0 & 0 & y^{2} & yz^{2} & z^{4} & 0 & 0 & 0 & -z^{3} \\
0 & x & 0 & -z^{3} & y^{2} & yz^{2} & 0 & 0 & 0 & -yz \\
0 & 0 & x & -yz & -z^{3} & y^{2} & 0 & 0 & z^{2} & 0 \\
y & -z^{2} & 0 & -x & 0 & 0 & 0 & 0 & 0 & 0 \\
0 & y & -z^{2} & 0 & -x & 0 & z & 0 & 0 & 0 \\
z & 0 & y & 0 & 0 & -x & 0 & z^{2} & 0 & 0 \\
\hline
0 & 0 & 0 & 0 & 0 & 0 & x & 0 & z^{3} & y^{2} \\
0 & 0 & 0 & 0 & 0 & 0 & 0 & x & -y & z^{2} \\
0 & 0 & 0 & 0 & 0 & 0 & z^{2} & -y^{2} & -x & 0 \\
0 & 0 & 0 & 0 & 0 & 0 & y & z^{3} & 0 & -x%
\end{array} \right]$$
$$u_{E_6}=v_{E_6}=
\left[
\begin{array}{cccccc|cccccc}
x & 0 & 0 & y^{2} & yz^{2} & z^{4} & 0 & -yz & 0 & 0 & 0 & 0 \\
0 & x & 0 & -z^{3} & y^{2} & yz^{2} & z^{2} & 0 & 0 & 0 & 0 & 0 \\
0 & 0 & x & -yz & -z^{3} & y^{2} & y & 0 & 0 & 0 & 0 & 0 \\
y & -z^{2} & 0 & -x & 0 & 0 & 0 & 0 & 0 & 0 & yz & z^{3} \\
0 & y & -z^{2} & 0 & -x & 0 & 0 & 0 & 0 & 0 & 0 & 0 \\
z & 0 & y & 0 & 0 & -x & 0 & 0 & 0 & -y & 0 & 0 \\
\hline
0 & 0 & 0 & 0 & 0 & 0 & -x & 0 & 0 & y^{2} & yz^{2} & z^{4} \\
0 & 0 & 0 & 0 & 0 & 0 & 0 & -x & 0 & -z^{3} & y^{2} & yz^{2} \\
0 & 0 & 0 & 0 & 0 & 0 & 0 & 0 & -x & -yz & -z^{3} & y^{2} \\
0 & 0 & 0 & 0 & 0 & 0 & y & -z^{2} & 0 & x & 0 & 0 \\
0 & 0 & 0 & 0 & 0 & 0 & 0 & y & -z^{2} & 0 & x & 0 \\
0 & 0 & 0 & 0 & 0 & 0 & z & 0 & y & 0 & 0 & x%
\end{array}
\right].$$
$$u_{F_2}=v_{F_2}=\left[\begin{array}{cccc}
x&0&-z^{3}&y^2\\
0& x& y & z^2\\
-z^2& y^2& -x&0\\
y& z^3& 0&-x
\end{array}\right]$$
$$u_{F_4}=v_{F_4}=\left[ \begin{array}{cccc|cccc}
x & 0 & z^{3} & y^{2} & 0 & yz^{2} & 0 & 0 \\
0 & x & -y & z^{2} & z & 0 & 0 & 0 \\
z^{2} & -y^{2} & -x & 0 & 0 & 0 & 0 & yz \\
y & z^{3} & 0 & -x & 0 & 0 & z^{2} & 0 \\
\hline
0 & 0 & 0 & 0 & -x & 0 & -z^{3} & y^{2} \\
0 & 0 & 0 & 0 & 0 & -x & -y & -z^{2} \\
0 & 0 & 0 & 0 & -z^{2} & -y^{2} & x & 0 \\
0 & 0 & 0 & 0 & y & -z^{3} & 0 & x%
\end{array}\right]$$
$$u_{G_3}=v_{G_3}=\left[\begin{array}{cccccc}
x & 0 & 0 & y^{2} & yz^{2} & z^{4} \\
0 & x & 0 & -z^{3} & y^{2} & yz^{2} \\
0 & 0 & x & -yz & -z^{3} & y^{2} \\
y & -z^{2} & 0 & -x & 0 & 0 \\
0 & y & -z^{2} & 0 & -x & 0 \\
z & 0 & y & 0 & 0 & -x%
\end{array}  \right]$$
\end{Prop}

\begin{pf}
Similarly as in the case $(2,3,4)$ we check that $uv=f\,\id=vu$ and verify that these matrix factorizations are indecomposable. The case of rank-two bundles is covered by Proposition \ref{prop:mf:2ab:rk2}. The matrix factorizations for the rank-three bundles $E_3$ and $F_3$ are verified following the arguments of Proposition \ref{prop:mf:234}. Concerning the remaining vector bundles,  we determine the pair $(u_{E_i},v_{E_i})$ for $i=4,5,6$ and  $(u_{F_i}$, $v_{F_i})$ for $i=4$ by specialization of factorization frames coming from distinguished exact sequences by means of Proposition \ref{proj_cover_2_3_5}. In particular, for $(u_{E_4},v_{E_4})$, $(u_{E_5},v_{E_5})$, $(u_{E_6},v_{E_6})$ and $(u_{F_4},v_{F_4})$, we use the exact sequences $0\lra \tauX^2F_2\lra E_4\lra \tauX^{-2} F_2\lra 0$, $0\lra \tauX G_3\lra E_5\lra \tauX^- F_2\lra 0$, $0\lra G_3\lra E_6\lra \tauX^- G_3\lra 0$ and $0\lra \tauX F_2\lra F_4\lra F_2\lra 0$.
\end{pf}

\begin{Rem}
Observe that $0\lra\tauX F_4\lra E_6\lra \tauX^-F_2\lra 0$ is a distinguished exact sequence, thus satisfying the assumptions of Lemma \ref{matrix_frame_from_sequence}. Using the resulting direct decomposition $\pc{E_6}=\pc{\tau F_4}\oplus \pc{\tau^- F_2}$, we get another -- essentially different-- pair of matrices, also yielding a matrix factorization of $E_6$:
$$u'_{E_6}=v'_{E_6}=\left[ \begin{array}{cccccccc|cccc}
x & 0 & z^{3} & y^{2} & 0 & yz^{2} & 0 & 0 & 0 & 0 & 0 & -z^{3} \\
0 & x & -y & z^{2} & z & 0 & 0 & 0 & 0 & 0 & 0 & y \\
z^{2} & -y^{2} & -x & 0 & 0 & 0 & 0 & yz & y & z^{3} & 0 & 0 \\
y & z^{3} & 0 & -x & 0 & 0 & z^{2} & 0 & 0 & 0 & 0 & 0 \\
0 & 0 & 0 & 0 & -x & 0 & -z^{3} & y^{2} & 0 & 0 & 0 & 0 \\
0 & 0 & 0 & 0 & 0 & -x & -y & -z^{2} & 0 & 0 & 0 & 0 \\
0 & 0 & 0 & 0 & -z^{2} & -y^{2} & x & 0 & 0 & 0 & 0 & 0 \\
0 & 0 & 0 & 0 & y & -z^{3} & 0 & x & 0 & 0 & 0 & 0 \\
\hline
0 & 0 & 0 & 0 & 0 & 0 & 0 & 0 & x & 0 & z^{3} & y^{2} \\
0 & 0 & 0 & 0 & 0 & 0 & 0 & 0 & 0 & x & -y & z^{2} \\
0 & 0 & 0 & 0 & 0 & 0 & 0 & 0 & z^{2} & -y^{2} & -x & 0 \\
0 & 0 & 0 & 0 & 0 & 0 & 0 & 0 & y & z^{3} & 0 & -x%
\end{array}\right]$$
In particular, the numbers of zero entries differ for both matrix factorizations.
\end{Rem}

\begin{Rem}
If we compare our matrix factorizations for $E_6$, with the matrix factorization obtained in \cite{KST-1}, we also see that essentially different matrix factorizations are obtained, since for one matrix factorization there appear monomial entries $z^4$, for the other one not.
\end{Rem}

\section{Appendix: Tables of projective covers}
The figures of this section yield compact visual information on the projective covers of indecomposable vector bundles of rank at least two. Our figures may be especially useful for specialists from the representation theory of finite dimensional algebras investigating the related situation in preprojective, or preinjective, components for tame concealed quivers. We note that, in the representation-theoretic context, the line bundle notation $\Oo(\vx)$, reduced in the figures to $(\vx)$, is not established, so the given positions in the mesh category of the associated extended Dynkin quiver should be useful. The names $E_i$, $F_j$ and $G_l$ for selected vector bundles are those from Section \ref{sec_proj_covers}.

\subsection*{Weight type $(2,2,n)$} \quad

\newcommand{\neven}{\xymatrix @=20pt @!0 {
\ar@{-}[dddrrr]&&*+<1pt>{\tb{(\vx_3+\vu)}} \ar@{-}[dddrrr]&&\ar@{-}[dddrrr]&&\ar@{-}[ddrr]&&\\
\ar@{-}[rr]&&*+<1pt>{\bb{(-\vom+\vu)}}\ar@{-}[rrrrrr]&&&&&&\\
\ar@{-}[dr]\ar@{-}[uurr]&&&&&&&&\\
&*+<1pt>{}\ar@{-}[uuurrr]&&*+<1pt>{}\ar@{-}[uuurrr]&&*+<1pt>{E}\ar@{-}[uuurrr]&&*+<1pt>{}\ar@{-}[ur]&\\
*+<1pt>{}\ar@{.}[ur]&&*+<1pt>{}\ar@{.}[ur]&&*+<1pt>{}\ar@{.}[ur]&&*+<1pt>{}\ar@{.}[ur]&&\\
\ar@{-}[rr]&&*+<1pt>{\tb{(\vx_3)}}\ar@{-}[rrrrrr]&&&&&&\\
\ar@{-}[uurr]&&*+<1pt>{\tb{(-\vw)}}\ar@{-}[uull]\ar@{-}[uurr]&&\ar@{-}[uull]\ar@{-}[uurr]&&\ar@{-}[uull]\ar@{-}[uurr]&&\ar@{-}[uull]\\
}}
\newcommand{\nodd}{\xymatrix @=20pt @!0 {
&\ar @{-}[2,2]&&*+<1pt>{\bb{(\vx_1+\vw})}\ar @{-}[2,2]&&\ar @{-}[2,2]&&\ar @{-}[1,1]&\\
\ar@{-}[0,3]\ar@{-}[dr]\ar@{-}[ur]&&&*+<1pt>{\bb{(\vx_2+\vw})}\ar@{-}[0,5]&&&&&\\
&*+<1pt>{}\ar@{-}[uurr]&&*+<1pt>{}\ar@{-}[uurr]&&*+<1pt>{E}\ar@{-}[uurr]&&*+<1pt>{}\ar@{-}[ur]&\\
*+<1pt>{}\ar@{.}[ur]&&*+<1pt>{}\ar@{.}[ur]&&*+<1pt>{}\ar@{.}[ur]&&*+<1pt>{}\ar@{.}[ur]&&\\
\ar@{-}[rr]&&*+<1pt>{\tb{(\vx_3)}}\ar@{-}[rrrrrr]&&&&&&\\
\ar@{-}[uurr]&&*+<1pt>{\tb{(-\vw)}}\ar@{-}[uull]\ar@{-}[uurr]&&\ar@{-}[uull]\ar@{-}[uurr]&&\ar@{-}[uull]\ar@{-}[uurr]&&\ar@{-}[uull]\\
}}

\begin{figure}[H]
\small
$$
\begin{array}{cc}
\tiny\neven & \tiny\nodd \\
n \textrm{ even}& n \textrm{ odd}
\end{array}
$$
\caption{$(2,2,n)$: Projective cover of $E=\extb{}{i\vx_3}$, where $\vu=\vx_1-\vx_2$}
\end{figure}

\subsection*{Weight type $(2,3,3)$}
The projective covers of the `remaining' indecomposable vector bundles then are obtained by applying twice the rotation $X\mapsto X(\vx_2-\vx_3)$ around the central axis.

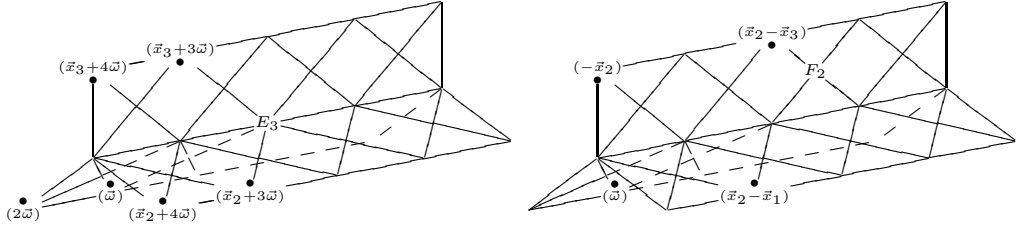
\begin{figure}[H]
{\tiny
$\xymatrix @=3.3pt @!0 {
&&&&&&&&&&&&&&&&&&&&&&&&&&&&&&&&&&&&&&&&&&&&&&&&*+<1pt>{}&&&&&&&\\
&&&&&&&&&&&&&&&&&&&&&&&&&&&&&&&&&&&&&&&&&&&&&&&&&&&&&&&&\\
&&&&&&&&&&&&&&&&&&&&&&&&&&&&&&&&&&&&&&*+<1pt>{}\ar @{-}[8,10] &&&&&&&&&&&&&&&\\
&&&&&&&&&&&&&&&&&&&&&&&&&&&&&&&&&&&&&&&&&&&&&&&&&&&&&&&&\\
&&&&&&&&&&&&&&&&&&&&&&&&&&&&*+<1pt>{}\ar @{-}[8,10] &&&&&&&&&&&&&&&&&&&&&&&&&&&&\\
&&&&&&&&&&&&&&&&&&&&&&&&&&&&&&&&&&&&&&&&&&&&&&&&&&&&&&&&\\
&&&&&&&&&&&&&&&&&&*+<1pt>{\tb{(\vx_3+3\vw)}}\ar @{-}@[ashy][-6,30]\ar @{-}[8,10] &&&&&&&&&&&&&&&&&&&&&&&&& &&&&&&&&&&&&&\\
&&&&&&&&&&&&&&&&&&&&&&&&&&&&&&&&&&&&&&&&&&&&&&&&&&&&&&\\
&&&&&&&&*+<1pt>{\tb{(\vx_3+4\vw)}}\ar @{-}@[ashy][-2,10]\ar @{-}[8,10] &&&&&&&&&&&&&&&&&&&&&&&&&&&&&&&&&&&&&&&&&&&&&&&&\\
&&&&&&&&&&&&&&&&&&&&&&&&&&&&&&&&&&&&&&&&&&&&&&&&&&&&&&\\
&&&&&&&&&&&&&&&&&&&&&&&&&&&&&&&&&&&&&&&&&&&&&&&&*+<1pt>{}\ar @{-}@[ashy][-10,0]\ar @{--}@[ashy][6,-8] &&&&&&&&\\
&&&&&&&&&&&&&&&&&&&&&&&&&&&&&&&&&&&&&&&&&&&&&&&&&&&&&&&&\\
&&&&&&&&&&&&&&&&&&&&&&&&&&&&&&&&&&&&&&*+<1pt>{}\ar @{-}[-12,10]\ar @{-}[4,18] &&&&&&&&&&&&&&&&&&\\
&&&&&&&&&&&&&&&&&&&&&&&&&&&&&&&&&&&&&&&&&&&&&&&&&&&&&&&&\\
&&&&&&&&&&&&&&&&&&&&&&&&&&&&*+<1pt>{E_3}\ar @{-}@[ashy][-4,20] \ar @{-}[-12,10]\ar @{-}[4,18] &&&&&&&&&&&&&&&&&&&&&&&&&&&&\\
&&&&&&&&&&&&&&&&&&&&&&&&&&&&&&&&&&&&&&&&&&&&&&&&&&&&&&&&\\
&&&&&&&&&&&&&&&&&&*+<1pt>{} \ar @{-}[-12,10]\ar @{-}[4,18] \ar @{--}[4,2] &&&&&&&&&&&&&&&&&&&&&&&&&&&&&&&&&&&&&&*+<1pt>{}\ar @{-}@[ashy][-6,-8]\\
&&&&&&&&&&&&&&&&&&&&&&&&&&&&&&&&&&&&&&&&&&&&&&&&&&&&&&&&\\
&&&&&&&&*+<1pt>{}\ar @{-}@[ashy][-10,0]\ar @{-}@[ashy][-4,20]\ar @{-}@[ashy][6,8] \ar @{-}[-12,10]\ar @{-}[4,18]\ar @{-}[4,2] &&&&&&&&&&&&&&&&&&& &&&&&&&&&&&&&&&&&&&*+<1pt>{}\ar @{-}[-8,2]&&&&&&&&&&\\
&&&&&&&&&&&&&&&&&&&&&&&&&&&&&&&&&&&&&&&&&&&&&&&&&&&&&&&&\\
&&&&&&&&&*+<1pt>{}\ar @{--}[-4,9] &&&&&&&& &&&&&&&&&&&&&&&&&&&*+<1pt>{}\ar @{-}[-8,2]&&&&&&&&&&&&&&&&&&&&\\
&&&&&&&&&&&&&&&&&&&&&&&&&&&&&&&&&&&&&&&&&&&&&&&&&&&&&&&&\\
&&&&&&&&&&*+<1pt>{\bb{(\vw)}}\ar @{--}@[ashy][-6,30]\ar @{--}[-8,18] &&&&&&&&&&&&&&&&*+<1pt>{\bb{(\vx_2+3\vw)}}\ar @{-}@[ashy][-6,30]\ar @{-}[-8,2]&&&&&&&&&&&&&&&&&&&&&&&&&&&&&&\\
&&&&&&&&&&&&&&&&&&&&&&&&&&&&&&&&&&&&&&&&&&&&&&&&&&&&&&&&\\
*+<1pt>{\bb{(2\vw)}}\ar @{-}[-4,9]\ar @{-}@[ashy][-2,10]\ar @{-}@[ashy][-6,8]  &&&&&&&&&&&&&&&&*+<1pt>{\bb{(\vx_2+4\vw)}}\ar @{-}@[ashy][-2,10]\ar @{-}[-8,2]&&&&&&&&&&&&&&&&&&&&&&&&&&&&&&&&&&&&&&&&\\}
\xymatrix @=3.3pt @!0 {
&&&&&&&&&&&&&&&&&&&&&&&&&&&&&&&&&&&&&&&&&&&&&&&&*+<1pt>{}&&&&&&&\\
&&&&&&&&&&&&&&&&&&&&&&&&&&&&&&&&&&&&&&&&&&&&&&&&&&&&&&&&\\
&&&&&&&&&&&&&&&&&&&&&&&&&&&&&&&&&&&&&&*+<1pt>{}\ar @{-}[8,10] &&&&&&&&&&&&&&&\\
&&&&&&&&&&&&&&&&&&&&&&&&&&&&&&&&&&&&&&&&&&&&&&&&&&&&&&&&\\
&&&&&&&&&&&&&&&&&&&&&&&&&&&&*+<1pt>{\tb{(\vx_2-\vx_3)}}\ar @{-}[4,5]\ar @{-}@[ashy][-4,20] &&&&&&&&&&&&&&&&&&&&&&&&&&&&\\
&&&&&&&&&&&&&&&&&&&&&&&&&&&&&&&&&&&&&&&&&&&&&&&&&&&&&&&&\\
&&&&&&&&&&&&&&&&&&*+<1pt>{}\ar @{-}[8,10] &&&&&&&&&&&&&&&&&&&&&&&&& &&&&&&&&&&&&&\\
&&&&&&&&&&&&&&&&&&&&&&&&&&&&&&&&&&&&&&&&&&&&&&&&&&&&&&\\
&&&&&&&&*+<1pt>{\tb{(-\vx_2)}}\ar @{-}@[ashy][-4,20]\ar @{-}[8,10] &&&&&&&&&&&&&&&&&&&&&&&&&*+<1pt>{F_2}\ar @{-}[4,5]\ar @{-}[-6,5] &&&&&&&&&&&&&&&&&&&&&&&\\
&&&&&&&&&&&&&&&&&&&&&&&&&&&&&&&&&&&&&&&&&&&&&&&&&&&&&&\\
&&&&&&&&&&&&&&&&&&&&&&&&&&&&&&&&&&&&&&&&&&&&&&&&*+<1pt>{}\ar @{-}@[ashy][-10,0]\ar @{--}@[ashy][6,-8] &&&&&&&&\\
&&&&&&&&&&&&&&&&&&&&&&&&&&&&&&&&&&&&&&&&&&&&&&&&&&&&&&&&\\
&&&&&&&&&&&&&&&&&&&&&&&&&&&&&&&&&&&&&&*+<1pt>{} \ar @{-}[-12,10]\ar @{-}[4,18] &&&&&&&&&&&&&&&&&&\\
&&&&&&&&&&&&&&&&&&&&&&&&&&&&&&&&&&&&&&&&&&&&&&&&&&&&&&&&\\
&&&&&&&&&&&&&&&&&&&&&&&&&&&&*+<1pt>{} \ar @{-}[-6,5]\ar @{-}[4,18] &&&&&&&&&&&&&&&&&&&&&&&&&&&&\\
&&&&&&&&&&&&&&&&&&&&&&&&&&&&&&&&&&&&&&&&&&&&&&&&&&&&&&&&\\
&&&&&&&&&&&&&&&&&&*+<1pt>{}\ar @{--}[4,2] \ar @{-}[-12,10]\ar @{-}[4,18] &&&&&&&&&&&&&&&&&&&&&&&&&&&&&&&&&&&&&&*+<1pt>{}\ar @{-}@[ashy][-6,-8]\\
&&&&&&&&&&&&&&&&&&&&&&&&&&&&&&&&&&&&&&&&&&&&&&&&&&&&&&&&\\
&&&&&&&&*+<1pt>{}\ar @{-}@[ashy][-10,0]\ar @{-}@[ashy][-8,40]\ar @{-}@[ashy][6,8] \ar @{-}[-12,10]\ar @{-}[4,18]\ar @{-}[4,2] &&&&&&&&&&&&&&&&&&& &&&&&&&&&&&&&&&&&&&*+<1pt>{}\ar @{-}[-8,2]&&&&&&&&&&\\
&&&&&&&&&&&&&&&&&&&&&&&&&&&&&&&&&&&&&&&&&&&&&&&&&&&&&&&&\\
&&&&&&&&&*+<1pt>{}\ar @{--}[-4,9] &&&&&&&& &&&&&&&&&&&&&&&&&&&*+<1pt>{}\ar @{-}[-8,2]&&&&&&&&&&&&&&&&&&&&\\
&&&&&&&&&&&&&&&&&&&&&&&&&&&&&&&&&&&&&&&&&&&&&&&&&&&&&&&&\\
&&&&&&&&&&*+<1pt>{\bb{(\vw)}}\ar @{--}@[ashy][-6,30]\ar @{--}[-8,18]&&&&&&&&&&&&&&&&*+<1pt>{\bb{(\vx_2-\vx_1)}}\ar @{-}@[ashy][-6,30]\ar @{-}[-8,2]&&&&&&&&&&&&&&&&&&&&&&&&&&&&&&\\
&&&&&&&&&&&&&&&&&&&&&&&&&&&&&&&&&&&&&&&&&&&&&&&&&&&&&&&&\\
*+<1pt>{}\ar @{-}[-4,9]\ar @{-}@[ashy][-2,10]\ar @{-}@[ashy][-6,8]  &&&&&&&&&&&&&&&&*+<1pt>{}\ar @{-}@[ashy][-2,10]\ar @{-}[-8,2]&&&&&&&&&&&&&&&&&&&&&&&&&&&&&&&&&&&&&&&&\\}$}
\caption{$(2,3,3)$: Projective cover of $F_2$ and $E_3$}
\end{figure}

\subsection*{Weight type $(2,3,4)$}
The next figure yields the projective covers for the indecomposable vector bundles $E_2$, $E_3$, $E_4$ and $G_2$. For the projective cover of $E_4$ one has to combine the projective covers of $G_2$ and $\tau G_2$. By means of the reflection at the central horizontal axis $X\mapsto X(\vx_1-2\vx_3)$ one obtains the projective covers for the `missing' vector bundles $E_2(\vx_1-2\vx_3)$ and $E_3(\vx_1-2\vx_3)$.

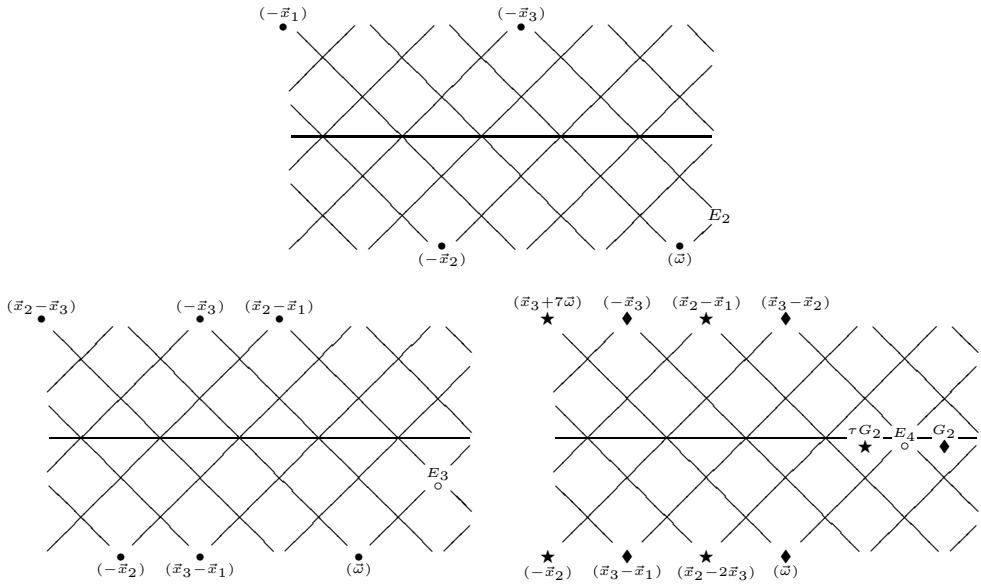
\begin{figure}[H]
{\tiny
$$\xymatrix @=15pt @!0 {
*+<1pt>{\tb{(-\vx_1)}}\ar@{-}[ddddddrrrrrr]&&\ar@{-}[ddddddrrrrrr]&&\ar@{-}[ddddddrrrrrr]&&*+<1pt>{\tb{(-\vx_3)}}\ar@{-}[dddddrrrrr]&&\ar@{-}[dddrrr]&&\ar@{-}[dr]&\\
&&&&&&&&&&&\\
\ar@{-}[uurr]\ar@{-}[ddddrrrr]&&&&&&&&&&&\\
\ar@{-}[rrrrrrrrrrr]&&&&&&&&&&&\\
\ar@{-}[uuuurrrr]\ar@{-}[ddrr]&&&&&&&&&&&\\
&&&&&&&&&&&*+<1pt>{E_2}\\
\ar@{-}[uuuuuurrrrrr]&&\ar@{-}[uuuuuurrrrrr]&&*+<1pt>{\bb{(-\vx_2)}}\ar@{-}[uuuuuurrrrrr]&&\ar@{-}[uuuuurrrrr]&&\ar@{-}[uuurrr]&&*+<1pt>{\bb{(\vom)}}\ar@{-}[ur]&\\
}
$$}
{\tiny
$$
\newcommand{\tpc}[1]{\overset{#1}{\blacklozenge}}
\newcommand{\tpcd}[1]{\underset{#1}{\blacklozenge}}
\newcommand{\spc}[1]{\overset{#1}{\bigstar}}
\newcommand{\spcd}[1]{\underset{#1}{\bigstar}}
\xymatrix @=15pt @!0 {
 \tb{(\vx_2-\vx_3)}\ar@{-}[ddddddrrrrrr]&&\ar@{-}[ddddddrrrrrr]&& \tb{(-\vx_3)}\ar@{-}[ddddddrrrrrr]&& \tb{(\vx_2-\vx_1)}\ar@{-}[ddddrrrr]&&\ar@{-}[dddrrr]&&\ar@{-}[dr]&\\
&&&&&&&&&&&\\
\ar@{-}[uurr]\ar@{-}[ddddrrrr]&&&&&&&&&&&\\
\ar@{-}[rrrrrrrrrrr]&&&&&&&&&&&\\
\ar@{-}[uuuurrrr]\ar@{-}[ddrr]&&&&&&&&&&*+<1pt>{\tc{E_3}}\ar@{-}[ur]\ar@{-}[dr]&\\
&&&&&&&&&&&\\
\ar@{-}[uuuuuurrrrrr]&& \bb{(-\vx_2)}\ar@{-}[uuuuuurrrrrr]&& \bb{(\vx_3-\vx_1)}\ar@{-}[uuuuuurrrrrr]&&\ar@{-}[uuuuurrrrr]&&{ \bb{(\vw)}}\ar@{-}[uurr]&&\ar@{-}[ur]&\\
}
\quad
\xymatrix @=15pt @!0 {
\spc{(\vx_3+7\vw)}\ar@{-}[ddddddrrrrrr]&&\tpc{(-\vx_3)}\ar@{-}[ddddddrrrrrr]&&\ar@{-}[ddddddrrrrrr]\spc{(\vx_2-\vx_1)}&&\tpc{\quad(\vx_3-\vx_2)}\ar@{-}[dddrrr]&&\ar@{-}[dddrrr]&&\ar@{-}[dr]&\\
&&&&&&&&&&&\\
\ar@{-}[uurr]\ar@{-}[ddddrrrr]&&&&&&&&&&&\\
\ar@{-}[rrrrrrrr]&&&&&&&&*+<1pt>{\spc{\tau G_2}}\ar@{-}[r]& *+<1pt>{\tc{E_4}} \ar@{-}[ddrr] \ar@{-}[uurr]\ar@{-}[r]&*+<1pt>{\tpc{G_2}}\ar@{-}[r]&\\
\ar@{-}[uuuurrrr]\ar@{-}[ddrr]&&&&&&&&&&&\\
&&&&&&&&&&&\\
\spcd{(-\vx_2)}\ar@{-}[uuuuuurrrrrr]&&\tpcd{(\vx_3-\vx_1)}\ar@{-}[uuuuuurrrrrr]&&\ar@{-}[uuuuuurrrrrr]\spcd{\quad(\vx_2-2\vx_3)}&&\tpcd{(\vw)}\ar@{-}[uuurrr]&&\ar@{-}[uuurrr]&&\ar@{-}[ur]&\\
}$$}
\caption{$(2,3,4)$: Projective covers for $E_2$, $E_3$, $E_4$, $G_2$ and $\tau G_2$}
\end{figure}
By {\tiny $\bigstar$} (resp. {\tiny $\blacklozenge$}) we have marked the line bundle summands of the projective covers of $\tau G_2$ (resp. $G_2$); together they form the projective cover of $E_4$.

\subsection*{Weight type $(2,3,5)$}
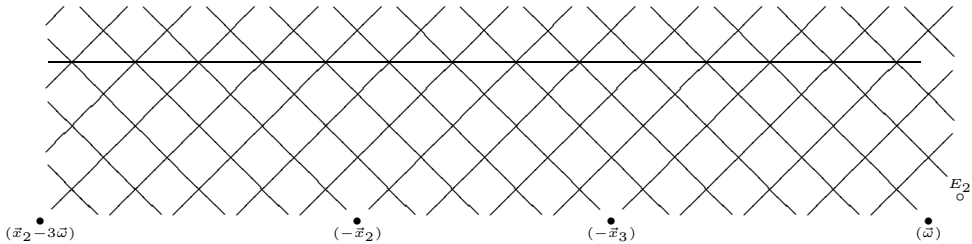
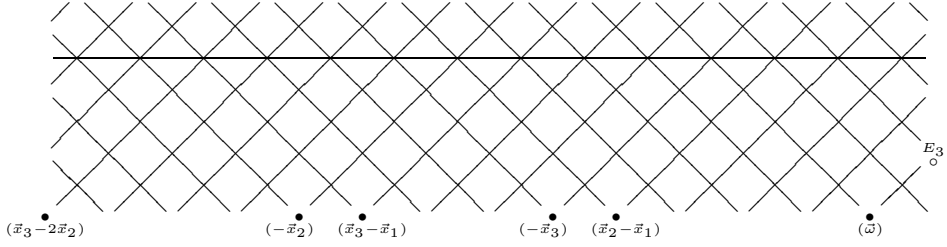
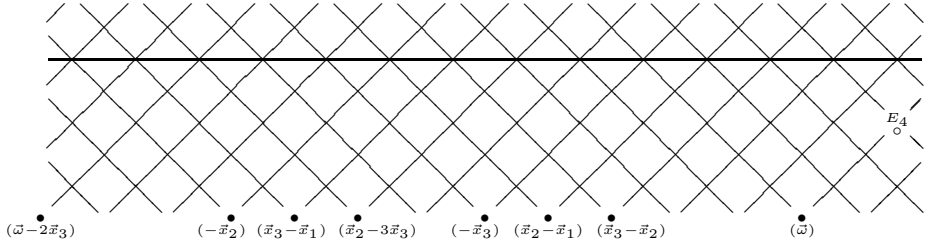
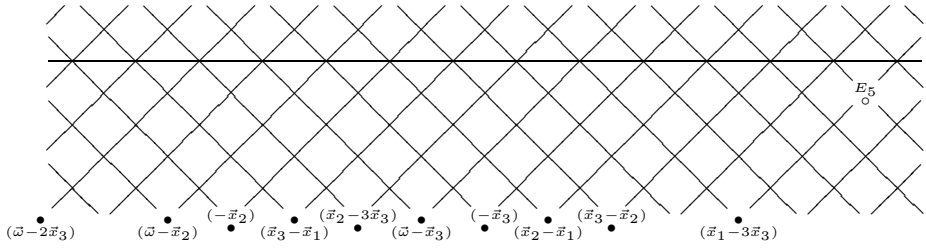
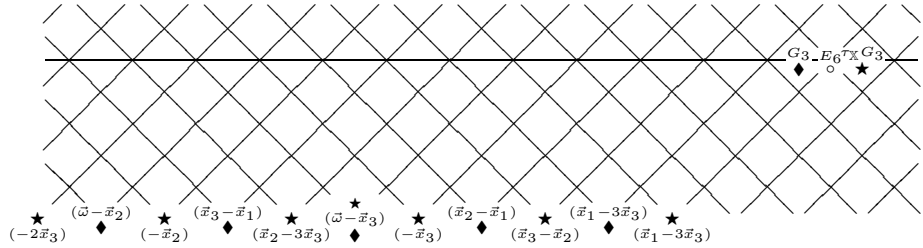
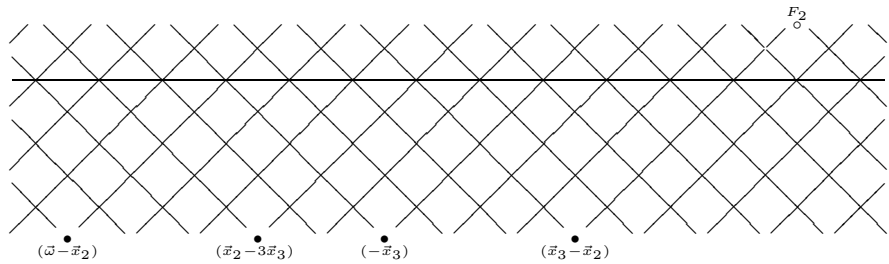
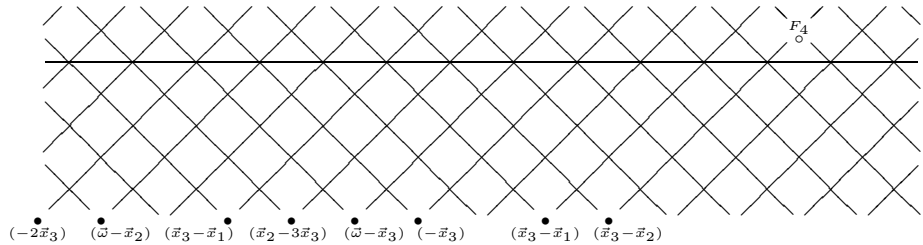
\newcommand{\tpc}[1]{\overset{#1}{\blacklozenge}}
\newcommand{\tpcd}[1]{\underset{#1}{\blacklozenge}}
\newcommand{\spc}[1]{\overset{#1}{\bigstar}}
\newcommand{\spcd}[1]{\underset{#1}{\bigstar}}

\begin{center}
{\tiny
$\xymatrix @=12pt @!0 {
&\ar@{-}[dddddddrrrrrrr]&&\ar@{-}[dddddddrrrrrrr]&&\ar@{-}[dddddddrrrrrrr]&&\ar@{-}[dddddddrrrrrrr]&&\ar@{-}[dddddddrrrrrrr]&&\ar@{-}[dddddddrrrrrrr]&&\ar@{-}[dddddddrrrrrrr]&&\ar@{-}[dddddddrrrrrrr]&&\ar@{-}[dddddddrrrrrrr]&&\ar@{-}[dddddddrrrrrrr]&&\ar@{-}[7,7]&&\ar@{-}[6,6]&&\ar@{-}[4,4]&&\ar@{-}[2,2]&&\\
\ar@{-}[ur]\ar@{-}[ddddddrrrrrr]&&&&&&&&&&&&&&&&&&&&&&&&&&&&&\\
\ar@{-}[rrrrrrrrrrrrrrrrrrrrrrrrrrrr]&&&&&&&&&&&&&&&&&&&&&&&&&&&&&\\
\ar@{-}[uuurrr]\ar@{-}[ddddrrrr]&&&&&&&&&&&&&&&&&&&&&&&&&&&&&\\
&&&&&&&&&&&&&&&&&&&&&&&&&&&&&\\
\ar@{-}[uuuuurrrrr]\ar@{-}[ddrr]&&&&&&&&&&&&&&&&&&&&&&&&&&&&&\\
&&&&&&&&&&&&&&&&&&&&&&&&&&&&&*+<1pt>{\tc{E_2}}\\
\bb{(\vx_2-3\vw)}\ar@{-}[uuuuuuurrrrrrr]&&\ar@{-}[uuuuuuurrrrrrr]&&\ar@{-}[uuuuuuurrrrrrr]&&\ar@{-}[uuuuuuurrrrrrr]&&\ar@{-}[uuuuuuurrrrrrr]&&\bb{(-\vx_2)}\ar@{-}[uuuuuuurrrrrrr]&&\ar@{-}[uuuuuuurrrrrrr]&&\ar@{-}[uuuuuuurrrrrrr]&&\ar@{-}[uuuuuuurrrrrrr]&&\bb{(-\vx_3)}\ar@{-}[uuuuuuurrrrrrr]&&\ar@{-}[uuuuuuurrrrrrr]&&\ar@{-}[-7,7]&&\ar@{-}[-5,5]&&\ar@{-}[-3,3]&&\bb{(\vw)}\ar@{-}[-1,1]&\\
}$}
\end{center}
\begin{center}
{\tiny
$\xymatrix @=12pt @!0 {
&\ar@{-}[dddddddrrrrrrr]&&\ar@{-}[dddddddrrrrrrr]&&\ar@{-}[dddddddrrrrrrr]&&\ar@{-}[dddddddrrrrrrr]&&\ar@{-}[dddddddrrrrrrr]&&\ar@{-}[dddddddrrrrrrr]&&\ar@{-}[dddddddrrrrrrr]&&\ar@{-}[dddddddrrrrrrr]&&\ar@{-}[dddddddrrrrrrr]&&\ar@{-}[dddddddrrrrrrr]&&\ar@{-}[dddddddrrrrrrr]&&\ar@{-}[dddddrrrrr]&&\ar@{-}[dddrrr]&&\ar@{-}[dr]&\\
\ar@{-}[ur]\ar@{-}[ddddddrrrrrr]&&&&&&&&&&&&&&&&&&&&&&&&&&&&\\
\ar@{-}[rrrrrrrrrrrrrrrrrrrrrrrrrrrr]&&&&&&&&&&&&&&&&&&&&&&&&&&&&\\
\ar@{-}[uuurrr]\ar@{-}[ddddrrrr]&&&&&&&&&&&&&&&&&&&&&&&&&&&&\\
&&&&&&&&&&&&&&&&&&&&&&&&&&&&\\
\ar@{-}[uuuuurrrrr]\ar@{-}[ddrr]&&&&&&&&&&&&&&&&&&&&&&&&&&&&*+<1pt>{\tc{E_3}}\\
&&&&&&&&&&&&&&&&&&&&&&&&&&&&\\
\bb{(\vx_3-2\vx_2)}\ar@{-}[uuuuuuurrrrrrr]&&\ar@{-}[uuuuuuurrrrrrr]&&\ar@{-}[uuuuuuurrrrrrr]&&\ar@{-}[uuuuuuurrrrrrr]&&\bb{(-\vx_2)\quad}\ar@{-}[uuuuuuurrrrrrr]&&\bb{\quad(\vx_3-\vx_1)}\ar@{-}[uuuuuuurrrrrrr]&&\ar@{-}[uuuuuuurrrrrrr]&&\ar@{-}[uuuuuuurrrrrrr]&&\bb{(-\vx_3)\quad}\ar@{-}[uuuuuuurrrrrrr]&&\bb{\quad(\vx_2-\vx_1)}\ar@{-}[uuuuuuurrrrrrr]&&\ar@{-}[uuuuuuurrrrrrr]&&\ar@{-}[uuuuuurrrrrr]&&\ar@{-}[uuuurrrr]&&\bb{(\vw)}\ar@{-}[uurr]&&\\
}$
\newline
$\xymatrix @=12pt @!0 {
&\ar@{-}[dddddddrrrrrrr]&&\ar@{-}[dddddddrrrrrrr]&&\ar@{-}[dddddddrrrrrrr]&&\ar@{-}[dddddddrrrrrrr]&&\ar@{-}[dddddddrrrrrrr]&&\ar@{-}[dddddddrrrrrrr]&&\ar@{-}[dddddddrrrrrrr]&&\ar@{-}[dddddddrrrrrrr]&&\ar@{-}[dddddddrrrrrrr]&&\ar@{-}[dddddddrrrrrrr]&&\ar@{-}[dddddddrrrrrrr]&&\ar@{-}[ddddrrrr]&&\ar@{-}[dddrrr]&&\ar@{-}[dr]&\\
\ar@{-}[ur]\ar@{-}[ddddddrrrrrr]&&&&&&&&&&&&&&&&&&&&&&&&&&&&\\
\ar@{-}[rrrrrrrrrrrrrrrrrrrrrrrrrrrr]&&&&&&&&&&&&&&&&&&&&&&&&&&&&\\
\ar@{-}[uuurrr]\ar@{-}[ddddrrrr]&&&&&&&&&&&&&&&&&&&&&&&&&&&&\\
&&&&&&&&&&&&&&&&&&&&&&&&&&&*+<1pt>{\tc{E_4}}\ar@{-}[ru]\ar@{-}[rd]&\\
\ar@{-}[uuuuurrrrr]\ar@{-}[ddrr]&&&&&&&&&&&&&&&&&&&&&&&&&&&&\\
&&&&&&&&&&&&&&&&&&&&&&&&&&&&\\
\bb{  (\vw-2\vx_3)}\ar@{-}[uuuuuuurrrrrrr]&&\ar@{-}[uuuuuuurrrrrrr]&&\ar@{-}[uuuuuuurrrrrrr]&& \bb{  (-\vx_2)\quad}\ar@{-}[uuuuuuurrrrrrr]&& \bb{  (\vx_3-\vx_1)\ }\ar@{-}[uuuuuuurrrrrrr] && \bb{\qquad  (\vx_2-3\vx_3)}\ar@{-}[uuuuuuurrrrrrr]&&\ar@{-}[uuuuuuurrrrrrr]&& \bb{  (-\vx_3)\quad}\ar@{-}[uuuuuuurrrrrrr]&& \bb{  (\vx_2-\vx_1)}\ar@{-}[uuuuuuurrrrrrr]&& \bb{\qquad  (\vx_3-\vx_2)}\ar@{-}[uuuuuuurrrrrrr]&&\ar@{-}[uuuuuuurrrrrrr]&&\ar@{-}[uuuuuurrrrrr]&& \bb{  (\vw)}\ar@{-}[uuurrr]&&\ar@{-}[uurr]&&\\}$
$$\xymatrix @=12pt @!0 {
&\ar@{-}[dddddddrrrrrrr]&&\ar@{-}[dddddddrrrrrrr]&&\ar@{-}[dddddddrrrrrrr]&&\ar@{-}[dddddddrrrrrrr]&&\ar@{-}[dddddddrrrrrrr]&&\ar@{-}[dddddddrrrrrrr]&&\ar@{-}[dddddddrrrrrrr]&&\ar@{-}[dddddddrrrrrrr]&&\ar@{-}[dddddddrrrrrrr]&&\ar@{-}[dddddddrrrrrrr]&&\ar@{-}[dddddddrrrrrrr]&&\ar@{-}[dddrrr]&&\ar@{-}[dddrrr]&&\ar@{-}[dr]&\\
\ar@{-}[ur]\ar@{-}[ddddddrrrrrr]&&&&&&&&&&&&&&&&&&&&&&&&&&&&\\
\ar@{-}[rrrrrrrrrrrrrrrrrrrrrrrrrrrr]&&&&&&&&&&&&&&&&&&&&&&&&&&&&\\
\ar@{-}[uuurrr]\ar@{-}[ddddrrrr]&&&&&&&&&&&&&&&&&&&&&&&&&&*+<1pt>{\tc{E_5}}\ar@{-}[ddrr]\ar@{-}[uurr]&&\\
&&&&&&&&&&&&&&&&&&&&&&&&&&&&\\
\ar@{-}[uuuuurrrrr]\ar@{-}[ddrr]&&&&&&&&&&&&&&&&&&&&&&&&&&&&\\
&&&&&&&&&&&&&&&&&&&&&&&&&&&&\\
 \bb{ (\vw-2\vx_3)}\ar@{-}[uuuuuuurrrrrrr]&&\ar@{-}[uuuuuuurrrrrrr]&& \bb{ (\vw-\vx_2)}\ar@{-}[uuuuuuurrrrrrr]&&*+<1pt>{ \tb{ (-\vx_2)}}\ar@{-}[uuuuuuurrrrrrr]&& \bb{ (\vx_3-\vx_1) }\ar@{-}[uuuuuuurrrrrrr] &&*+<1pt>{ \tb{ (\vx_2-3\vx_3)}}\ar@{-}[uuuuuuurrrrrrr]&& \bb{ (\vw-\vx_3)}\ar@{-}[uuuuuuurrrrrrr]&&*+<1pt>{ \tb{\quad (-\vx_3)}}\ar@{-}[uuuuuuurrrrrrr]&& \bb{ (\vx_2-\vx_1)}\ar@{-}[uuuuuuurrrrrrr]&&*+<1pt>{ \tb{ (\vx_3-\vx_2)}}\ar@{-}[uuuuuuurrrrrrr]&&\ar@{-}[uuuuuuurrrrrrr]&& \bb{ (\vx_1-3\vx_3)}\ar@{-}[uuuurrrr]&&\ar@{-}[uuuurrrr]&&\ar@{-}[uurr]&&\\
}$$

$$\xymatrix @=12pt @!0 {
&\ar@{-}[dddddddrrrrrrr]&&\ar@{-}[dddddddrrrrrrr]&&\ar@{-}[dddddddrrrrrrr]&&\ar@{-}[dddddddrrrrrrr]&&\ar@{-}[dddddddrrrrrrr]&&\ar@{-}[dddddddrrrrrrr]&&\ar@{-}[dddddddrrrrrrr]&&\ar@{-}[dddddddrrrrrrr]&&\ar@{-}[dddddddrrrrrrr]&&\ar@{-}[dddddddrrrrrrr]&&\ar@{-}[dddddddrrrrrrr]&&\ar@{-}[ddrr]&&\ar@{-}[dddrrr]&&\ar@{-}[dr]&\\
\ar@{-}[ur]\ar@{-}[ddddddrrrrrr]&&&&&&&&&&&&&&&&&&&&&&&&&&&&\\
\ar@{-}[rrrrrrrrrrrrrrrrrrrrrrrr]&&&&&&&&&&&&&&&&&&&&&&&&*+<1pt>{\tpc{G_3}}\ar@{-}[r]&*+<1pt>{\tc{E_6}}\ar@{-}[dddrrr]\ar@{-}[uurr]\ar@{-}[r]&*+<1pt>{\spc{\tauX G_3}}\ar@{-}[rr]&&\\
\ar@{-}[uuurrr]\ar@{-}[ddddrrrr]&&&&&&&&&&&&&&&&&&&&&&&&&&&&\\
&&&&&&&&&&&&&&&&&&&&&&&&&&&&\\
\ar@{-}[uuuuurrrrr]\ar@{-}[ddrr]&&&&&&&&&&&&&&&&&&&&&&&&&&&&\\
&&&&&&&&&&&&&&&&&&&&&&&&&&&&\\
\spcd{ (-2\vx_3)}\ar@{-}[uuuuuuurrrrrrr]&&*+<1pt>{\tpc{ (\vw-\vx_2)}}\ar@{-}[uuuuuuurrrrrrr]&&\spcd{ (-\vx_2)}\ar@{-}[uuuuuuurrrrrrr]&&*+<1pt>{\tpc{ (\vx_3-\vx_1)}}\ar@{-}[uuuuuuurrrrrrr]&&\spcd{ (\vx_2-3\vx_3)}\ar@{-}[uuuuuuurrrrrrr] &&*+<1pt>{\tpc{\spcd{ (\vw-\vx_3)}}}\ar@{-}[uuuuuuurrrrrrr]&&\spcd{ (-\vx_3)}\ar@{-}[uuuuuuurrrrrrr]&&*+<1pt>{\tpc{ (\vx_2-\vx_1)}}\ar@{-}[uuuuuuurrrrrrr]&&\spcd{ (\vx_3-\vx_2)}\ar@{-}[uuuuuuurrrrrrr]&&*+<1pt>{\tpc{ (\vx_1-3\vx_3)}}\ar@{-}[uuuuuuurrrrrrr]&&\spcd{ (\vx_1-3\vx_3)}\ar@{-}[uuuuurrrrr]&&\ar@{-}[uuuuuurrrrrr]&&\ar@{-}[uuuurrrr]&&\ar@{-}[uurr]&&\\
}$$

$$\xymatrix @=12pt @!0 {
&\ar@{-}[dddddddrrrrrrr]&&\ar@{-}[dddddddrrrrrrr]&&\ar@{-}[dddddddrrrrrrr]&&\ar@{-}[dddddddrrrrrrr]&&\ar@{-}[dddddddrrrrrrr]&&\ar@{-}[dddddddrrrrrrr]&&\ar@{-}[dddddddrrrrrrr]&&\ar@{-}[dddddddrrrrrrr]&&\ar@{-}[dddddddrrrrrrr]&&\ar@{-}[dddddddrrrrrrr]&&\ar@{-}[dddddddrrrrrrr]&&\ar@{-}[dr]&&*+<1pt>{\tc{F_2}}\ar@{-}[dddrrr]&&\ar@{-}[dr]&\\
\ar@{-}[ur]\ar@{-}[ddddddrrrrrr]&&&&&&&&&&&&&&&&&&&&&&&&*+<1pt>{}\ar@{-}[ddddrrrr]\ar@{-}[ur]&&&&\\
\ar@{-}[rrrrrrrrrrrrrrrrrrrrrrrrrrrr]&&&&&&&&&&&&&&&&&&&&&&&&&&&&\\
\ar@{-}[uuurrr]\ar@{-}[ddddrrrr]&&&&&&&&&&&&&&&&&&&&&&&&&&&&\\
&&&&&&&&&&&&&&&&&&&&&&&&&&&&\\
\ar@{-}[uuuuurrrrr]\ar@{-}[ddrr]&&&&&&&&&&&&&&&&&&&&&&&&&&&&\\
&&&&&&&&&&&&&&&&&&&&&&&&&&&&\\
\ar@{-}[uuuuuuurrrrrrr]&& \bb{ (\vw-\vx_2)}\ar@{-}[uuuuuuurrrrrrr]&&\ar@{-}[uuuuuuurrrrrrr]&&\ar@{-}[uuuuuuurrrrrrr]&& \bb{ (\vx_2-3\vx_3)\ }\ar@{-}[uuuuuuurrrrrrr] &&\ar@{-}[uuuuuuurrrrrrr]&& \bb{(-\vx_3)} \ar@{-}[uuuuuuurrrrrrr]&&\ar@{-}[uuuuuuurrrrrrr]&&\ar@{-}[uuuuuuurrrrrrr]&& \bb{(\vx_3-\vx_2)}\ar@{-}[uuuuuurrrrrr]&&\ar@{-}[uuuuuuurrrrrrr]&&\ar@{-}[uuuuuurrrrrr]&&\ar@{-}[uuuurrrr]&&\ar@{-}[uurr]&&\\
}$$

$$\xymatrix @=12pt @!0 {
&\ar@{-}[dddddddrrrrrrr]&&\ar@{-}[dddddddrrrrrrr]&&\ar@{-}[dddddddrrrrrrr]&&\ar@{-}[dddddddrrrrrrr]&&\ar@{-}[dddddddrrrrrrr]&&\ar@{-}[dddddddrrrrrrr]&&\ar@{-}[dddddddrrrrrrr]&&\ar@{-}[dddddddrrrrrrr]&&\ar@{-}[dddddddrrrrrrr]&&\ar@{-}[dddddddrrrrrrr]&&\ar@{-}[dddddddrrrrrrr]&&\ar@{-}[dr]&&\ar@{-}[dddrrr]&&\ar@{-}[dr]&\\
\ar@{-}[ur]\ar@{-}[ddddddrrrrrr]&&&&&&&&&&&&&&&&&&&&&&&&*+<1pt>{\tc{F_4}}\ar@{-}[ddddrrrr]\ar@{-}[ur]&&&&\\
\ar@{-}[rrrrrrrrrrrrrrrrrrrrrrrrrrrr]&&&&&&&&&&&&&&&&&&&&&&&&&&&&\\
\ar@{-}[uuurrr]\ar@{-}[ddddrrrr]&&&&&&&&&&&&&&&&&&&&&&&&&&&&\\
&&&&&&&&&&&&&&&&&&&&&&&&&&&&\\
\ar@{-}[uuuuurrrrr]\ar@{-}[ddrr]&&&&&&&&&&&&&&&&&&&&&&&&&&&&\\
&&&&&&&&&&&&&&&&&&&&&&&&&&&&\\
 \bb{ (-2\vx_3)}\ar@{-}[uuuuuuurrrrrrr]&& \bb{\qquad (\vw-\vx_2)}\ar@{-}[uuuuuuurrrrrrr]&&\ar@{-}[uuuuuuurrrrrrr]&& \bb{ (\vx_3-\vx_1)\qquad\quad}\ar@{-}[uuuuuuurrrrrrr]&& \bb{ (\vx_2-3\vx_3)\ }\ar@{-}[uuuuuuurrrrrrr] && \bb{\qquad (\vw-\vx_3)}\ar@{-}[uuuuuuurrrrrrr]&& \bb{\ \qquad  (-\vx_3)} \ar@{-}[uuuuuuurrrrrrr]&&\ar@{-}[uuuuuuurrrrrrr]&& \bb{ (\vx_3-\vx_1)}\ar@{-}[uuuuuuurrrrrrr]&& \bb{\qquad (\vx_3-\vx_2)}\ar@{-}[uuuuuurrrrrr]&&\ar@{-}[uuuuuuurrrrrrr]&&\ar@{-}[uuuuuurrrrrr]&&\ar@{-}[uuuurrrr]&&\ar@{-}[uurr]&&\\
}$$}
\end{center}

\begin{figure}[H]
\tiny{
$$\xymatrix @=12pt @!0 {
&\ar@{-}[dddddddrrrrrrr]&&\ar@{-}[dddddddrrrrrrr]&&\ar@{-}[dddddddrrrrrrr]&&\ar@{-}[dddddddrrrrrrr]&&\ar@{-}[dddddddrrrrrrr]&&\ar@{-}[dddddddrrrrrrr]&&\ar@{-}[dddddddrrrrrrr]&&\ar@{-}[dddddddrrrrrrr]&&\ar@{-}[dddddddrrrrrrr]&&\ar@{-}[dddddddrrrrrrr]&&\ar@{-}[dddddddrrrrrrr]&&\ar@{-}[dddddrrrrr]&&\ar@{-}[dddrrr]&&\ar@{-}[dr]&\\
\ar@{-}[ur]\ar@{-}[ddddddrrrrrr]&&&&&&&&&&&&&&&&&&&&&&&&&&&&\\
\ar@{-}[rrrrrrrrrrrrrrrrrrrrrrrr]&&&&&&&&&&&&&&&&&&&&&&&&*+<1pt>{\tc{G_3}}\ar@{-}[rrrr]&&&&\\
\ar@{-}[uuurrr]\ar@{-}[ddddrrrr]&&&&&&&&&&&&&&&&&&&&&&&&&&&&\\
&&&&&&&&&&&&&&&&&&&&&&&&&&&&\\
\ar@{-}[uuuuurrrrr]\ar@{-}[ddrr]&&&&&&&&&&&&&&&&&&&&&&&&&&&&\\
&&&&&&&&&&&&&&&&&&&&&&&&&&&&\\
\bb{ (-2\vx_3)}\ar@{-}[uuuuuuurrrrrrr]&&\ar@{-}[uuuuuuurrrrrrr]&&\bb{ (-\vx_2)}\ar@{-}[uuuuuuurrrrrrr]&&\ar@{-}[uuuuuuurrrrrrr]&&\bb{ (\vx_2-3\vx_3)\quad}\ar@{-}[uuuuuuurrrrrrr]&&\bb{\quad (\vw-\vx_3)}\ar@{-}[uuuuuuurrrrrrr]&&\ar@{-}[uuuuuuurrrrrrr]&&\bb{ (\vx_2-\vx_1)}\ar@{-}[uuuuuuurrrrrrr]&&\ar@{-}[uuuuuuurrrrrrr]&&\bb{ (\vx_1-3\vx_3)}\ar@{-}[uuuuuuurrrrrrr]&&\ar@{-}[uuuuuuurrrrrrr]&&\ar@{-}[uuuuuurrrrrr]&&\ar@{-}[uuuurrrr]&&\ar@{-}[uurr]&&\\
}$$}
\caption{$(2,3,5)$: Projective covers for $E_2$, $E_3$, $E_4$, $E_5$, $E_6$, $F_2$, $F_4$ and  $G_3$}
\end{figure}
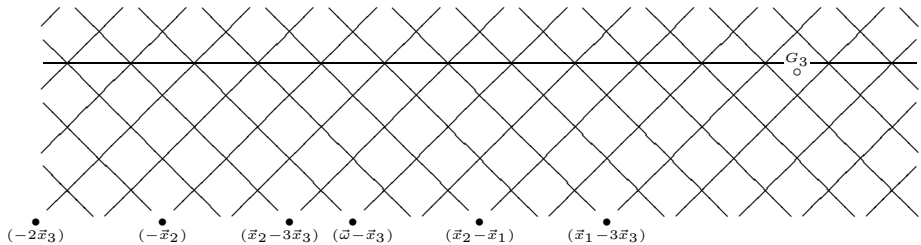

\section*{Acknowledgements}
The first author was supported by the Research Project from the MNiSW specific subsidy for young scientists on the  Department of Mathematics and Physics of Szczecin University 504--4000--240--835.

The third author was supported by the Polish Scientific Grant Narodowe Centrum
Nauki DEC-2011/01/B/ST1/06469.

The authors further want to thank the referee for helpful suggestions and a critical checking of the manuscript.

\normalsize

\def\cprime{$'$} \def\cprime{$'$} \def\cprime{$'$}
\providecommand{\bysame}{\leavevmode\hbox to3em{\hrulefill}\thinspace}
\providecommand{\MR}{\relax\ifhmode\unskip\space\fi MR }
\providecommand{\MRhref}[2]{%
  \href{http://www.ams.org/mathscinet-getitem?mr=#1}{#2}
}
\providecommand{\href}[2]{#2}

\end{document}